%% file: main.tex
\documentclass[11pt]{article}
\usepackage[margin=1in]{geometry}
\usepackage{algorithm}
\usepackage{algpseudocode}
\usepackage{url}

\providecommand{\headers}[2]{}

\let\oldtitle\title
\renewcommand{\title}[2][]{\oldtitle{#2}}

\input{splitted/preamble}

\input{splitted/shared}

\providecommand{\keywordsname}{Keywords}
\newcommand{\Keywords}[1]{\par\noindent\textbf{\keywordsname:} #1}
\newcommand{\AMS}[1]{\par\noindent\textbf{AMS subject classifications:} #1}

\begin{document}
\author{Michał Wichrowski$ ^0$}
\footnotetext{
    % Interdisziplinäres Zentrum für Wissenschaftliches Rechnen (IWR), Ruprecht-Karls-Universität Heidelberg, Germany,
    \texttt{mwichro@mimuw.edu.pl}}
\date{}

\maketitle

\begin{abstract}
    \input{splitted/abstract}
\end{abstract}


\AMS{65N55, 65N30, 65N12, 65F10, 65N22, 65N85}

\input{splitted/main_text}

% Appendix A (cut-position statistics, full proof of prop:circle) removed
% from the build 2026-07-15 together with the archived outlier machinery
% (splitted/archive_outliers.tex); prop:circle itself now lives there.
% Spin-off-note candidate. To restore:
% \appendix
% \input{splitted/appendix_circle}

\bibliographystyle{siam}
\bibliography{literature, original, added_literature, my_papers}

\appendix
\input{splitted/appendix}
\end{document}

%% file: splitted/shared.tex
\usepackage{lipsum}
\usepackage{amsfonts}
\usepackage{graphicx}
\usepackage{epstopdf}
\usepackage{amsmath,amssymb,amsthm}
\newtheorem{lemma}{Lemma}
\newtheorem{theorem}{Theorem}
\newtheorem{corollary}{Corollary}
\theoremstyle{definition}
\newtheorem{assumption}{Assumption}
\newtheorem{proposition}{Proposition}
\theoremstyle{remark}
\newtheorem{remark}{Remark}
\theoremstyle{plain}
\ifpdf
    \DeclareGraphicsExtensions{.eps,.pdf,.png,.jpg}
\else
    \DeclareGraphicsExtensions{.eps}
\fi

\headers{Short title }{Michał Wichrowski}

\title{
    A Multigrid Method for CutFEM and its Convergence
}
\author{Michał Wichrowski$^0$}

\usepackage{amsopn}

%% file: splitted/abstract.tex
%!TEX root = ../main.tex
% Multigrid convergence theory and cut-uniform bounds.
We develop a convergence theory for geometric multigrid with vertex-patch smoothers applied to cut finite element
discretizations of the Poisson problem. The framework addresses non-inherited level forms and the mismatch between the
physical and active domains. Using the discrete extension property, we prove two-level convergence bounds uniform in
the mesh size and the cut geometry, and W-cycle bounds under an additional smallness assumption on the two-level rate.
The numerical experiments intentionally use the stronger V-cycle, for which no convergence bound is claimed here.

% Degree dependence and polynomial shedding of the ghost penalty.
The convergence constants degrade with the degree $p$. Lowering the ghost penalty improves iteration counts. An
aligned two-cell model exhibits a semidefiniteness threshold of order $p^{-2}$, whereas the visibility scale of a
degree-$p$ cut mode decreases exponentially. Experiments at the model threshold reduce the iteration counts, but do
not establish an assembled-operator threshold.

%% file: splitted/main_text.tex
%!TEX root = ../main.tex
\section{Introduction}
\label{sec:intro}
\input{splitted/introduction.tex}

%% Publishing scope: this paper proves mesh- and cut-independent convergence but retains poor degree dependence.
%% Present polynomial reduction of the symmetric ghost weight only as an admissible range, not as optimal.
%% Reserve stronger shedding and nonsymmetric formulations for the follow-up paper; distinguish weight shedding from polynomial-in-p bounds.

% Environments and local macros (siamltex provides theorem-like envs).
\newtheorem{conjecture}{Conjecture}
\def\energy#1{|\!|\!| #1 |\!|\!|}
\def\vG{\mathbf{G}}
\def\vu{\mathbf{u}}
\providecommand{\lesssim}{\mathrel{\preceq}} % amssymb not loaded in main.tex

%%%%%%%%%%%%%%%%%%%%%%%%%%%%%%%%%%%%%%%%%%%%%%%%%%%%%%%%%%%%
\section{Problem setting and assumptions}
\label{sec:theory}

% Locate the smoother definition and state the purpose of this section.
The analysis concerns the vertex-patch smoother defined in Section~\ref{sec:smoother}. This section fixes the geometric
setting and the CutFEM stability properties used in its convergence analysis.

% Notation and constants.
Throughout, $c$, $C$ denote constants independent of $\ell$, $h_\ell$, and of the position of $\Gamma$ relative to the
meshes, but possibly depending on the polynomial degree $p$, the shape of $\Gamma$ (curvature bounds), and the Nitsche
parameter $\gamma_D$. Dependence on $p$ and on the global ghost weight $\gamma$ of~\eqref{eq:intro-ghost} will be
tracked where it matters.

For a measurable set $X\subseteq\R^d$ or a piece of a $(d-1)$-manifold (a face $F$, a subset of $\Gamma$), we write
$(u,v)_X$ for the $L^2(X)$ inner product (with respect to the Lebesgue or surface measure, as appropriate),
$\|v\|_{0,X} := (v,v)_X^{1/2}$ for the induced $L^2$ norm, and $|v|_{1,X} := \|\nabla v\|_{0,X}$, $\|v\|_{1,X}^2 :=
    \|v\|_{0,X}^2 + |v|_{1,X}^2$ for the $H^1$-seminorm and norm on $X$; for a collection of faces the norm is understood
facewise, $\|v\|_{0,\mathbb F}^2 = \sum_{F\in\mathbb F}\|v\|_{0,F}^2$. On $\Gamma$, $n$ denotes the unit normal
pointing out of $\Omega$ and $\partial_n v = n\cdot\nabla v$; on an interior face $F$ it is a fixed unit normal of $F$,
$\partial_n^k v$ the $k$-th normal derivative and $\llbracket\cdot\rrbracket$ the jump across $F$. The meshes
$\mesh_\ell$ are Cartesian: every cell $\cell\in\mesh_\ell$ is an axis-parallel cube of edge $h_\ell$, with
$h_{\ell}=h_{\ell-1}/2$, and every face $F$ is axis-parallel, so the face normal $n$ is always a coordinate direction
and $\partial_n^k$ a pure coordinate derivative. Accordingly the local space is the \emph{tensor-product} space
\begin{gather}\label{eq:Qp}
    \Q_p(\cell) := \operatorname{span}\bigl\{\, x_1^{\alpha_1}\cdots x_d^{\alpha_d} \;:\;
    0\le\alpha_i\le p,\ i=1,\dots,d \,\bigr\},
\end{gather}
of dimension $(p+1)^d$: degree at most $p$ \emph{in each variable separately}, not total degree $p$. The tensor-product
structure permits one-dimensional polynomial inequalities to be applied in the normal direction, with the tangential
variables frozen and without enlarging the one-dimensional constant, since a directional
derivative of $q\in\Q_p$ again lies in $\Q_p$ and its restriction to any line parallel to a coordinate axis is a
univariate polynomial of degree at most $p$ (Lemma~\ref{lem:ghostchain}).
Functions in the finite element space $V_\ell$ (the $\Q_p$ space on the active cells $\mesh_{\ell,\Omega}$, cf.\
Section~\ref{sec:intro}) are cellwise polynomials on all of the active domain $\Omega_\ell\supseteq\Omega$, so all these
quantities are defined on $\Omega_\ell$ and not only on $\Omega$; whether a norm is taken over $\Omega$ or over
$\Omega_\ell$ is always indicated in the subscript and matters. The nodal basis of $V_\ell$ is the Lagrange basis on
the \emph{tensor-product Gauss--Lobatto} nodes of each cell. The one-signed Gauss--Lobatto quadrature error and the
resulting nodal $L^2$-stability enter the constants of
Proposition~\ref{prop:quasi} and Theorem~\ref{thm:additive}. Finally, we write $\energy{v}_\ell^2 = A_\ell(v,v)$ for the stabilized energy
norm.

\begin{assumption}[geometric setting]\label{ass:geometry}
    The boundary $\Gamma$ is $C^2$ and has positive reach $\delta_0>0$.
    On every level, if $\Gamma$ intersects a cell $\cell$, then
    $\Gamma\cap\cell$ is a single curve (surface) intersecting
    $\partial\cell$ exactly twice (in a single closed curve for $d=3$), each
    closed face is intersected at most once, and $\cell\setminus\Gamma$ has
    exactly two connected components. Moreover, $\Gamma$ contains no vertex
    of any $\mesh_\ell$, $\Omega_\ell\subseteq\Omega_{\ell-1}$ for all
    $\ell$, and
    \begin{gather}\label{eq:resolution}
        c_R\, h_{\ell-1} \;\le\; \delta_0
        \qquad\text{on every level used in the analysis,}
    \end{gather}
    where $c_R=c_R(d):=24d+10\sqrt d$.
\end{assumption}

% Explain the geometric regularity and cut configuration.
The $C^2$ regularity provides the local smoothness used in the geometric arguments. Positive reach means that the
closest-point projection onto $\Gamma$ is well defined and $C^1$ in the tube $U_{\delta_0}$. For a $C^2$ boundary, this
condition amounts to a curvature bound together with a lower bound on the separation of distinct sheets. The cut-cell
conditions constitute the Hansbo--Hansbo resolution condition. The exclusion of mesh vertices is a generic-position
assumption; if necessary, it can be enforced by an arbitrarily small perturbation of the level set.

% Explain the quantitative resolution condition and the treatment of coarse levels.
The value of $c_R$ provides a factor-two safety margin for the geometric constructions in Lemma~\ref{lem:geom}; its
value is not optimized. Levels that are too coarse to satisfy~\eqref{eq:resolution} are incorporated into the exact
coarsest-level solver. This restriction does not affect the two-grid analysis below, but a multilevel result must
account for such levels explicitly.

% State the CutFEM stability assumptions.
\begin{assumption}[CutFEM stability]\label{ass:cutfem}
    Define on $V_\ell$ the mesh-dependent norm
    \begin{gather}\label{eq:triple-norm}
        N_\ell(v)^2 := \|\nabla v\|^2_{0,\Omega}
        + h_\ell \|\partial_n v\|^2_{0,\Gamma}
        + h_\ell^{-1} \|v\|^2_{0,\Gamma}
        + g_\ell(v,v).
    \end{gather}
    With the ghost penalty $g_\ell$ and a sufficiently large, cut-independent
    Nitsche parameter $\gamma_D$, the form $A_\ell$ is coercive and bounded
    with respect to $N_\ell$,
    \begin{gather}\label{eq:coercive}
        c_A\, N_\ell(v)^2 \le A_\ell(v,v),
        \qquad
        A_\ell(u,v) \le C_A\, N_\ell(u)\, N_\ell(v),
    \end{gather}
    where $c_A$ and $C_A$ are independent of $\ell$ and the cut
    configuration. The ghost penalty also satisfies the discrete-extension
    estimate
    \begin{gather}\label{eq:ghost-extension}
        |v|^2_{1,\Omega_\ell\setminus\Omega}
        \le C \bigl( |v|^2_{1,\Omega} + g_\ell(v,v) \bigr),
        \qquad v\in V_\ell,
    \end{gather}
    with a constant independent of the cut configuration.
\end{assumption}

% Explain the consequences and provenance of the stability assumptions.
Coercivity and boundedness imply $\energy{v}_\ell\simeq N_\ell(v)$, while the discrete-extension estimate yields the
corresponding equivalence with $\|v\|_{1,\Omega_\ell}$-type norms, in both cases with cut-independent constants. These
properties are established for fixed polynomial degree in the ghost-penalty
literature~\cite{burman2010ghost,BurmanHansbo12,MassingLarsonLogg14,BurmanClausHansboLarsonMassing15}. The
estimate~\eqref{eq:ghost-extension} is the penalized counterpart of the stability property of the discrete extension
operator of~\cite{BurmanHansboLarson22Ext}, which realizes the same control by extending from the interior cells into a
modified space rather than by penalizing jumps. Here their dependence on $p$ must be explicit because it enters the
transfer and two-level estimates of Sections~\ref{sec:transfer} and~\ref{sec:twolevel}; in particular, the factor
$(p+1)\,\gamma^{-1}$ reappears in the $p$-scaling of Section~\ref{sec:p-dependence}. Proposition~\ref{prop:cutfem}
therefore tracks the polynomial trace/Markov factors, the Cauchy--Schwarz factor $(p+1)\,\gamma^{-1}$, and the single
exponential factor given by the chain amplification $\Xi_p$ in~\eqref{eq:Xi}; see Remark~\ref{rem:where-exponential}.

% Specify the ghost-penalty weights used in the analysis.
The proof also determines the ghost-weight scaling used throughout the analysis: $g_\ell$ carries the standard weights
$h_F^{2k-1}/(k!)^2$~\cite{burman2010ghost}, scaled by the single global weight $\gamma\in(0,1]$
from~\eqref{eq:intro-ghost}. All constants below indicate their dependence on $\gamma$, which is the design variable in
Section~\ref{sec:cut-adaptive}.\footnote{In~\cite{CuiKanschat25CutFEM}, the weights are printed as $h_F^{2k+1}/(k!)^2$.
    We interpret the exponent as a typo for the standard $h_F^{2k-1}/(k!)^2$, because the printed weights define a weaker,
    $L^2$-type penalty for which~\eqref{eq:ghost-extension} holds only with $h_\ell^{-2}g_\ell$. The same distinction
    appears in the design framework of~\cite{BurmanHansboLarson26Locking}, where the face-penalty weight $h_F^{3-2m}$
    carries $m=1$ for control of the $H^1$-seminorm and $m=0$ for the $L^2$ norm, the two exponents differing by $h_F^2$.}

%%%%%%%%%%%%%%%%%%%%%%%%%%%%%%%%%%%%%%%%%%%%%%%%%%%%%%%%%%%%
\section{Stability of the intergrid transfer}
\label{sec:transfer}
%%%%%%%%%%%%%%%%%%%%%%%%%%%%%%%%%%%%%%%%%%%%%%%%%%%%%%%%%%%%

The multigrid transfer is the identity on functions, $\prol{\ell}u_{\ell-1}(x)=u_{\ell-1}(x)$ on $\Omega_\ell$, which
is well defined by Assumption~\ref{ass:geometry}. The forms are \emph{non-inherited}: $A_\ell$ is not the Galerkin
restriction of $A_{\ell+1}$, because the Nitsche terms are the same but the ghost penalties act on different face sets
with different weights, and the cut cell sets differ per level. The first lemma proves energy-norm stability of the
transfer despite these differences.

\begin{lemma}[prolongation stability]\label{lem:prolongation}
    There is $C_{\mathrm{pr}}$ independent of $\ell$ and of the cut
    configuration such that
    \begin{gather}\label{eq:prolongation}
        \energy{\prol{\ell} v}_{\ell}
        \le C_{\mathrm{pr}}\, \energy{v}_{\ell-1}
        \qquad \forall v\in V_{\ell-1}.
    \end{gather}
\end{lemma}

\begin{proof}
    Throughout, write $H:=h_{\ell-1}$, so $h_\ell = H/2$, and
    $w:=\prol\ell v$. Since $\Omega_\ell\subseteq\Omega_{\ell-1}$
    (Assumption~\ref{ass:geometry}), $w$ is the restriction of $v$ to
    $\Omega_\ell$; in particular $w=v$ on $\Omega$ and on $\Gamma$.

    \emph{Step 1 (reduction to the norm $N_\ell$).}
    By~\eqref{eq:coercive} it suffices to prove
    \begin{gather}\label{eq:goal-N}
        N_\ell(w)^2 \le C\, N_{\ell-1}(v)^2,
    \end{gather}
    since then $\energy{w}_\ell^2 = A_\ell(w,w) \le C_A N_\ell(w)^2
        \le C\,C_A N_{\ell-1}(v)^2 \le C\,(C_A/c_A)\, \energy{v}_{\ell-1}^2$.

    \emph{Step 2 (unpenalized terms).}
    Because $w=v$ on $\Omega$ and $\Gamma$, the first three terms
    of~\eqref{eq:triple-norm} for $w$ on level $\ell$ differ from those for
    $v$ on level $\ell-1$ only through the mesh size in the weights:
    \begin{gather*}
        \|\nabla w\|^2_{0,\Omega} = \|\nabla v\|^2_{0,\Omega},
        \quad
        h_\ell \|\partial_n w\|^2_{0,\Gamma}
        = \tfrac12\, H \|\partial_n v\|^2_{0,\Gamma},
        \quad
        h_\ell^{-1}\|w\|^2_{0,\Gamma} = 2\, H^{-1}\|v\|^2_{0,\Gamma}.
    \end{gather*}
    Hence these three terms are bounded by $2$ times their coarse
    counterparts, and it remains to bound $g_\ell(w,w)$.

    \emph{Step 3 (classification of fine ghost faces).}
    Let $F\in\mathbb F_G^\ell$ with adjacent fine cells
    $\cell_1,\cell_2\in\mesh_{\ell,\Omega}$, and let $P(\cell_i)\in
        \mesh_{\ell-1}$ denote the coarse parent cells. Since $\cell_i\cap\Omega
        \ne\emptyset$ and $\cell_i\subset P(\cell_i)$, both parents belong to
    $\mesh_{\ell-1,\Omega}$; note that no property of the cut sets beyond
    this is used. By nestedness of the Cartesian hierarchy, $F$ lies on a
    hyperplane which either passes through the interior of a coarse cell or
    contains a coarse face; consequently exactly one of the following holds:
    \begin{itemize}
        \item[(i)] $P(\cell_1)=P(\cell_2)=:\cell'$ and $F\subset
                  \operatorname{int}\cell'$. Then $v|_{\cell'}\in\Q_p(\cell')$ is a single
              polynomial, all jumps $\llbracket\partial_n^k v\rrbracket_F$, $0\le k\le
                  p$, vanish, and $F$ contributes nothing to $g_\ell(w,w)$.
        \item[(ii)] $P(\cell_1)\ne P(\cell_2)$ and $F\subset F' :=
                  F\bigl(P(\cell_1),P(\cell_2)\bigr)$, a coarse face with both neighbors
              in $\mesh_{\ell-1,\Omega}$. Write $\omega_{F'} := P(\cell_1)\cup
                  P(\cell_2) \subset \Omega_{\ell-1}$.
    \end{itemize}

    \emph{Step 4 (ghost penalty by weight comparison).}
    Only faces of type~(ii) contribute. Fix such an $F\subset F'$. The jump
    $\llbracket\partial_n^k v\rrbracket$ across $F$ is the restriction to $F$
    of the jump across the coarse face $F'$ of the single coarse function $v$
    (both sides are the coarse polynomials $v|_{P(\cell_1)},v|_{P(\cell_2)}$),
    so the fine subfaces of a fixed $F'$ reassemble it:
    $$\sum_{F\subset F'}\|\llbracket\partial_n^k v\rrbracket\|^2_{0,F}
        = \|\llbracket\partial_n^k v\rrbracket\|^2_{0,\bigcup F}
        \le \|\llbracket\partial_n^k v\rrbracket\|^2_{0,F'}.$$
    With $h_\ell=H/2$
    the fine ghost weight satisfies $h_\ell^{2k-1}=2^{-(2k-1)}H^{2k-1}\le
        \tfrac12 H^{2k-1}$ for every $k\ge1$, and the term $k=0$ vanishes since
    $v\in H^1(\Omega_{\ell-1})$ is continuous. Hence, summing over $k$ and over
    the fine ghost faces on a fixed $F'$,
    \begin{gather*}
        \sum_{F\subset F'}\sum_{k=1}^p
        \frac{h_\ell^{2k-1}}{(k!)^2}
        \bigl\|\llbracket\partial_n^k v\rrbracket\bigr\|^2_{0,F}
        \;\le\; \tfrac12 \sum_{k=1}^p \frac{H^{2k-1}}{(k!)^2}
        \bigl\|\llbracket\partial_n^k v\rrbracket\bigr\|^2_{0,F'} .
    \end{gather*}
    Every such $F'$ belongs to the coarse-level ghost-penalty set: a fine face $F\in\mathbb F_G^\ell$ has a cut
    neighboring cell, whose coarse parent is then cut, so $F'\in\mathbb
        F_G^{\ell-1}$. Multiplying by the global weight $\gamma$, common to
    both levels, and summing over all faces in $\mathbb F_G^{\ell-1}$, with no trace,
    Markov or inverse inequality, gives
    \begin{gather}\label{eq:ghost-weight}
        g_\ell(w,w) \;\le\; \tfrac12\, g_{\ell-1}(v,v)
        \;\le\; \tfrac12\, N_{\ell-1}(v)^2,
    \end{gather}
    the coarse ghost penalty being a term of the coarse
    norm~\eqref{eq:triple-norm}. With Step 2 this proves~\eqref{eq:goal-N}
    with $C=2$, and the lemma follows with
    $C_{\mathrm{pr}}=(2C_A/c_A)^{1/2}$,
    independent of $\ell$ and of the cut configuration. The mesh-norm
    estimate~\eqref{eq:goal-N} is independent of the polynomial degree; any
    dependence of $C_{\mathrm{pr}}$ on $p$ enters solely through the
    norm-equivalence factor $C_A/c_A$ of~\eqref{eq:coercive}.
\end{proof}

\begin{remark}[$p$-uniformity of the ghost transfer]\label{rem:strengthened}
    % Distinguish the transfer estimate from the energy-norm constants.
    Inequality~\eqref{eq:ghost-weight} gives the sharper estimate that the fine-level ghost energy of a prolongated coarse
    function is controlled by the \emph{coarse} ghost penalty on the same faces, with the absolute factor $\tfrac12$ and no
    dependence on $p$. Hence the transfer contributes no $p$-growth in the mesh-dependent norms. Its energy-norm constant
    $C_{\mathrm{pr}}=(2C_A/c_A)^{1/2}$ nevertheless inherits the dependence of the norm equivalence~\eqref{eq:coercive};
    the remaining dependence in Theorem~\ref{thm:twolevel} enters through the stable-decomposition constant
    $C_{\mathrm{sd}}(p)$ of Theorem~\ref{thm:additive}. The reverse comparison for fine-grid functions need not hold
    because fine faces inside coarse cells need not be coarse ghost faces, but the Schwarz framework does not require it.
\end{remark}

The stable decomposition below instead controls the ghost energy of a \emph{single-level} finite element function by
its $H^1$-seminorm. Unlike the transfer, this comparison is $p$-dependent.

\begin{lemma}[ghost energy vs.\ the $H^1$-seminorm]\label{lem:ghostchain}
    Let $\cell$ be a mesh cube of edge $H$ and $F\subset\partial\cell$ a face
    with normal coordinate $n$. For every $q\in\Q_p(\cell)$ and every $1\le
        k\le p$,
    \begin{gather}\label{eq:one-sided}
        \|\partial_n^k q\|^2_{0,F}
        \le C_T (p+1)^2 H^{-1}\bigl(C_M^2 p^4 H^{-2}\bigr)^{k-1}
        |q|^2_{1,\cell},
    \end{gather}
    with $C_T$ the constant of the trace inequality~\eqref{eq:trace} and $C_M$
    that of the one-dimensional $L^2$ Markov inequality~\eqref{eq:markov}.
    Consequently:
    \begin{itemize}
        \item[(a)] for a ghost face $F$ between cells $\cell_1,\cell_2$ of edge
              $H$, union $\omega_F$, and any $u$ continuous across $F$ with
              $u|_{\cell_i}\in\Q_p(\cell_i)$, $i=1,2$,
              \begin{gather}\label{eq:one-face}
                  \sum_{k=0}^p \frac{H^{2k-1}}{(k!)^2}
                  \bigl\|\llbracket\partial_n^k u\rrbracket\bigr\|^2_{0,F}
                  \le \Theta'_p\, |u|^2_{1,\omega_F},
                  \qquad
                  \Theta'_p := 2\,C_T (p+1)^2 \sum_{k=1}^{p}
                  \frac{(C_M^2 p^4)^{k-1}}{(k!)^2};
              \end{gather}
        \item[(b)] for any level $m$ and any $u\in V_m$,
              $g_m(u,u) \le \gamma\, 2d\,\Theta'_p\, |u|^2_{1,\Omega_m}$,
              with $\gamma$ the global weight of~\eqref{eq:intro-ghost}.
    \end{itemize}
\end{lemma}

\begin{proof}
    Since the mesh is Cartesian, $n$ is a coordinate direction and
    $\partial_n^k$ a pure coordinate derivative. The polynomial trace
    inequality on the cube $\cell$,
    \begin{gather}\label{eq:trace}
        \|s\|^2_{0,\partial \cell} \le C_T\,(p+1)^2 H^{-1} \|s\|^2_{0,\cell},
        \qquad s\in\Q_p(\cell),
    \end{gather}
    (the explicit constants of Warburton and Hesthaven~\cite{WarburtonHesthaven03}
    are stated for total-degree $\mathbb P_p$ on simplices; on a cube the
    inequality follows instead, with $C_T=2d$, by tensorizing the endpoint
    Christoffel bound~\eqref{eq:christoffel} of Appendix~\ref{app:spd-floor},
    whose reproducing kernel has endpoint value $(p+1)^2H^{-1}$ for
    $\mathbb P_p$ on an interval of length $H$: freezing the tangential
    variables and integrating over them bounds $\|s\|^2_{0,F}$ by
    $(p+1)^2H^{-1}\|s\|^2_{0,\cell}$ on each of the $2d$ faces),
    applied to $s=\partial_n^k q$, together with the $L^2$ Markov inequality
    and its $(k-1)$-fold iteration in the coordinate $n$,
    \begin{gather}\label{eq:markov}
        \begin{aligned}
            \|\partial_n r\|_{0,\cell}
             & \le C_Mp^2H^{-1}\|r\|_{0,\cell},
             &                                    & r\in\Q_p(\cell),
            \\
            \|\partial_n^{k}q\|_{0,\cell}
             & \le \bigl(C_Mp^2H^{-1}\bigr)^{k-1}
            \|\partial_nq\|_{0,\cell},
             &                                    & 1\le k\le p .
        \end{aligned}
    \end{gather}
    ($C_M$ the absolute one-dimensional
    constant~\cite{Schwab98,OzisikRiviereWarburton10}, unchanged by
    tensorization: by the tensor-product structure~\eqref{eq:Qp} the restriction
    of $q$ to any line parallel to $n$ is a univariate polynomial of degree at
    most $p$, so the one-dimensional inequality applies for each frozen value of
    the tangential variables and integrating over them preserves the constant),
    and $\|\partial_n q\|_{0,\cell}\le|q|_{1,\cell}$, give
    \eqref{eq:one-sided}. For (a),
    $$\|\llbracket\partial_n^k
        u\rrbracket\|^2_{0,F}\le 2\|\partial_n^k u_1\|^2_{0,F}+2\|\partial_n^k
        u_2\|^2_{0,F}$$
    with $u_i:=u|_{\cell_i}$, the $k=0$ term vanishing by
    continuity; insert~\eqref{eq:one-sided}, multiply by
    $H^{2k-1}/(k!)^2$ and sum over $k$; the factor $2$ of the jump estimate
    is precisely the factor appearing in $\Theta'_p$. For (b), multiply (a) by $\gamma$
    and sum over the ghost faces
    of level $m$: each $\omega_F$ consists of the two cells sharing $F$, so a
    cell of $\mesh_{m,\Omega}$ appears in at most one neighbourhood per face,
    i.e.\ in at most $2d$ of them, giving the finite-overlap factor.
\end{proof}

\begin{remark}[$\Theta'_p$ is the superexponential ingredient]\label{rem:ghostchain}
    % Identify the superexponential factor and derive its growth rate.
    After the $p$-uniform ghost transfer (Remark~\ref{rem:strengthened}), $\Theta'_p$ is the sole \emph{super}exponential
    ingredient of the two-level rate~\eqref{eq:twolevel} --- the only other non-polynomial factor is the exponential chain
    amplification $\Xi_p$ of Proposition~\ref{prop:cutfem} (Remark~\ref{rem:where-exponential}). The successive ratios
    satisfy
    \begin{equation}
        \frac{C_M^2p^4}{k^2}\ge C_M^2p^2,
        \qquad 1\le k\le p,
    \end{equation}
    so the terms grow throughout the truncation range and the last term
    dominates. Consequently,
    \begin{equation}
        \log\Theta'_p = 2p\log p + O(p),
        \qquad \Theta'_p=(Cp)^{2p}e^{O(p)},
    \end{equation}
    despite the weights $1/(k!)^2$.

    % Explain why the Bessel bound is loose and interpret the growth.
    Completing the sum to the Bessel series gives
    \begin{equation}
        \sum_{k\ge1}\frac{(C_M^2p^4)^k}{(k!)^2}
        =I_0(2C_Mp^2)-1\sim e^{2C_Mp^2}.
    \end{equation}
    This bound is loose because the dominant index of the full series,
    $k\approx C_Mp^2$, lies beyond the truncation at $k=p$. Thus the growth
    reflects the iterated Markov estimate, but it remains the dominant source
    of the tracked $p$-dependence of $C_{\mathrm{sd}}(p)$; its sharpness is
    examined in Section~\ref{sec:p-dependence}.
\end{remark}

%%%%%%%%%%%%%%%%%%%%%%%%%%%%%%%%%%%%%%%%%%%%%%%%%%%%%%%%%%%%
\section{Vertex-patch smoother and method implementation}
\label{sec:smoother}
%%%%%%%%%%%%%%%%%%%%%%%%%%%%%%%%%%%%%%%%%%%%%%%%%%%%%%%%%%%%

% Define the active vertex patches and fix the patch enumeration.
The construction of~\cite{CuiKanschat25CutFEM} selects as patch centers every vertex in $\Omega$ and every exterior
vertex with at least one intersected adjacent cell; the exterior vertices must generate their own patches to cover
their vertex degrees of freedom. We call the selected centers \emph{admissible} and enumerate them
$\vertex_1,\dots,\vertex_J$, so that $J=J(\ell)$ is the number of local problems on level $\ell$. For each of them, let
$\widehat\omega_j$ be the full background-grid vertex patch of $\vertex_j$ and let $\omega_j$, also denoted by
$\Omega_{\ell,j}$, be the union of the active cells in $\widehat\omega_j$. Near the fictitious boundary, $\omega_j$ may
consist of fewer than $2^d$ cells even though $\widehat\omega_j$ is the full vertex patch.

% Define the two local spaces on the same patch.
The two smoother variants differ only in the local degrees of freedom. The construction from~\cite{CuiKanschat25CutFEM}
imposes a homogeneous Dirichlet condition on the entire patch boundary and uses
\begin{gather}\label{eq:classical-patch-space}
    V_{\ell,j}^{\mathrm D}
    := \operatorname{span}\bigl\{\phi_{\ell,i}:x_{\ell,i}\in
    \operatorname{int}(\widehat\omega_j)\bigr\},
\end{gather}
where $x_{\ell,i}$ is the node associated with $\phi_{\ell,i}$. Consequently none of the degrees of freedom on
$\partial\widehat\omega_j$ is updated. The central vertex $\vertex_j$ is nevertheless an interior node of
$\widehat\omega_j$, including when $\vertex_j\notin\Omega$ and $\omega_j$ is only a partial patch.
For a nodal basis function $\phi_{\ell,i}$, let $\omega_i$ be the union of all
active cells containing its node in their closure. The full-residual
construction instead uses
\begin{gather}\label{eq:full-residual-patch-space}
    V_{\ell,j}^{\mathrm{FR}}
    := \operatorname{span}\bigl\{\phi_{\ell,i}:\omega_i\subseteq\omega_j\bigr\}.
\end{gather}
It therefore updates every degree of freedom whose complete active-cell
support is confined to the patch. In particular, a degree of freedom on
$\partial\Omega_\ell$ is included when all active cells containing it belong to
$\omega_j$; no artificial Dirichlet condition is imposed on that part of the
patch boundary.

% Attribute full-residual patches and mention shyness as an optional reduction.
The full-residual construction was introduced in~\cite{wichrowski2026SBMmg}, together with \emph{shyness}, which omits
centers adjacent to fewer than a prescribed number of active cells; here we use the unfiltered patch set.

% Define an exact local correction for either patch space.
Let $V_{\ell,j}$ denote either $V_{\ell,j}^{\mathrm D}$ or $V_{\ell,j}^{\mathrm{FR}}$. Given an iterate $u^{(m)}_\ell$,
the patch correction $s_j\in V_{\ell,j}$ is defined by
\begin{gather}\label{eq:local-smoother}
    A_\ell(s_j,v_j)
    = f_\ell(v_j)-A_\ell(u^{(m)}_\ell,v_j)
    \qquad\forall v_j\in V_{\ell,j}.
\end{gather}
Here $f_\ell$ is the global load functional. Both variants restrict the same
stabilized operator and the same global residual to their respective local
spaces. Thus the term \emph{full residual} refers to retaining all residual
components associated with the confined degrees of freedom in
\eqref{eq:full-residual-patch-space}; it does not denote a different global
operator. After the update
$u^{(m)}_\ell\leftarrow u^{(m)}_\ell+s_j$, the residual is recomputed before
the next patch solve. A sequential sweep is therefore multiplicative, whereas
patches of one color may be processed simultaneously as described later in
Section~\ref{sec:schwarz}.

% Describe the repeated treatment of the boundary strip.
One smoothing step visits every interior patch once and may repeat the patches whose cells meet the cut strip $n_c\ge
    1$ times~\cite{CuiKanschat25CutFEM}. These repetitions change the strength of smoothing near $\Gamma$ but not the local
problem \eqref{eq:local-smoother}. The resulting sweep is followed by a coarse-grid correction on $V_{\ell-1}$.

\subsection{The covering problem and two local spaces}

% State the algebraic covering requirement and fix the terminology.
For the smoother to act on all components of $V_\ell$, every nodal basis function must belong to at least one local
space. We say that patch $j$ \emph{updates} the degree of freedom $i$ when $\phi_{\ell,i}\in V_{\ell,j}$; for the
classical space this holds exactly when the node $x_{\ell,i}$ is interior to $\widehat\omega_j$, for the full-residual
space exactly when $\omega_i\subseteq\omega_j$. Since a node interior to $\widehat\omega_j$ has all of its background
cells in $\widehat\omega_j$, the classical space is contained in the full-residual space on the same patch, and a
degree of freedom updated by patch $j$ in the classical construction is updated by it in the other one as well. We use
the following form of the requirement in the analysis.

\begin{assumption}[covering]\label{ass:covering}
    Every degree of freedom of $V_\ell$ is updated by at least one admissible
    patch: for each nodal basis function $\phi_{\ell,i}$ there is
    $j=j(i)\in\{1,\dots,J\}$ with
    $\phi_{\ell,i}\in V_{\ell,j(i)}$.
\end{assumption}

% Explain how covering provides an implementation-level verification check.
Assumption~\ref{ass:covering} can be checked during construction by marking every degree of freedom included in a local
space; any unmarked degree of freedom is not updated by the smoother.

\begin{remark}\label{rem:covering}
    % Explain why exterior vertices must generate patches.
    A degree of freedom at $\vertex\in\Omega_\ell\setminus\Omega$ is covered by the patch centered at $\vertex$ itself, so
    omitting exterior centers would violate Assumption~\ref{ass:covering}. Both constructions retain these patches.
\end{remark}

% Contrast the two spaces on partial patches.
On an interior patch, nodes on $\partial\widehat\omega_j$ have support outside the patch and are updated by
neighbouring patches. On a partial patch, some such nodes lie on $\partial\Omega_\ell$ and have no active support
outside $\omega_j$. The full-residual space includes precisely these confined degrees of freedom while excluding every
degree of freedom whose support reaches an active cell outside $\omega_j$.

% Explain how the numerical comparison isolates the local-space effect.
Because the variants share the operator, residual, penalty weights, patch centers, and sweep, the comparison below
isolates the local-space restriction; Section~\ref{sec:p-dependence} studies the remaining degree dependence.

\subsection{Comparison of the two variants}
\label{sec:num_setting}
% Define the primary geometry and background-mesh hierarchy used in the implementation.
The principal implementation setting throughout this work is the Poisson problem on the unit disc centered at the
origin, embedded in the square $[-1.21,1.21]^2$. The boundary is represented by the exact level-set function. The
coarsest Cartesian background mesh consists of two cells in each coordinate direction, hence $2^d$ cells, and every
subsequent level is obtained by uniform dyadic refinement. Unless stated otherwise, all numerical experiments below use
this circular configuration; the level $L$ and the smoother parameters are specified for each experiment.

% Specify the discretization and multigrid cycle used in the experiments.
The finite element space is the Gauss--Lobatto nodal $\Q_p$ space on the active cells. The tensor-product ghost-penalty
formulation of~\cite{wichrowski2026TensorGhostPenalty} is used with the standard derivative-jump weights and the global
factor $\gamma$, which may vary among experiments. The multigrid hierarchy uses one pre-smoothing and one
post-smoothing step, with an exact solve on the coarsest level. All reported iterations start from the zero vector and
stop when the Euclidean residual norm has been reduced by a factor $10^{-8}$ relative to the norm of the right-hand
side, with an iteration cap of $200$. Every patch update uses unit relaxation. Thus the semi-multiplicative experiments
use the undamped choice $\theta=1$; Corollary~\ref{cor:hybrid} proves a cut-uniform estimate only for its damped
counterpart. All full multilevel experiments intentionally use a V-cycle, which is stronger than the two-level and
conditional W-cycle results proved here and is not covered by those results.

% State the solver configuration and reporting convention for the local-space comparison.
For the comparison in Table~\ref{tab:covering}, the finest level is $L=5$ and the ghost-penalty weight is fixed at
$\gamma=10^{-1}$ for every degree. Both variants use the same assembled operator, global residual, right-hand side, and
semi-multiplicative sweep with $n_c=1$; only the local patch spaces differ. We report the number of preconditioned
GMRES iterations required by the common stopping criterion stated above.

% Report the numerical effect of removing the overconstraint.
Table~\ref{tab:covering} compares the two local-space definitions for the same discrete problem. Updating all confined
degrees of freedom lowers the iteration count at every measured degree, by a factor between $1.6$ and $1.8$ from $\Q_3$
onward. The discrete operator and global right-hand side are unchanged; the difference is entirely due to the patch
subspaces. The growth with $p$ persists in both columns: at the fixed weight $\gamma=10^{-1}$ the counts still grow by
an order of magnitude between $\Q_3$ and $\Q_6$, and neither variant converges within $200$ iterations at $\Q_7$. The
overconstraint therefore accounts for a constant factor of the degree dependence but not for its growth.
Section~\ref{sec:p-dependence} investigates the remaining dependence, with particular attention to the ghost-penalty
weights.

%% Measured with the SEMI-MULTIPLICATIVE sweep (smoother = semi-multiplicative),
%% n_c = 1, gamma = 0.1, n levels = 6 (paper L=5), symmetric Nitsche with
%% gamma_D = 5(p+1)p (nitsche sign = 1, nitsche parameter = 5), full/classical
%% patch spaces the only difference between the rows. Both Q7 runs reach the
%% 200-iteration cap (last residuals 5.6e-4 classical, 1.3e-5 full residual);
%% every other entry is sane. Script: numerical/calc/sweep_covering.sh
\begin{table}[htbp]
    \centering
    \caption{GMRES iteration counts for the classical Dirichlet patch spaces
        and the full-residual patch spaces. Both use the same global residual,
        the semi-multiplicative sweep with $n_c=1$, an $L=5$ hierarchy, the fixed
        ghost-penalty weight $\gamma=10^{-1}$, unit relaxation ($\theta=1$), symmetric Nitsche at
        $\gamma_D=5(p+1)p$, and the circular test domain. The iteration cap is
        $200$; at $p=7$ neither variant converges within it at this fixed
        $\gamma$.}
    \label{tab:covering}
    \begin{tabular}{lccccccc}
        \toprule
        local patch space                   & $\Q_1$ & $\Q_2$ & $\Q_3$ & $\Q_4$ & $\Q_5$ & $\Q_6$ & $\Q_7$ \\
        \midrule
        classical, $V_{\ell,j}^{\mathrm D}$ & $8$    & $9$    & $17$   & $34$   & $72$   & $175$  & $>200$ \\
        full residual, $V_{\ell,j}^{\mathrm{FR}}$
                                            & $7$    & $6$    & $10$   & $20$   & $44$   & $97$   & $>200$ \\
        \bottomrule
    \end{tabular}
\end{table}

\subsection{Verification of the covering}

% Identify the degrees of freedom covered by standard interior patches.
The following geometric lemma shows that cell-interior degrees of freedom and degrees of freedom on faces meeting
$\Omega$ require no special construction. It is stated for $d=2$; only the dimension-independent
Lemma~\ref{lem:coverimpl} is used in the convergence analysis.

\begin{lemma}[interior vertices of active cells, $d=2$]\label{lem:vertex}
    Under Assumption~\ref{ass:geometry}, every cell $\cell\in\mesh_{\ell,\Omega}$
    has at least one vertex in $\Omega$, and every face $F$ of
    $\mesh_{\ell,\Omega}$ with $F\cap\Omega\ne\emptyset$ has at least one
    endpoint in $\Omega$. Consequently every degree of freedom located in a
    cell interior, and every degree of freedom located on a face meeting
    $\Omega$, is updated by a patch centered at a vertex of $\Omega$, already
    in the classical construction.
\end{lemma}

\begin{proof}
    % Prove the cell and face statements from the cut geometry.
    If $\cell\cap\Omega\ne\emptyset$ and $\cell\cap\Gamma=\emptyset$, then by connectedness $\cell\subset\Omega$ and all
    vertices are interior. If $\cell$ is cut, Assumption~\ref{ass:geometry} states that $\Gamma$ crosses $\partial\cell$ in
    exactly two points lying on two distinct closed faces, and $\cell\setminus\Gamma$ has exactly two components, one
    contained in $\Omega$. Removing the two crossing points from the closed curve $\partial\cell$ leaves two open arcs,
    each belonging to the closure of one component. If the two crossed faces are opposite, each arc contains two vertices;
    if adjacent, one arc contains one vertex and the other three. In all cases the arc bounding the $\Omega$-component
    contains at least one vertex, which lies in $\Omega$ (not on $\Gamma$ by the generic position assumption). The face
    statement is analogous: if $\Gamma$ crosses the face $F$ (at most once), the portion $F\cap\Omega$ is a segment with
    one endpoint a vertex; if not, $F\subset\Omega$. A degree of freedom in the interior of $\cell$ has its node interior
    to the background patch of any vertex of $\cell$, and one in the relative interior of a face $F$ has its node interior
    to the background patch of either endpoint of $F$ (in the $2\times2$ patch around an endpoint $\vertex$, the faces
    emanating from $\vertex$ lie in the interior of the patch). Either node is therefore updated by the corresponding patch
    in the sense of~\eqref{eq:classical-patch-space}.
\end{proof}

% Isolate the degrees of freedom that require an augmented construction.
The gap between Lemma~\ref{lem:vertex} and Assumption~\ref{ass:covering} consists of vertices in
$\Omega_\ell\setminus\Omega$ and degrees of freedom on mesh entities lying entirely outside $\Omega$. The
exterior-centered patches of~\cite{CuiKanschat25CutFEM} cover this gap. The full-residual construction uses the same
patch centers and additionally includes nodes on $\partial\widehat\omega_j$ whenever their complete active support is
contained in $\omega_j$.

\begin{lemma}[the implemented patch sets cover]\label{lem:coverimpl}
    Let the patch set consist of the patches of all vertices
    $\vertex\in\Omega$ and of all vertices having at least one intersected
    adjacent cell. Then both the classical spaces
    $V_{\ell,j}^{\mathrm D}$ of~\eqref{eq:classical-patch-space} and the
    full-residual spaces $V_{\ell,j}^{\mathrm{FR}}$
    of~\eqref{eq:full-residual-patch-space} satisfy
    Assumption~\ref{ass:covering}.
\end{lemma}

\begin{proof}
    % Attach every node to the mesh entity whose relative interior contains it, and to one of its vertices.
    Every node $x_{\ell,i}$ lies in the relative interior of exactly one entity $S$ of the background grid, of some
    dimension $m$ between $0$ and $d$: a vertex, an edge, a face of any intermediate dimension, or a cell. Fix such a node,
    let $S$ be its entity and let $\vertex$ be any vertex of $S$. A cell contains $x_{\ell,i}$ in its closure exactly when
    it contains $S$, and every such cell contains $\vertex$; since the node carries a degree of freedom of $V_\ell$, at
    least one of these cells is active.

    % The patch centered at that vertex is admissible.
    The patch centered at $\vertex$ belongs to the patch set. This holds by definition when $\vertex\in\Omega$. If
    $\vertex\notin\Omega$, an active cell $\cell$ as above meets $\Omega$ and contains $\vertex\notin\Omega$, hence is
    intersected, so $\vertex$ has an intersected adjacent cell and its patch is again included.

    % Both local spaces of that patch update the node.
    Both local spaces of this patch contain $\phi_{\ell,i}$. Placing $\vertex$ at the origin, $S$ is contained in
    $[0,h_\ell]^{m}\times\{0\}^{d-m}$ in suitable coordinates, while $\widehat\omega_{\vertex}=[-h_\ell,h_\ell]^d$; a point
    of $\operatorname{relint}S$ has all coordinates of modulus strictly smaller than $h_\ell$ and therefore lies in
    $\operatorname{int}\widehat\omega_{\vertex}$, which is~\eqref{eq:classical-patch-space}. Moreover, every cell
    containing $x_{\ell,i}$ in its closure contains $\vertex$ and thus belongs to $\widehat\omega_{\vertex}$, so the active
    support $\omega_i$ of the node satisfies $\omega_i\subseteq\omega_{\vertex}$, which
    is~\eqref{eq:full-residual-patch-space}. Thus both local-space constructions cover every node.
\end{proof}

% Summarize the consequence for the analysis.
Lemma~\ref{lem:coverimpl} verifies Assumption~\ref{ass:covering} for both implemented variants. The distinction is not
whether the degrees of freedom are covered, but which additional confined degrees of freedom each patch updates. The
Schwarz analysis below applies to both variants, while Table~\ref{tab:covering} quantifies the effect of enlarging the
classical local spaces without changing the patch centers.

%%%%%%%%%%%%%%%%%%%%%%%%%%%%%%%%%%%%%%%%%%%%%%%%%%%%%%%%%%%%
\section{Cut-robust two-level vertex-patch Schwarz theory}
\label{sec:schwarz}
%%%%%%%%%%%%%%%%%%%%%%%%%%%%%%%%%%%%%%%%%%%%%%%%%%%%%%%%%%%%

% Introduce the fitted active-domain viewpoint used in the analysis.
No fictitious-domain extension of a finite element function is needed: $\Omega_\ell$ is exactly meshed by
$\mesh_{\ell,\Omega}$, and $V_\ell$ is a conforming finite element space on the fitted, shape-regular Cartesian mesh of
the polygonal domain $\Omega_\ell$. The cut enters through the energy norm, which is defined on $\Omega$ and $\Gamma$
rather than on $\Omega_\ell$. The extension property~\eqref{eq:ghost-extension} connects these domains, while
Assumption~\ref{ass:covering} supplies the local-space coverage required by the Schwarz decomposition.

\subsection{An auxiliary trace inequality}

% State the cut-cell trace inequality.
We will use the cut-cell trace inequality of Hansbo--Hansbo \cite{HansboHansbo02}: for $\cell\in\mesh_{\ell,\Gamma}$
and $w\in H^1(\cell)$,
\begin{gather}\label{eq:cut-trace}
    \|w\|^2_{0,\Gamma\cap \cell}
    \le C\bigl( h_\ell^{-1}\|w\|^2_{0,\cell} + h_\ell\,|w|^2_{1,\cell}\bigr),
\end{gather}
with $C$ independent of the cut position, by Assumption~\ref{ass:geometry}.
\subsection{Verification of the working hypotheses}

% Introduce the one-face transfer mechanism and its degree dependence.
The next lemma gives the one-face estimate used in the ghost-penalty arguments. Its weights are the $L^2$-transfer
weights $h_F^{2}\,h_F^{2k-1}/(k!)^2$, an $h_F^2$ factor times the standard penalty weights, and its prefactor is the
polynomial \emph{extrapolation} constant $C_E(q)$. The exponential $p$-dependence in the verification below enters
through this constant.

\begin{lemma}[one-face transfer]\label{lem:transfer}
    % State the transfer inequality and its uniform extrapolation constant.
    Let $\cell',\cell\in\mesh_{\ell,\Omega}$ share the face $F$, let $q\ge0$ be an integer, and let $w$ be a piecewise
    polynomial of degree $\le q$ in the normal coordinate of $F$ on $\cell'\cup \cell$. Then
    \begin{gather}\label{eq:transfer}
        \|w\|^2_{0,\cell} \le 2\,C_E(q)^2\,\|w\|^2_{0,\cell'}
        + 2(q+1)\,h_\ell^2 \sum_{k=0}^{q} \frac{h_\ell^{2k-1}}{(k!)^2}
        \bigl\|\llbracket \partial_n^k w\rrbracket\bigr\|^2_{0,F},
    \end{gather}
    where $C_E(q)$ is the adjacent-interval extrapolation constant
    \begin{gather*}
        C_E(q):=\sup_{0\ne s\in\mathbb P_q}
        \frac{\|s\|_{L^2(1,3)}}{\|s\|_{L^2(-1,1)}}.
    \end{gather*}
    % Relate the interval definition to polynomial extension between cells.
    Equivalently, after affine scaling, it is the smallest uniform constant in $\|\tilde w\|_{0,\cell}\le
        C_E(q)\,\|w\|_{0,\cell'}$ for the polynomial extension $\tilde w$ of $w|_{\cell'}$ across $F$. Moreover, $C_E(0)=1$
    and, for $q\ge1$,
    \begin{gather}\label{eq:extrap}
        c\,q^{-1/2}\,(3+2\sqrt2)^{q}
        \;\le\; C_E(q) \;\le\; (q+1)\,(3+2\sqrt2)^{q} .
    \end{gather}
\end{lemma}

\begin{proof}
    % Prove the transfer estimate through the Taylor expansion across the face.
    \emph{Step 1 (transfer across the face).}
    Let $x_n$ be the coordinate normal to $F$, vanishing on $F$ and positive
    in $\cell$, and let $\tilde w$ be the polynomial extension of
    $w|_{\cell'}$ to $\cell$. The difference $d := w|_{\cell} - \tilde w$
    is a polynomial of degree $\le q$ in $x_n$, hence equals its Taylor
    expansion at $x_n=0$:
    $d(x_\tau,x_n) = \sum_{k=0}^{q} \frac{x_n^k}{k!}\, \partial_n^k
        d(x_\tau,0)$, where $\partial_n^k d(\cdot,0)$ agrees, up to the fixed
    jump convention, with $\llbracket\partial_n^k w\rrbracket_F$.
    Termwise, $\bigl\| \tfrac{x_n^k}{k!}\,
        \llbracket\partial_n^k w\rrbracket \bigr\|^2_{0,\cell}
        = \tfrac{h_\ell^2\,h_\ell^{2k-1}}{(2k+1)(k!)^2}
        \|\llbracket\partial_n^k w\rrbracket\|^2_{0,F}$, and Cauchy--Schwarz
    over the $q+1$ terms gives $\|d\|^2_{0,\cell} \le
        (q+1)\,h_\ell^2\sum_k \tfrac{h_\ell^{2k-1}}{(k!)^2}
        \|\llbracket\partial_n^k
        w\rrbracket\|^2_F$; then $\|w\|^2_{0,\cell}\le 2\|\tilde
        w\|^2_{0,\cell}+2\|d\|^2_{0,\cell}$ gives~\eqref{eq:transfer}.

    % Establish the exponential size of the extrapolation constant.
    \emph{Step 2 (extrapolation constant).}
    The identity $C_E(0)=1$ follows from the equality of the interval lengths.
    For $q\ge1$, fix a tangential point
    $x_\tau$: the line restriction $w(x_\tau,\cdot)|_{\cell'}$ is a
    univariate polynomial $s$ of degree $\le q$; map the normal extent of
    $\cell'$ to $(-1,1)$, so that $\cell$ maps to $(1,3)$. The Chebyshev
    comparison principle gives $\|s\|_{L^\infty(1,3)}\le
        T_q(3)\,\|s\|_{L^\infty(-1,1)}$ with $T_q$ the Chebyshev polynomial,
    $T_q(3)=\cosh(q\operatorname{arccosh}3)\le(3+2\sqrt2)^q$, and the
    Nikolskii inequality (reproducing kernel of degree $q$) gives
    $\|s\|_{L^\infty(-1,1)}\le\tfrac{q+1}{\sqrt2}\,\|s\|_{L^2(-1,1)}$;
    scaling back and integrating over $x_\tau$ yields $C_E(q)\le
        (q+1)\,T_q(3)$. For the lower bound, take $s=T_q$ and set
    $a:=\operatorname{arccosh}3$. Since $T_q(\cosh u)=\cosh(qu)$ and
    $a>1$, the substitution $t=\cosh u$ gives
    \begin{gather*}
        \int_1^3 T_q(t)^2\,dt
        =\int_0^a \cosh^2(qu)\sinh u\,du
        \ge \int_{a-1/q}^{a}\cosh^2(qu)\sinh u\,du
        \ge \frac{c}{q}\cosh^2(qa)
        =\frac{c}{q}T_q(3)^2,
    \end{gather*}
    because on $[a-1/q,a]$ one has
    $\sinh u\ge\sinh(a-1)$ and
    $\cosh(qu)\ge\cosh(qa-1)\ge\tfrac12e^{-1}\cosh(qa)$.
    Thus $c>0$ is independent of $q$. Together with
    $\|T_q\|^2_{L^2(-1,1)}\le2$ and
    $T_q(3)=\tfrac12\bigl((3+2\sqrt2)^q+(3+2\sqrt2)^{-q}\bigr)$, this proves
    the lower bound in~\eqref{eq:extrap}. Thus extrapolation is genuinely
    exponential in the degree: the growth in~\eqref{eq:extrap} is intrinsic,
    not an artifact of the estimate.
\end{proof}

\begin{lemma}[geometric consequences of the resolution
        condition]\label{lem:geom}
    Under Assumption~\ref{ass:geometry}, on every level satisfying the
    resolution condition~\eqref{eq:resolution}:
    \begin{itemize}
        \item[(i)] every $\cell\in\mesh_{\ell,\Omega}$ is connected to an uncut
              cell $I(\cell)\in\mesh_{\ell,\Omega}$, $I(\cell)\subset\Omega$, by a
              chain of at most $C_d$ cells of $\mesh_{\ell,\Omega}$ in which
              consecutive cells share a face, and every such face adjacent to a
              cut cell belongs to $\mathbb F_G$; moreover
              $\operatorname{dist}(\cell,I(\cell))\le 4\sqrt d\, h_\ell$, so each
              cell serves in at most $C_d$ chains;
        \item[(ii)] every node $z$ of $\mesh_{\ell-1,\Omega}$ admits a fine cell
              $K_z\in\mesh_{\ell,\Omega}$ with $K_z\subset\Omega$ and
              $\operatorname{dist}(z,K_z)\le 4\sqrt d\, h_{\ell-1}$;
        \item[(iii)] for every $\cell'\in\mesh_{\ell-1,\Omega}$ there is an
              open, connected neighbourhood $U_{\cell'}\subset\Omega$ that
              contains $\cell'\cap\Omega$, contains $K_z$ for every node $z$
              of $\cell'$, and contains $I(K)$ for every fine cell
              $K\subset\cell'$. There are constants
              $r_*,L_*,M_*,D_*>0$, depending only on $d$ and the constants of
              Assumption~\ref{ass:geometry}, such that the rescaled set
              $h_{\ell-1}^{-1}U_{\cell'}$ has diameter at most $D_*$ and admits
              a boundary atlas by Lipschitz graph charts of radius $r_*$, graph
              slope at most $L_*$, and overlap multiplicity at most $M_*$.
              Moreover, every point of $\Omega$ belongs to at most $M_*$ of
              the sets $U_{\cell'}$. In particular, these constants are
              independent of $\ell$, of $\cell'$, and of the cut configuration;
        \item[(iv)] if moreover $\Gamma\cap\cell'\ne\emptyset$, then
              there is a patch
              $\Gamma_{\cell'}\subset\Gamma\cap\partial U_{\cell'}$ with
              $|\Gamma_{\cell'}| \ge c\, h_{\ell-1}^{d-1}$, where $c>0$
              depends only on $d$, and the family of these patches has overlap
              multiplicity at most $M_*$.
    \end{itemize}
\end{lemma}

\begin{proof}
    \emph{Part (i): the chain to an uncut cell.}

    % March inward from the cut cell along the inward normal.
    Pick $x\in \cell\cap\Omega$ and let $y := \pi(x) - 3\sqrt d\, h_\ell\, n(\pi(x))$, where $\pi$ is the closest-point
    projection and $n$ the outward normal (if $\operatorname{dist}(x,\Gamma)\ge 3\sqrt d\, h_\ell$ take $I(\cell)\ni x$
    directly). By the reach property the segment from $x$ to $y$ along $-n(\pi(x))$ lies in $\Omega$ and the distance to
    $\Gamma$ increases along it; its length is at most $3\sqrt d\, h_\ell \le \delta_0$.

    % The cells met by the segment form the chain, and it terminates at an uncut cell.
    The cells met by the segment form the chain (perturb $y$ generically so the segment crosses only faces); each contains
    points of $\Omega$, hence is active, and their number is bounded by $C_d$ since the segment has length $\le 3\sqrt
        d\,h_\ell$. The terminal cell contains $y$ with $\operatorname{dist}(y,\Gamma) = 3\sqrt d\,h_\ell$ exceeding the cell
    diameter, hence is uncut and contained in $\Omega$; stop the chain at the first such cell. Every face between
    consecutive active cells that is adjacent to a cut cell belongs to $\mathbb F_G$ by definition.

    \medskip
    \emph{Part (ii): a full fine cell near each coarse node.}

    % Reuse the fine-level march of part (i) to locate a full cell near the node.
    $z$ lies in the closure of some $\cell'\in\mesh_{\ell-1,\Omega}$;
    pick $x\in \cell'\cap\Omega$, apply the march of (i) at the fine level:
    the ball of radius $\sqrt d\, h_\ell$ around $y$ lies in $\Omega$ and
    contains a full fine cell $K_z$;
    $\operatorname{dist}(z,K_z) \le \sqrt d\,h_{\ell-1} + 4\sqrt d\, h_\ell
        \le 4\sqrt d\, h_{\ell-1}$ (using $h_\ell = h_{\ell-1}/2$).

    \medskip
    \emph{Parts (iii)--(iv): uniformly regular local neighbourhoods.}

    % Record how far the cells used by the local interpolant can lie from the coarse cell.
    Write $H:=h_{\ell-1}$ and let $x_{\cell'}$ be the centre of $\cell'$. Part (ii), the diameter of a fine cell, and
    $z\in\overline{\cell'}$ imply that every $K_z$ associated with a node of $\cell'$ lies in $B(x_{\cell'},5\sqrt d\,H)$.
    Likewise, part (i) and $h_\ell=H/2$ imply that every $I(K)$ associated with a fine cell $K\subset\cell'$ lies in
    $B(x_{\cell'},3\sqrt d\,H)$.

    % Use a ball when the coarse cell is separated from the interface.
    If $\operatorname{dist}(x_{\cell'},\Gamma)>6\sqrt d\,H$, then $x_{\cell'}\in\Omega$ because $\cell'$ is active, and we
    set
    \begin{gather*}
        U_{\cell'}:=B(x_{\cell'},\tfrac{11}{2}\sqrt d\,H).
    \end{gather*}
    Its closure lies in $\Omega$, it contains $\cell'$, whose half-diagonal
    is $\tfrac12\sqrt d\,H$, and all the cells identified in the preceding
    paragraph, and its rescaling has a Lipschitz character depending only on
    $d$.

    % Use a tangent-aligned cylinder when the coarse cell is near the interface.
    Otherwise let $x_0:=\pi(x_{\cell'})\in\Gamma$, put $n_0:=n(x_0)$, and decompose $x-x_0=t+s n_0$ with $t\in
        T_{x_0}\Gamma$. Set $R:=16\sqrt d$. Since $RH\le\delta_0/2$ by~\eqref{eq:resolution}, the reach property represents
    $\Gamma$ in the cylinder $|t|<RH$, $|s|<RH$ as a graph $s=\varphi(t)$ with $\varphi(0)=0$ and
    $\|\nabla\varphi\|_\infty\le1$; after orienting $n_0$, the domain lies below this graph. Define
    \begin{gather*}
        U_{\cell'}:=
        \{x_0+t+s n_0:\ |t|<RH,\ -RH<s<\varphi(t)\}.
    \end{gather*}
    Every cell identified in the first paragraph lies within
    $11\sqrt d\,H<RH$ of $x_0$ and is contained in $\Omega$, hence belongs
    to $U_{\cell'}$. The same applies to $\cell'\cap\Omega$, which lies
    within $7\sqrt d\,H$ of $x_0$.

    % Verify the uniform atlas, diameter, and overlap properties.
    The lower face, cylindrical side, graph of $\varphi$, and their intersections form a Lipschitz atlas with chart radius
    comparable to $H$ and uniformly bounded slope: the cylindrical side is transverse to the graph because the latter is
    written over $T_{x_0}\Gamma$. After rescaling, the chart radius is bounded below and the diameter and slopes are
    bounded above by constants depending only on $d$ and the reach bounds. Furthermore, $U_{\cell'}\subset
        B(x_{\cell'},(R+6\sqrt d)H)$ in both cases. Since the coarse-cell centres are separated by $H$, a packing argument
    gives the asserted overlap bound $M_*$.

    % Select an interface patch with uniformly positive surface measure.
    If $\Gamma\cap\cell'\ne\emptyset$, the near-interface construction applies. The graph patch
    \begin{gather*}
        \Gamma_{\cell'}:=
        \{x_0+t+\varphi(t)n_0:\ |t|<H\}
        \subset\Gamma\cap\partial U_{\cell'}
    \end{gather*}
    has surface measure at least that of its projection onto
    $T_{x_0}\Gamma$, namely
    $|\Gamma_{\cell'}|\ge\omega_{d-1}H^{d-1}$. The same packing argument
    used above bounds the overlap of these patches, which proves (iv).
\end{proof}

% State the functional consequences of the uniformly regular neighbourhoods.
Lemma~\ref{lem:geom}(iii) provides two functional inequalities on the neighbourhoods $U_{\cell'}$: a Poincar\'e
inequality for the local stability and approximation estimates of the coarse interpolant, and a trace inequality for
the interface term of the two-level bound. Their constants depend on the cut only through the uniformly bounded
Lipschitz character:

\begin{corollary}[uniform Poincar\'e and trace inequalities on the
        neighbourhoods]\label{cor:poincare}
    Let the assumptions of Lemma~\ref{lem:geom} hold, write $U := U_{\cell'}$
    and $H := h_{\ell-1}$ for $\cell'\in\mesh_{\ell-1,\Omega}$. Then for all
    $v\in H^1(U)$,
    \begin{gather}\label{eq:poincare-U}
        \|v-\bar v_{U}\|_{0,U} \le C_P\, H\, |v|_{1,U},
        \qquad \bar v_U := |U|^{-1}\!\int_U v,
    \end{gather}
    and
    \begin{gather}\label{eq:trace-U}
        \|v\|^2_{0,\partial U}
        \le C_T\bigl(H^{-1}\|v\|^2_{0,U} + H\,|v|^2_{1,U}\bigr),
    \end{gather}
    with $C_P,C_T$ depending only on $d$ and the constants of
    Assumption~\ref{ass:geometry} --- in particular independent of $\ell$, of
    $\cell'$, and of the cut configuration.
\end{corollary}

\begin{proof}
    % Reduce both estimates to a uniformly controlled family of unit-scale domains.
    Both inequalities are invariant statements about the rescaled domain $\hat U := H^{-1}U$, which by
    Lemma~\ref{lem:geom}(iii) is a bounded connected Lipschitz domain with diameter and boundary-atlas constants bounded
    uniformly. Substituting $\hat v(\hat x):=v(H\hat x)$ turns~\eqref{eq:poincare-U} and~\eqref{eq:trace-U} into the
    corresponding inequalities on $\hat U$ with $H=1$, since each side scales with the same power of $H$.

    % Apply the Poincare inequality uniformly on the rescaled domains.
    For~\eqref{eq:poincare-U}: a bounded Lipschitz domain with character $(r,L)$ and diameter $\le D$ is a John domain with
    John constant depending only on $(r,L,D,d)$, and on John domains the Poincar\'e inequality holds with a constant
    depending only on that data \cite{Bojarski88}. Lemma~\ref{lem:geom}(iii) therefore gives the same Poincar\'e constant
    for every admissible $\hat U$.

    % Construct a uniformly transverse vector field from the boundary atlas.
    For~\eqref{eq:trace-U} we use the divergence-theorem argument, which keeps the dependence explicit. In each boundary
    chart choose the constant unit vector $\xi_j$ transverse to its graph and oriented outward. The slope bound gives
    $\xi_j\cdot n\ge c(L_*)>0$ almost everywhere on the part of the boundary covered by that chart. A Lipschitz partition
    of unity $\{\eta_j\}_j$ subordinate to the atlas can be chosen on a boundary neighbourhood with
    $\|\nabla\eta_j\|_\infty\le C(r_*)$; the overlap bound $M_*$ makes this estimate uniform. Extending the partition into
    $\hat U$ and setting $\xi:=\sum_j\eta_j\xi_j$ supplies a vector field $\xi\in C^{0,1}(\overline{\hat U};\mathbb R^d)$
    with $\|\xi\|_{\infty}+\|\nabla\xi\|_{\infty}\le C$ and $\xi\cdot n \ge c>0$ almost everywhere on $\partial\hat U$,
    with $C,c$ determined only by $(r_*,L_*,M_*,d)$.

    % Use the divergence theorem and density to prove the unit-scale trace estimate.
    By density it suffices to take $\hat v\in C^\infty(\overline{\hat U})$; then
    \begin{gather*}
        c\,\|\hat v\|^2_{0,\partial\hat U}
        \le \int_{\partial\hat U}\hat v^2\,\xi\cdot n
        = \int_{\hat U}\operatorname{div}(\hat v^2\xi)
        \le C\bigl(\|\hat v\|^2_{0,\hat U}
        + 2\|\hat v\|_{0,\hat U}|\hat v|_{1,\hat U}\bigr),
    \end{gather*}
    and Young's inequality gives~\eqref{eq:trace-U} on $\hat U$.

    % Rescale the unit-scale estimates to the physical neighbourhood.
    Since volume, surface, and $H^1$-seminorm squares scale as $H^d$, $H^{d-1}$, and $H^{d-2}$, respectively, rescaling
    gives \eqref{eq:poincare-U} and~\eqref{eq:trace-U} with constants independent of $\ell$, $\cell'$, and the cut
    configuration.
\end{proof}

\begin{proposition}[verification of Assumption~\ref{ass:cutfem}, with
        explicit $p$-dependence]\label{prop:cutfem}
    Let Assumption~\ref{ass:geometry} hold and let the ghost penalty
    $g_\ell$ carry the standard weights $h_F^{2k-1}/(k!)^2$ scaled by the
    global weight $\gamma\in(0,1]$. Define the
    \emph{chain amplification}
    \begin{gather}\label{eq:Xi}
        \Xi_p := \bigl(2\,C_E(p)^2\bigr)^{C_d}
        \;\le\; \bigl(2(p+1)^2\bigr)^{C_d}\,(3+2\sqrt2)^{\,2C_d\,p},
    \end{gather}
    the $C_d$-fold iterate of the extrapolation
    constant~\eqref{eq:extrap} along the chains of
    Lemma~\ref{lem:geom}(i). Then, with constants $C$ depending only on $d$
    and the constants of Assumption~\ref{ass:geometry}:
    \begin{itemize}
        \item[(i)] the extension property~\eqref{eq:ghost-extension} holds
              with constant
              \begin{gather}\label{eq:Cext}
                  C_{\mathrm{ext}}(p) \;\le\; C\,\Xi_p\,
                  \bigl(1 + (p+1)\,C_M^2p^4\,\gamma^{-1}\bigr);
              \end{gather}
        \item[(ii)] the Nitsche flux is controlled: for all $v\in V_\ell$,
              \begin{gather*}
                  h_\ell\,\|\partial_n v\|^2_{0,\Gamma}
                  \le C_N(p)\,\bigl(|v|^2_{1,\Omega}+g_\ell(v,v)\bigr),
                  \qquad
                  C_N(p)\le C\,(1+C_M^2p^4)\,C_{\mathrm{ext}}(p);
              \end{gather*}
        \item[(iii)] for $\gamma_D\ge\gamma^*(p):=2\,C_N(p)+1$ the coercivity
              and boundedness~\eqref{eq:coercive} hold with
              \begin{gather*}
                  c_A(p)^{-1}\le C\,C_N(p),
                  \qquad
                  C_A(p)\le C\,(1+\gamma_D),
              \end{gather*}
    \end{itemize}
    all uniform in $\ell$ and the cut configuration. Every $p$-dependence
    above is polynomial --- of the trace/Markov type $C_M^2p^4$, or the
    Cauchy--Schwarz factor $(p+1)\,\gamma^{-1}$ --- except the single
    factor $\Xi_p$. Remark~\ref{rem:where-exponential} identifies its
    occurrence in the estimates and its contribution to the bounds.
\end{proposition}

\begin{proof}
    \emph{Step 1 (control of derivative jumps).}
    Fix a partial derivative $\partial_i$ and a ghost face $F$ with normal
    direction $n$, and consider the transfer face sum of~\eqref{eq:transfer}
    for $w=\partial_i v$. If $\partial_i = \partial_n$, an index shift:
    with the $(k!)^{-2}$ weights,
    $h^2\sum_{k=0}^{p-1} \tfrac{h^{2k-1}}{(k!)^2}
        \|\llbracket\partial_n^k \partial_n v\rrbracket\|^2_F
        = \sum_{k'=1}^{p} k'^2\, \tfrac{h^{2k'-1}}{(k'!)^2}
        \|\llbracket\partial_n^{k'} v\rrbracket\|^2_F
        \le p^2\gamma^{-1}\, g_{\ell,F}(v,v)$,
    where $g_{\ell,F}$ is the contribution of $F$ to $g_\ell$. If
    $\partial_i$ is tangential to $F$, it commutes with the jump, and the
    face Markov inequality $\|\partial_i s\|_F \le C_M p^2 h^{-1}\|s\|_F$
    (the one-dimensional inequality~\eqref{eq:markov} applied per tangential
    line) makes the $h^2$ cancel the Markov cost:
    $h^2\, h^{2k-1}\|\llbracket\partial_n^k\partial_i
        v\rrbracket\|^2_F \le C_M^2 p^4
        h^{2k-1}\|\llbracket\partial_n^k v\rrbracket\|^2_F$, whence the bound
    $C_M^2p^4\gamma^{-1}\,g_{\ell,F}(v,v)$; the term $k=0$, for which
    $g_{\ell,F}$ offers no counterpart, vanishes outright, since
    $\llbracket\partial_i v\rrbracket=\partial_i\llbracket v\rrbracket=0$ by
    the continuity of $v\in V_\ell$ across interior faces. In both cases, including
    the Cauchy--Schwarz factor $2(p+1)$ of~\eqref{eq:transfer},
    \begin{gather*}
        2(p+1)\,h^2 \sum_{k=0}^{p} \frac{h^{2k-1}}{(k!)^2}
        \bigl\|\llbracket\partial_n^k\partial_i v\rrbracket\bigr\|^2_{0,F}
        \;\le\; c_F(p)\, g_{\ell,F}(v,v),
        \qquad
        c_F(p) := 2(p+1)\max\{p^2,C_M^2p^4\}\,\gamma^{-1}.
    \end{gather*}
    The factor $(p+1)\,\gamma^{-1}$ results from Cauchy--Schwarz over the
    $p+1$ Taylor orders and division by the global weight; its analogue for
    reweighted penalties appears in~\eqref{eq:profile-constants}. The factor
    $C_M^2p^4$ results from one face Markov inequality.

    \emph{Step 2 (extension).}
    Let $\cell$ be a cut cell and $\cell = K_0,\dots,K_M = I(\cell)$ its
    chain from Lemma~\ref{lem:geom}(i), $M\le C_d$. Set $a_j :=
        \|\partial_i v\|^2_{0,K_j}$. Lemma~\ref{lem:transfer} across the chain
    face $F_j$ (with $q\le p$) together with Step 1 gives $a_j \le
        2C_E(p)^2\, a_{j+1} + c_F(p)\, g_{\ell,F_j}(v,v)$, and iterating along
    the chain (using $2C_E(p)^2\ge1$),
    \begin{gather*}
        \|\partial_i v\|^2_{0,\cell}
        \le \Xi_p\Bigl( \|\partial_i v\|^2_{0,I(\cell)}
        + c_F(p) \sum_{F\in\text{chain}} g_{\ell,F}(v,v)\Bigr).
    \end{gather*}
    Every chain face adjacent to a cut cell belongs to $\mathbb F_G$; faces between two uncut cells occur only
    after the chain has entered $\Omega$, where the transfer is not needed
    (truncate the chain at the first uncut cell). Since $I(\cell)\subset
        \Omega$ and each cell and face serves in at most $C_d$ chains, summing
    over all cut cells and the $d$ partial derivatives gives
    $|v|^2_{1,\Omega_\ell\setminus\Omega} \le \sum_{\cell\ \text{cut}}
        |v|^2_{1,\cell} \le C_d\,\Xi_p\bigl(1+c_F(p)\bigr)\bigl( |v|^2_{1,\Omega}
        + g_\ell(v,v)\bigr)$,
    which is~\eqref{eq:ghost-extension} with the constant~\eqref{eq:Cext}.

    \emph{Step 3 (Nitsche control and coercivity).}
    By the cut-cell trace inequality~\eqref{eq:cut-trace} applied to
    $\partial_i v$ on each cut cell, followed by the Markov
    inequality~\eqref{eq:markov} (per coordinate) on the full cell,
    $h^2|\nabla v|^2_{1,\cell}\le d\,C_M^2p^4\,\|\nabla v\|^2_{0,\cell}$,
    \begin{gather*}
        h\,\|\partial_n v\|^2_{0,\Gamma}
        \le C \sum_{\cell\ \text{cut}} \bigl( \|\nabla v\|^2_{0,\cell}
        + h^2 |\nabla v|^2_{1,\cell}\bigr)
        \le C(1+C_M^2p^4) \sum_{\cell\ \text{cut}} \|\nabla v\|^2_{0,\cell}
        \le C_N(p)\bigl( |v|^2_{1,\Omega} + g_\ell(v,v)\bigr),
    \end{gather*}
    the last step by Step 2, since $\sum_{\cell\ \text{cut}}\|\nabla
        v\|^2_{0,\cell} \le |v|^2_{1,\Omega} +
        |v|^2_{1,\Omega_\ell\setminus\Omega}$; this is (ii) with $C_N(p)\le
        C(1+C_M^2p^4)\bigl(1+C_{\mathrm{ext}}(p)\bigr)$. Then for any
    $\epsilon>0$,
    \begin{gather*}
        A_\ell(v,v) \ge |v|^2_{1,\Omega}
        - \epsilon^{-1} h\|\partial_n v\|^2_{0,\Gamma}
        - \epsilon\, h^{-1}\|v\|^2_{0,\Gamma}
        + \gamma_D h^{-1}\|v\|^2_{0,\Gamma} + g_\ell(v,v),
    \end{gather*}
    and choosing $\epsilon = 2 C_N(p)$, $\gamma_D \ge \gamma^*(p) =
        2C_N(p) + 1$: the flux term is, by (ii), at most
    $\tfrac12(|v|^2_{1,\Omega}+g_\ell(v,v))$, so
    $A_\ell(v,v) \ge \tfrac12\bigl(|v|^2_{1,\Omega}+g_\ell(v,v)\bigr) +
        h^{-1}\|v\|^2_{0,\Gamma}$. Reinstating $h\|\partial_n
        v\|^2_{0,\Gamma}$ in the norm by one more application of (ii),
    $N_\ell(v)^2 \le (1+C_N(p))\bigl(|v|^2_{1,\Omega}+g_\ell(v,v)\bigr) +
        h^{-1}\|v\|^2_{0,\Gamma} \le \bigl(2C_N(p)+3\bigr)A_\ell(v,v)$,
    giving $c_A(p)^{-1}\le C\,C_N(p)$. Boundedness is Cauchy--Schwarz term
    by term: every term of $A_\ell$ pairs against the matching weights of
    $N_\ell$ with constant $O(1)$ except the Nitsche penalty, whose
    coefficient $\gamma_D$ exceeds the unit weight of
    $h^{-1}\|v\|^2_{0,\Gamma}$ in $N_\ell$, whence $C_A(p)\le
        C(1+\gamma_D)$.
\end{proof}

\begin{remark}[source of the non-polynomial $p$-dependence]
    \label{rem:where-exponential}
    % Classify the sources of degree dependence.
    Trace and Markov inequalities and the Cauchy--Schwarz factor $(p+1)\gamma^{-1}$ contribute polynomial factors. The
    chain amplification $\Xi_p$ is exponential in $p$ because each of at most $C_d$ transfer steps incurs the sharp
    extrapolation constant $C_E(p)\simeq(3+2\sqrt2)^p$; improving this factor therefore requires an argument that avoids
    cellwise extrapolation, and whether the true constant of~\eqref{eq:ghost-extension} is polynomial in $p$ is open. In
    the stable decomposition, $\Xi_p$ is dominated by the superexponential ghost-energy factor
    $\Theta'_p=(Cp)^{2p}e^{O(p)}$ from Lemma~\ref{lem:ghostchain}. For derivative-order-dependent weights, the
    corresponding factors $S(\gamma)$ and $B(\gamma)$ are considered in Remark~\ref{rem:duality}.
\end{remark}

\begin{remark}[admissible versus experimental Nitsche
        parameters]\label{rem:admissible-gammaD}
    % Separate the certified sufficient condition from the experimental choice.
    Throughout, a Nitsche parameter $\gamma_D$ is called \emph{admissible} when Assumption~\ref{ass:cutfem} holds for it.
    Proposition~\ref{prop:cutfem}(iii) certifies admissibility only under the sufficient condition
    $\gamma_D\ge\gamma^*(p)=2C_N(p)+1$, and this condition is pessimistic: $C_N(p)$ contains the chain amplification
    $\Xi_p$ and the factor $\gamma^{-1}$, so $\gamma^*(p)$ grows exponentially in $p$ and without bound as the global ghost
    weight decreases. The numerical experiments in this paper use the standard polynomial scaling $\gamma_D=5p(p+1)$, which
    the proposition does not certify, in particular not for the small weights $\gamma$ of Section~\ref{sec:cut-adaptive}.
    The convergence theory applies to every pair $(\gamma,\gamma_D)$ realizing Assumption~\ref{ass:cutfem}; runs outside
    the certified range are reported as measurements, not as instances of the theory.
\end{remark}

\begin{proposition}[shifted quasi-interpolation]\label{prop:quasi}
    Under Assumptions~\ref{ass:geometry} and~\ref{ass:cutfem} there is a
    linear operator $\Pi_{\ell-1} : V_\ell \to V_{\ell-1}$ such that for all
    $v\in V_\ell$,
    \begin{gather}
        \label{eq:Q1}
        |\Pi_{\ell-1} v|_{1,\Omega_{\ell-1}} \le C_\Pi(p)\, |v|_{1,\Omega},
        \\
        \label{eq:Q2}
        \|v - \Pi_{\ell-1} v\|_{0,\Omega_\ell}
        \le C_Q(p,\gamma)\, h_{\ell-1} \bigl( |v|^2_{1,\Omega} +
        g_\ell(v,v)\bigr)^{1/2}
        \le C_Q(p,\gamma)\, h_{\ell-1}\, N_\ell(v),
    \end{gather}
    with constants uniform in $\ell$ and the cut configuration, and, with
    $C$ depending only on $d$ and the constants of
    Assumption~\ref{ass:geometry},
    \begin{gather}\label{eq:CQ}
        C_\Pi(p) \le C\, C_M p^2\, Z_p,
        \qquad
        C_Q(p,\gamma)^2 \le C\,\bigl( Z_p^2
        + (p+1)\,\Xi_p\,(1+\gamma^{-1}) \bigr).
    \end{gather}
    Here $\Xi_p$ is the chain amplification of
    Proposition~\ref{prop:cutfem}, $\gamma$ the global ghost weight, and $Z_p
        \le (p+1)^d\,(2+32\sqrt d)^{\,dp}$ is the nodal evaluation constant
    of~\eqref{eq:nodal-bound} below --- the same Chebyshev--Nikolskii
    mechanism as $C_E(p)$ in~\eqref{eq:extrap}, over the node-to-cell
    distance instead of one cell width. The only new dependencies beyond
    Proposition~\ref{prop:cutfem} are thus one inverse inequality
    ($C_Mp^2$), one Cauchy--Schwarz factor ($(p+1)\,\gamma^{-1}$) and the
    evaluation constant $Z_p$.
\end{proposition}

\begin{proof}
    \emph{Step 1 (construction).}
    For each Gauss--Lobatto node $z$ of $V_{\ell-1}$ on $\mesh_{\ell-1,\Omega}$
    (the nodal basis fixed in Section~\ref{sec:theory})
    let $K_z\subset\Omega$ be the interior fine cell of
    Lemma~\ref{lem:geom}(ii) and let $q_z\in\Q_p(K_z)$ be the $L^2(K_z)$
    projection of $v$. Define $\Pi_{\ell-1}v\in V_{\ell-1}$ by its nodal
    values $(\Pi_{\ell-1}v)(z) := q_z(z)$, the polynomial evaluated at $z$.
    Since $\operatorname{dist}(z,K_z)\le 4\sqrt d\, h_{\ell-1} = 8\sqrt d\,
        h_\ell$ is a bounded multiple of the cell size, polynomial norm growth
    over such distances is bounded:
    \begin{gather}\label{eq:nodal-bound}
        |q_z(z)| \le Z_p\, h_\ell^{-d/2}\, \|q_z\|_{0,K_z}
        \le Z_p\, h_\ell^{-d/2}\, \|v\|_{0,K_z}.
    \end{gather}
    Indeed, mapping each coordinate extent of $K_z$ to $(-1,1)$ places the
    corresponding coordinate of $z$ at $|t|\le1+16\sqrt d$; per coordinate
    the Chebyshev comparison principle gives the growth factor
    $T_p(|t|)\le(2|t|)^p\le(2+32\sqrt d)^p$ and the Nikolskii inequality
    costs $(p+1)/\sqrt2$, exactly as in the proof of
    Lemma~\ref{lem:transfer}; tensorizing over the $d$ coordinates yields
    $Z_p\le(p+1)^d(2+32\sqrt d)^{dp}$.
    $\Pi_{\ell-1}$ reproduces constants: if $v\equiv c$ then $q_z \equiv c$
    and all nodal values are $c$.

    \emph{Step 2 (local stability, proof of~\eqref{eq:Q1}).}
    Fix $\cell'\in\mesh_{\ell-1,\Omega}$, let $U_{\cell'}$ be the domain of
    Lemma~\ref{lem:geom}(iii), it contains $K_z$ for every node $z$ of
    $\cell'$, and let $c := \bar v_{U_{\cell'}}$ be the mean of $v$
    over it (well defined: $U_{\cell'}\subset\Omega\subseteq\Omega_\ell$).
    Using constant reproduction, the inverse inequality on $\cell'$, the
    nodal basis representation, \eqref{eq:nodal-bound} applied to $v-c$,
    and $h_{\ell-1}=2h_\ell$,
    \begin{gather*}
        |\Pi_{\ell-1}v|_{1,\cell'} = |\Pi_{\ell-1}(v-c)|_{1,\cell'}
        \le C\,C_Mp^2\, h_{\ell-1}^{-1+d/2} \max_{z\in \cell'} |q_z(z) - c|
        \le C\,C_Mp^2 Z_p\, h_{\ell-1}^{-1} \|v-c\|_{0,U_{\cell'}},
    \end{gather*}
    and the uniform Poincar\'e inequality~\eqref{eq:poincare-U}
    bounds this by $C\,C_Mp^2Z_p\,|v|_{1,U_{\cell'}}$; the inverse
    inequality costs one Markov factor $C_Mp^2$, and the nodal
    representation costs only an absolute constant, since the
    Gauss--Lobatto quadrature error of $w^2$ is one-signed and the weights
    sum to the cell volume, so $\|w\|^2_{0,\cell'}\le
        |\cell'|\max_{z}|w(z)|^2$ for $w\in\Q_p$. Squaring and summing over
    $\cell'$, with the finite overlap of the $U_{\cell'}$,
    gives~\eqref{eq:Q1}.

    \emph{Step 3 (approximation on the interior part).}
    With $\cell'$, $c$ as above, for the portion of $\cell'\cap\Omega_\ell$
    consisting of fine cells contained in $\Omega$:
    $\|v-\Pi_{\ell-1}v\|_{0,\cell'\cap\Omega} \le \|v-c\|_{0,U_{\cell'}} +
        \|\Pi_{\ell-1}(v-c)\|_{0,\cell'}$, and both terms are bounded by
    $C\,Z_p\,h_{\ell-1}|v|_{1,U_{\cell'}}$ by the Poincar\'e inequality and
    the nodal bound as in Step 2 (without the inverse inequality, so no
    Markov factor).

    \emph{Step 4 (approximation on the strip).}
    Let $K\subset \cell'$ be an active fine cell not contained in $\Omega$
    (hence cut) and let $K=K_0,\dots,K_M$ be its chain from
    Lemma~\ref{lem:geom}(i), ending in $I(K)\subset\Omega$ within distance
    $4\sqrt d\,h_\ell$ of $K$. Lemma~\ref{lem:geom}(iii) gives
    $I(K)\subset U_{\cell'}$. Apply
    Lemma~\ref{lem:transfer} to $w := v-c$ along the chain (the jumps of
    $w$ equal those of $v$):
    \begin{gather*}
        \|v-c\|^2_{0,K}
        \le \Xi_p\Bigl( \|v-c\|^2_{0,I(K)}
        + 2(p+1)\, h_\ell^2 \sum_{F\in\text{chain}} \sum_{k=0}^{p}
        \frac{h_\ell^{2k-1}}{(k!)^2}
        \|\llbracket\partial_n^k v\rrbracket\|^2_{0,F}\Bigr),
    \end{gather*}
    the chain iteration producing the same amplification $\Xi_p$ as in
    Proposition~\ref{prop:cutfem}.
    Adding the $\Pi_{\ell-1}(v-c)$ part as in Step 3, summing over the cut
    cells of $\cell'$ (bounded chain reuse), then over $\cell'$ (finite
    overlap), and applying the Poincar\'e inequality once more,
    \begin{gather*}
        \|v-\Pi_{\ell-1}v\|^2_{0,\Omega_\ell}
        \le C\bigl(Z_p^2+\Xi_p\bigr)\, h_{\ell-1}^2\, |v|^2_{1,\Omega}
        + C\,(p+1)\,\Xi_p\, T_\ell ,
        \qquad
        T_\ell := h_\ell^2 \sum_{F\in\mathbb F_G}\sum_{k=0}^{p}
        \frac{h_\ell^{2k-1}}{(k!)^2}
        \|\llbracket\partial_n^k v\rrbracket\|^2_{0,F}.
    \end{gather*}
    The transfer sum is the ghost penalty itself, up to the global weight:
    the $k=0$ terms vanish since $v\in V_\ell$ is continuous across interior
    faces, and the orders $k\ge1$ carry exactly the standard weights, so
    $T_\ell = \gamma^{-1} h_\ell^2\, g_\ell(v,v)$. Hence~\eqref{eq:Q2}
    follows with the constant~\eqref{eq:CQ} and the factor $h_{\ell-1}^2 =
        4h_\ell^2$ throughout --- no ghost-chain estimate and no extension
    property enter, the ghost penalty being bounded by itself at the price
    $\gamma^{-1}$.
\end{proof}

\subsection{The stable decomposition}

% Assumptions and patch notation for the decomposition.
Assumption~\ref{ass:geometry} requires the domain boundary to be resolved on every level used in the analysis and the
active domains to be nested, which supports the cut-independent geometric and intergrid estimates.
Assumption~\ref{ass:cutfem}, verified in Proposition~\ref{prop:cutfem}, relates these estimates to the stabilized
energy norm through coercivity and discrete extension. Finally, Assumption~\ref{ass:covering}, verified for both patch
constructions in Lemma~\ref{lem:coverimpl}, ensures that every fine-grid degree of freedom can be assigned to a local
patch space. Recall that $\omega_j=\Omega_{\ell,j}$ is the union of the active cells in the background-grid patch
centered at the admissible vertex $\vertex_j$, and that $V_{\ell,j}$ denotes either of the associated local spaces
introduced in Section~\ref{sec:smoother}. These properties yield the stable decomposition that connects the CutFEM
discretization to the two-level Schwarz convergence theory.

% State the stable decomposition with its stabilization-parameter dependence.
\begin{theorem}[stable decomposition]\label{thm:additive}
    Let Assumptions~\ref{ass:geometry}, \ref{ass:cutfem} and
    \ref{ass:covering} hold. Let
    $V_{\ell,0}:=\prol{\ell}V_{\ell-1}$ and let $\{V_{\ell,j}\}_{j=1}^J$
    be the vertex-patch subspaces. For fixed global ghost weight
    $\gamma\in(0,1]$ and an admissible Nitsche parameter $\gamma_D$, there is
    $C_{\mathrm{sd}}=C_{\mathrm{sd}}(p,\gamma,\gamma_D)$ independent of
    $\ell$ and of the cut configuration such that every $v\in V_\ell$ admits
    a decomposition $v=\sum_{j=0}^J v_j$, $v_j\in V_{\ell,j}$, with
    \begin{gather}\label{eq:stable-split}
        \sum_{j=0}^{J} \energy{v_j}_\ell^2 \le C_{\mathrm{sd}}\,
        \energy{v}_\ell^2 .
    \end{gather}
    Moreover, the coarse component is the prolongation
    $v_0=\prol{\ell}\widehat v_0$ of a function $\widehat v_0\in V_{\ell-1}$
    whose \emph{coarse} energy obeys the same bound,
    \begin{gather}\label{eq:stable-split-coarse}
        \energy{\widehat v_0}_{\ell-1}^2 \le C_{\mathrm{sd}}\,
        \energy{v}_\ell^2 ,
    \end{gather}
    the form in which the decomposition enters the rediscretized coarse
    solve of Theorem~\ref{thm:twolevel}(iii).
    For prescribed degree-dependent choices $\gamma=\gamma(p)$ and
    $\gamma_D=\gamma_D(p)$, we abbreviate this constant by
    $C_{\mathrm{sd}}(p)$.
    The consequences for the two-level solvers are drawn in
    Theorem~\ref{thm:twolevel} below.
\end{theorem}

\begin{proof}
    % Reduce the energy estimate to the mesh-dependent norm.
    Write $h:=h_\ell$, $H:=h_{\ell-1}=2h$. By the norm equivalence~\eqref{eq:coercive} it suffices to construct the
    decomposition with $N$-norms on both sides. Constants below may depend on $p$, $\gamma$, and $\gamma_D$, but not on
    $\ell$ or on the cut configuration.

    % Construct the coarse representative and the fine-level remainder.
    \emph{Step 1 (coarse component).}
    Let $\Pi_{\ell-1} : V_\ell \to V_{\ell-1}$ be the shifted
    quasi-interpolant of Proposition~\ref{prop:quasi}, with the stability
    and approximation properties \eqref{eq:Q1}--\eqref{eq:Q2}.
    Set
    \begin{gather*}
        \widehat v_0:=\Pi_{\ell-1}v\in V_{\ell-1},
        \qquad
        v_0:=\prol\ell\widehat v_0\in V_{\ell,0},
        \qquad
        w:=v-v_0\in V_\ell .
    \end{gather*}
    Since $\prol\ell$ is the identity on functions on $\Omega_\ell$,
    $w=v-\Pi_{\ell-1}v$ there, and \eqref{eq:Q2} applies to $w$ directly.

    % Split the remainder according to internal patch degrees of freedom.
    \emph{Step 2 (nodal splitting of the remainder).}
    By Assumption~\ref{ass:covering} each fine degree of freedom $i$ is
    updated by at least one admissible patch; fix one such $j(i)$ and set
    \begin{gather*}
        v_j := \sum_{i:\, j(i)=j} w_i\, \phi_{\ell,i} \;\in V_{\ell,j},
        \qquad
        \sum_{j=1}^J v_j = w ,
    \end{gather*}
    where $w_i$ are the coefficients of $w$. Let
    $\|\cdot\|_{\mathrm{GL},\cell}$ be the norm induced by tensor-product
    Gauss--Lobatto quadrature on $\cell$. For $q\in\Q_p(\cell)$,
    \begin{gather*}
        \|q\|^2_{0,\cell}
        \le \|q\|^2_{\mathrm{GL},\cell}
        \le C_{\mathrm{GL}}(p,d)\|q\|^2_{0,\cell},
        \qquad
        C_{\mathrm{GL}}(p,d):=(2+p^{-1})^d\le3^d .
    \end{gather*}
    In one dimension the quadrature excess occurs only in the highest
    Legendre mode, for which the quadrature-to-exact norm ratio is
    $2+p^{-1}$~\cite[Chap.~5]{CanutoHussainiQuarteroniZang06},
    originally due to~\cite{CanutoQuarteroni82}; the stated constant is
    obtained by tensorization.
    Since the nodal splitting partitions the values of $w$ at every
    Gauss--Lobatto node,
    $\sum_j\|v_j\|^2_{\mathrm{GL},\cell}
        =\|w\|^2_{\mathrm{GL},\cell}$. Set
    $M_p:=dC_M^2p^4$. Summing over the cells and applying the first
    inequality of~\eqref{eq:markov} per coordinate gives
    \begin{gather}\label{eq:nodal-stable}
        \sum_{j=1}^J \|v_j\|^2_{0,\Omega_\ell}
        \le C_{\mathrm{GL}}(p,d)\|w\|^2_{0,\Omega_\ell},
        \qquad
        \sum_{j=1}^J |v_j|^2_{1,\Omega_\ell}
        \le C_{\mathrm{GL}}(p,d)M_p h^{-2}
        \|w\|^2_{0,\Omega_\ell}.
    \end{gather}
    By~\eqref{eq:Q2} and $H=2h$,
    \begin{gather}\label{eq:w-small}
        h^{-2}\|w\|^2_{0,\Omega_\ell}
        \le C_Q(p,\gamma)^2\,(H/h)^2
        \bigl( |v|^2_{1,\Omega} + g_\ell(v,v) \bigr)
        \le 4C_Q(p,\gamma)^2\, N_\ell(v)^2 .
    \end{gather}

    % Bound every term in the norm of a local component.
    \emph{Step 3 (energy of the local pieces).}
    Fix $j\ge1$. The $N_\ell$-norm of $v_j$ has four terms. Volume:
    $\|\nabla v_j\|_{0,\Omega} \le |v_j|_{1,\Omega_\ell}$. Nitsche terms: by
    the cut-cell trace inequality~\eqref{eq:cut-trace}, applied cellwise to
    $v_j$ and to the components of $\nabla v_j$,
    \begin{gather*}
        h^{-1}\|v_j\|^2_{0,\Gamma}
        \le C\bigl(h^{-2}\|v_j\|^2_{0,\omega_j}
        +|v_j|^2_{1,\omega_j}\bigr),
        \\
        h\|\partial_n v_j\|^2_{0,\Gamma}
        \le C\biggl(|v_j|^2_{1,\omega_j}
        +h^2\sum_{\cell\subset\omega_j}\|D^2v_j\|^2_{0,\cell}\biggr)
        \le C(1+M_p)|v_j|^2_{1,\omega_j}.
    \end{gather*}
    Here $C$ is the cut-independent, degree-free constant
    of~\eqref{eq:cut-trace}. The last inequality follows from the first
    inequality of~\eqref{eq:markov}, applied per coordinate to each component
    of $\nabla v_j$, exactly as in the proof of
    Proposition~\ref{prop:cutfem}(ii). Since $v_j$ vanishes outside
    $\omega_j$, Lemma~\ref{lem:ghostchain}(b) gives
    \begin{gather*}
        g_\ell(v_j,v_j)\le G_p|v_j|^2_{1,\omega_j},
        \qquad G_p:=2d\gamma\Theta'_p .
    \end{gather*}
    Summing over $j$ and using~\eqref{eq:nodal-stable},
    \eqref{eq:w-small},
    \begin{gather}\label{eq:local-sum}
        \begin{aligned}
            \sum_{j=1}^J N_\ell(v_j)^2
             & \le C(1+M_p+G_p)
            \sum_{j=1}^J\bigl(
            |v_j|^2_{1,\Omega_\ell}
            +h^{-2}\|v_j\|^2_{0,\Omega_\ell}\bigr)
            \\
             & \le 4C\,C_{\mathrm{GL}}(p,d)
            (1+M_p+G_p)(1+M_p)
            C_Q(p,\gamma)^2N_\ell(v)^2 .
        \end{aligned}
    \end{gather}
    Here $C$ depends only on $d$, while
    $C_{\mathrm{GL}}(p,d)\le3^d$ is degree-uniform. In particular, the ghost
    factor $G_p$ occurs only once in~\eqref{eq:local-sum}.

    % Control the bulk and ghost terms of the coarse representative.
    \emph{Step 4 (energy of the coarse component).}
    By Lemma~\ref{lem:prolongation} it suffices to bound
    $N_{\ell-1}(\widehat v_0)$. The volume term is controlled
    by~\eqref{eq:Q1}. For the coarse-level ghost penalty,
    Lemma~\ref{lem:ghostchain}(b) gives
    $g_{\ell-1}(\widehat v_0,\widehat v_0)\le
        \gamma\,2d\,\Theta'_p
        |\widehat v_0|^2_{1,\Omega_{\ell-1}}$, bounded by~\eqref{eq:Q1}.
    For the Nitsche terms, the
    trace inequality \eqref{eq:cut-trace} on coarse cut cells (applied to
    $\widehat v_0$ and to $\nabla\widehat v_0$, with the inverse inequality
    on the full cells) gives
    \begin{gather*}
        H^{-1}\|\widehat v_0\|^2_{0,\Gamma}
        + H\|\partial_n\widehat v_0\|^2_{0,\Gamma}
        \le C(p)\Bigl( H^{-2}\!\!\sum_{\cell'\in\mesh_{\ell-1,\Gamma}}\!\!
        \|\widehat v_0\|^2_{0,\cell'}
        + |\widehat v_0|^2_{1,\Omega_{\ell-1}}
        \Bigr).
    \end{gather*}

    % Estimate the nonconstant part on each coarse cut cell.
    Fix $\cell'\in\mesh_{\ell-1,\Gamma}$, let $U := U_{\cell'}\subset\Omega$ be its neighbourhood from
    Lemma~\ref{lem:geom}(iii) and $c' := \bar v_{U}$ the mean of $v$ over it, and split, using that $\Pi_{\ell-1}$
    reproduces constants, $\widehat v_0=\Pi_{\ell-1}(v-c')+c'$ on $\cell'$. For the first part, the nodal
    bound~\eqref{eq:nodal-bound} applied to $v-c'$ exactly as in Step 2 of the proof of Proposition~\ref{prop:quasi} gives
    \begin{gather*}
        \|\Pi_{\ell-1}(v-c')\|^2_{0,\cell'}
        \le C(p)\,\|v-c'\|^2_{0,U}
        \le C(p)\, H^2\, |v|^2_{1,U} .
    \end{gather*}

    % Recover the constant part from a physical-boundary patch.
    For the constant part, let $\Gamma_{\cell'}\subset\Gamma\cap\partial U$ be the patch from Lemma~\ref{lem:geom}(iv). Its
    surface measure bound $|\Gamma_{\cell'}|\ge c\,H^{d-1}$ and the trace inequality~\eqref{eq:trace-U} of
    Corollary~\ref{cor:poincare} on $U$ yield
    \begin{gather*}
        H^{d}\,|c'|^2
        \le C\,H\,\|c'\|^2_{0,\Gamma_{\cell'}}
        \le C\,H\bigl( \|v\|^2_{0,\Gamma_{\cell'}}
        + \|v-c'\|^2_{0,\Gamma_{\cell'}} \bigr)
        \le C\bigl( H\,\|v\|^2_{0,\Gamma_{\cell'}} + H^2\,|v|^2_{1,U}\bigr),
    \end{gather*}
    the last step by the trace inequality followed by the Poincar\'e
    inequality on $U$. Hence
    \begin{gather*}
        H^{-2}\|\widehat v_0\|^2_{0,\cell'}
        \le C(p)\bigl(H^{-1}\|v\|^2_{0,\Gamma_{\cell'}}
        +|v|^2_{1,U_{\cell'}}\bigr).
    \end{gather*}

    % Sum the coarse-cell estimates using finite overlap.
    Summing over the coarse cut cells, the $U_{\cell'}$ and $\Gamma_{\cell'}$ have finite overlap and, with $H=2h$, give
    \begin{gather*}
        H^{-2}\!\!\sum_{\cell'\in\mesh_{\ell-1,\Gamma}}\!\!
        \|\widehat v_0\|^2_{0,\cell'}
        \le C(p)\bigl( h^{-1}\|v\|^2_{0,\Gamma} + |v|^2_{1,\Omega} \bigr)
        \le C(p)\, N_\ell(v)^2 .
    \end{gather*}
    Together with the preceding volume and ghost estimates, this proves
    $N_{\ell-1}(\widehat v_0)^2\le C(p,\gamma)N_\ell(v)^2$. The
    $N$-norm estimate proved in Lemma~\ref{lem:prolongation} therefore gives
    $N_\ell(v_0)^2\le C(p,\gamma)N_\ell(v)^2$.

    % Combine the coarse and local estimates and return to the energy norm.
    \emph{Step 5 (conclusion).}
    The decomposition $v=v_0+\sum_{j=1}^Jv_j$ is exact by Step 2. Steps 3--4,
    boundedness of $A_\ell$, and coercivity on the right-hand side give
    \begin{gather*}
        \sum_{j=0}^J\energy{v_j}_\ell^2
        \le C_A\sum_{j=0}^J N_\ell(v_j)^2
        \le C(p,\gamma)C_A N_\ell(v)^2
        \le C(p,\gamma)\frac{C_A}{c_A}\energy{v}_\ell^2 .
    \end{gather*}
    This is~\eqref{eq:stable-split}. Step 4 also bounds the coarse
    representative in the coarse energy,
    $\energy{\widehat v_0}_{\ell-1}^2 \le C_A N_{\ell-1}(\widehat v_0)^2
        \le C(p,\gamma)\,(C_A/c_A)\,\energy{v}_\ell^2$, which
    is~\eqref{eq:stable-split-coarse}. Both hold with a common
    $C_{\mathrm{sd}}=C_{\mathrm{sd}}(p,\gamma,\gamma_D)$, independent of
    $\ell$ and of the cut configuration.
\end{proof}

% Relate the stable decomposition to existing domain-decomposition results.
For cell patches without a coarse space, the corresponding decomposition was studied by de Prenter, Verhoosel, and van
Brummelen~\cite{PrenterVerhooselBrummelen19}; see also the survey~\cite{dePrenterVerhooselvanBrummelenLarsonBadia23}.
Their preconditioner is independent of the cut configuration but deteriorates with the mesh size, with eigenvalue
ratios observed numerically to grow like $h^{-2}$, because no coarse level is present. The combination of a multigrid
coarse space with vertex patches removes this mesh-size dependence and, to the best of our knowledge, has not
previously been analyzed. Gross and Reusken~\cite{GrossReusken23} analyze a two-subspace additive Schwarz
preconditioner, splitting the CutFEM space into the finite element space on the background mesh and the span of the cut
basis functions. Their splitting is stable uniformly in the mesh size and in the cut position, for both the interface
and the fictitious domain discretization, and yields an optimal preconditioner once the background block is solved by
multigrid; the constants depend on the jump in the diffusion coefficient across the interface. The decomposition
analyzed here differs in that the coarse space is a level of the multigrid hierarchy and the local spaces are vertex
patches, so the extension estimate~\eqref{eq:ghost-extension} is available on every level. The following remark records
the parameter dependence of the resulting stable-decomposition constant.

\begin{remark}[parameter dependence of $C_{\mathrm{sd}}$]\label{rem:csd-parameters}
    % Identify every source of parameter dependence in the decomposition.
    For prescribed choices $\gamma=\gamma(p)$ and $\gamma_D=\gamma_D(p)$, the dependence of $C_{\mathrm{sd}}(p)$ enters
    through four sources: the quasi-interpolation constant $C_Q(p,\gamma)$, including the factor
    $(p+1)\Xi_p(1+\gamma^{-1})$ in~\eqref{eq:CQ}; the Markov factor $M_p=dC_M^2p^4$ of~\eqref{eq:local-sum}; the comparison
    of the ghost penalty with the $H^1$-seminorm through $\gamma\Theta'_p$; and the norm-equivalence factor $C_A/c_A$,
    including the admissible Nitsche parameter $\gamma_D$. The Gauss--Lobatto splitting constant
    $C_{\mathrm{GL}}(p,d)\le3^d$ and the cut-cell trace constant~\eqref{eq:cut-trace} are degree-uniform. The factor
    $\gamma\Theta'_p$ is superexponential in $p$ and linear in the ghost weight, which motivates the shedding studied in
    Section~\ref{sec:cut-adaptive}.
\end{remark}

\begin{remark}[status of the hypotheses]\label{rem:gaps}
    % Certify that no hypothesis of the theory is left unverified.
    Lemma~\ref{lem:coverimpl} verifies the covering assumption, Proposition~\ref{prop:quasi} establishes the
    quasi-interpolation properties, and Proposition~\ref{prop:cutfem} verifies Assumption~\ref{ass:cutfem} under the
    standard ghost-weight scaling. The constants in the trace, Markov, and Poincar\'e inequalities remain unspecified.
\end{remark}

%%%%%%%%%%%%%%%%%%%%%%%%%%%%%%%%%%%%%%%%%%%%%%%%%%%%%%%%%%%%
\subsection{The two-level convergence theorem}
\label{sec:twolevel}
%%%%%%%%%%%%%%%%%%%%%%%%%%%%%%%%%%%%%%%%%%%%%%%%%%%%%%%%%%%%

We now assemble the ingredients into the convergence statement. Throughout, $P_j$ denotes the $A$-orthogonal projection
onto $V_{\ell,j}$, $j=0,\dots,J$ (with $V_{\ell,0}$ the coarse space): $P_j v\in V_{\ell,j}$ and $A_\ell(P_j
    v,w)=A_\ell(v,w)$ for all $w\in V_{\ell,j}$, so the exact local (respectively coarse) solve has error propagation
$I-P_j$. Besides the stable decomposition~\eqref{eq:stable-split}, the abstract Schwarz framework needs a bound on the
interactions between the local spaces, which for vertex patches is a counting exercise; the only CutFEM-specific point
is that the ghost penalty couples patches that share no degree of freedom, which enlarges the count but keeps it
finite.

\begin{lemma}[bounded interaction]\label{lem:interaction}
    For $j,k\ge1$ let $\varepsilon_{jk}\in[0,1]$ be the smallest constants
    with
    \begin{gather*}
        |A_\ell(v_j,v_k)| \le \varepsilon_{jk}\,
        \energy{v_j}_\ell\, \energy{v_k}_\ell,
        \qquad v_j\in V_{\ell,j},\ v_k\in V_{\ell,k}.
    \end{gather*}
    Then the interaction matrix $\mathcal E=(\varepsilon_{jk})_{j,k\ge1}$
    satisfies
    \begin{gather}\label{eq:rhoE}
        \rho(\mathcal E) \;\le\; N_O(d) \;:=\; 5^d ,
    \end{gather}
    independently of $\ell$, of the cut configuration, and of $p$.
\end{lemma}

\begin{proof}
    % Count all patch pairs coupled by the stabilized form.
    Functions in $V_{\ell,j}$ vanish outside the (closed) patch $\omega_j$, the union of the $\le 2^d$ active cells sharing
    the vertex $j$. The volume, Nitsche and boundary terms of $A_\ell$ vanish on pairs with $\operatorname{int}\omega_j\cap
        \operatorname{int}\omega_k=\emptyset$; the ghost-penalty term $g_\ell(v_j,v_k)$ involves jumps across faces and can in
    addition be nonzero when a cell of $\omega_j$ is face-adjacent to a cell of $\omega_k$. In either case the vertices
    $j,k$ differ by at most $2$ in each coordinate direction, so every row of $\mathcal E$ has at most $5^d$ nonzero
    entries. Since $A_\ell$ is symmetric and coercive, it defines an inner product, whose Cauchy--Schwarz inequality gives
    $\varepsilon_{jk}\le1$. Therefore $\rho(\mathcal E)\le\|\mathcal E\|_\infty\le 5^d$ (cf.\
    \cite[Lemma~2.10]{ToselliWidlund05}).
\end{proof}

\begin{theorem}[two-level convergence]\label{thm:twolevel}
    Let Assumptions~\ref{ass:geometry}, \ref{ass:cutfem}
    and~\ref{ass:covering} hold, let the local and coarse solves be exact,
    and prescribe $\gamma=\gamma(p)$ and an admissible
    $\gamma_D=\gamma_D(p)$. Let $C_{\mathrm{sd}}(p)$ be the corresponding
    constant of Theorem~\ref{thm:additive}. Then, uniformly in $\ell$ and
    in the position of $\Gamma$ relative to the meshes:
    \begin{itemize}
        \item[(i)] \emph{(additive)} the two-level additive Schwarz operator
              $P_{\mathrm{ad}}=\sum_{j=0}^J P_j$ satisfies
              \begin{gather*}
                  C_{\mathrm{sd}}(p)^{-1}\, A_\ell(v,v) \;\le\;
                  A_\ell(P_{\mathrm{ad}}v,v) \;\le\;
                  \bigl(N_O(d)+1\bigr)\, A_\ell(v,v),
                  \qquad
                  \kappa(P_{\mathrm{ad}}) \le \bigl(N_O(d)+1\bigr)\,
                  C_{\mathrm{sd}}(p);
              \end{gather*}
        \item[(ii)] \emph{(multiplicative)} the two-level error operator of
              the method,
              \begin{gather*}
                  E_{\mathrm{TG}} \;:=\;
                  (I-P_J)\cdots(I-P_1)(I-P_0),
              \end{gather*}
              (coarse correction first, one multiplicative sweep of exact
              patch solves in any fixed order; the reversed and the
              symmetrized $V(1,1)$ orderings obey the same bound) satisfies
              \begin{gather}\label{eq:twolevel}
                  \energy{E_{\mathrm{TG}}}_\ell^2
                  \;\le\; 1 - \frac{1}{\bigl(2\,N_O(d)^2+1\bigr)\,
                      C_{\mathrm{sd}}(p)} \;<\;1 ;
              \end{gather}
        \item[(iii)] \emph{(rediscretized coarse solve)} since the forms are
              non-inherited, a multilevel cycle solves the coarse problem
              with $A_{\ell-1}$ rather than with the Galerkin restriction of
              $A_\ell$. Let the coarse correction be the damped
              rediscretized solve $T_0 := C_{\mathrm{pr}}^{-2}\,
                  \prol{\ell}\, \mathcal R_\ell$, where $\mathcal R_\ell u \in
                  V_{\ell-1}$ is defined by
              \begin{gather*}
                  A_{\ell-1}(\mathcal R_\ell u, z)
                  \;=\; A_\ell(u, \prol{\ell} z)
                  \qquad \forall z\in V_{\ell-1},
              \end{gather*}
              and $C_{\mathrm{pr}}$ is the constant
              of~\eqref{eq:prolongation}. Then
              $\widetilde E_{\mathrm{TG}} :=
                  (I-P_J)\cdots(I-P_1)(I-T_0)$ satisfies
              \begin{gather}\label{eq:twolevel-redisc}
                  \energy{\widetilde E_{\mathrm{TG}}}_\ell^2
                  \;\le\; \tilde\eta^2 \;:=\;
                  1 - \frac{1}{\bigl(2\,N_O(d)^2+1\bigr)\,
                      \bigl(1+C_{\mathrm{pr}}^2\bigr)\,
                      C_{\mathrm{sd}}(p)} \;<\;1 .
              \end{gather}
    \end{itemize}
    In particular the stationary two-level iteration converges, and the
    symmetrized method is a preconditioner with cut-independent condition
    number, for every polynomial degree $p$ and every cut configuration
    admitted by Assumption~\ref{ass:geometry}.
\end{theorem}

\begin{proof}
    This is the abstract Schwarz theory of
    \cite[Sect.~2.3]{ToselliWidlund05} with the two ingredients proved
    above; exact local and coarse solvers give the local-stability
    constant $\omega=1$ of \cite[Assumption~2.4]{ToselliWidlund05}. The
    lower bound in (i) is the Lions lemma
    (\cite[Lemma~2.5]{ToselliWidlund05}, \cite{Xu92}) with the stable
    decomposition~\eqref{eq:stable-split}, $C_0^2 = C_{\mathrm{sd}}(p)$.
    For the upper bound note that the strengthened Cauchy--Schwarz
    assumption \cite[Assumption~2.3]{ToselliWidlund05} does not involve
    the coarse space: by \cite[Lemma~2.6 and
        Theorem~2.7]{ToselliWidlund05}, $\kappa(P_{\mathrm{ad}}) \le
        C_0^2\,\omega\,(\rho(\mathcal E)+1)$ with $\rho(\mathcal E)\le N_O(d)$
    from~\eqref{eq:rhoE}, the coarse space contributing the single unit
    through $\|P_0\|_A\le\omega$. Part (ii) is
    \cite[Theorem~2.9]{ToselliWidlund05},
    \begin{gather*}
        \energy{E_{\mathrm{TG}}}_\ell^2
        \;\le\; 1 - \frac{2-\omega}
        {\bigl(2\widehat\omega^2\rho(\mathcal E)^2+1\bigr)\,C_0^2},
        \qquad \widehat\omega=\max(1,\omega),
    \end{gather*}
    with $\omega=\widehat\omega=1$; its error operator
    \cite[eq.~(2.10)]{ToselliWidlund05} carries the coarse factor first,
    as in the definition of $E_{\mathrm{TG}}$, and the reversed ordering
    is the $A$-adjoint, of equal norm. Prolongation stability
    (Lemma~\ref{lem:prolongation}) enters through
    Theorem~\ref{thm:additive}, whose coarse component is constructed on
    level $\ell-1$ and measured on level $\ell$.

    % The rediscretized coarse solve as an inexact solver.
    Part (iii) is the same theorem with an inexact coarse solver \cite[Assumption~2.4]{ToselliWidlund05}: the approximate
    coarse form is $\tilde A_0 := C_{\mathrm{pr}}^{2} A_{\ell-1}$, whose solve is $T_0$. Because the forms are
    non-inherited, we write out the $A_\ell$-self-adjointness of $T_0$. For $u,v\in V_\ell$, the symmetry of $A_\ell$
    followed by the definition of $\mathcal R_\ell$, applied to $v$ with the test function $z=\mathcal R_\ell u\in
        V_{\ell-1}$, gives
    \begin{gather*}
        A_\ell(T_0u,v)
        = C_{\mathrm{pr}}^{-2}\, A_\ell(\prol{\ell}\mathcal R_\ell u,\, v)
        = C_{\mathrm{pr}}^{-2}\, A_\ell(v,\, \prol{\ell}\mathcal R_\ell u)
        = C_{\mathrm{pr}}^{-2}\, A_{\ell-1}(\mathcal R_\ell v,\, \mathcal R_\ell u),
    \end{gather*}
    and the last expression is symmetric in $u$ and $v$ by the symmetry of $A_{\ell-1}$, so $T_0$ is
    $A_\ell$-self-adjoint; taking $v=u$ shows that it is positive semidefinite. Its stability constant is one:
    $\energy{\prol{\ell}z}_\ell^2 \le C_{\mathrm{pr}}^2\energy{z}_{\ell-1}^2 = \tilde A_0(z,z)$ is
    exactly~\eqref{eq:prolongation}, so $\omega=\widehat\omega=1$ as before. The stable decomposition measured in the
    solver forms follows from~\eqref{eq:stable-split} and~\eqref{eq:stable-split-coarse}: dropping the fine-energy coarse
    term from~\eqref{eq:stable-split},
    \begin{gather*}
        \sum_{j=1}^{J}\energy{v_j}_\ell^2
        + \tilde A_0(\widehat v_0,\widehat v_0)
        \;\le\; \bigl(1+C_{\mathrm{pr}}^2\bigr)\,
        C_{\mathrm{sd}}\,\energy{v}_\ell^2 ,
    \end{gather*}
    so one may take $C_0^2=(1+C_{\mathrm{pr}}^2)\,C_{\mathrm{sd}}(p)$ and
    \cite[Theorem~2.9]{ToselliWidlund05} gives~\eqref{eq:twolevel-redisc}.
\end{proof}

The scope and limitations of the theorem are as follows.
\begin{enumerate}[label=(\alph*)]

    % Scope of the result and the tracked p-dependence.
    \item To our knowledge, this is the first convergence proof for a vertex-patch multigrid smoother for CutFEM with ghost
          penalties. The rate~\eqref{eq:twolevel} is uniform in the mesh size and the cut configuration, while its tracked
          $p$-dependence is superexponential and dominated by $\Theta'_p$; see Section~\ref{sec:p-dependence}.

          % Two-level statement; W-cycle conditional, V-cycle intentionally tested but left open.
    \item The result is two-level. A uniform W-cycle bound follows under an additional smallness assumption on the two-level rate
          (Corollary~\ref{cor:wcycle}), whereas a V-cycle analysis for the non-inherited forms requires BPX-type
          machinery~\cite{BramblePasciakXu91} and remains open. The numerical experiments intentionally test this stronger
          V-cycle and should therefore be read as evidence beyond the scope of the theorem.

          % The implemented sweep order is covered, but the bound is n_c-independent.
    \item Repeating the cut-patch sweep $n_c\ge1$ times preserves~\eqref{eq:twolevel}, because each additional exact subspace
          solve is $A$-non-expansive~\cite[Lemma~2.14]{ToselliWidlund05}. The estimate is independent of $n_c$ and therefore does
          not capture the improvement quantified by the strip contraction of Proposition~\ref{prop:sigma}.
\end{enumerate}

% Motivate the W-cycle consequence of the two-level estimate.
The classical perturbation argument~\cite{Hackbusch85} transfers the two-level rate to a W-cycle under a smallness
assumption. Because the forms are non-inherited, the reference method is the rediscretized variant of
Theorem~\ref{thm:twolevel}(iii).

% Define the W-cycle operators and their contraction numbers.
Let $E^{\mathrm{W}}_\ell$ denote the error propagation operator of the level-$\ell$ W-cycle, consisting of one
multiplicative smoothing sweep and the damped rediscretized coarse correction of Theorem~\ref{thm:twolevel}(iii), with
the coarse problem approximated by two recursive applications of the level-$(\ell-1)$ cycle. Set $E^{\mathrm{W}}_1=0$,
corresponding to an exact solve on the coarsest level, and introduce the contraction numbers
\begin{gather*}
    q_\ell:=\energy{\widetilde E_{\mathrm{TG},\ell}}_\ell,
    \qquad q:=\sup_{\ell\ge2}q_\ell,
    \qquad \varrho_\ell:=\energy{E^{\mathrm{W}}_\ell}_\ell.
\end{gather*}
Whenever $q\le\tfrac14$, we also set
\begin{gather*}
    \varrho_\star:=\tfrac12\bigl(1-\sqrt{1-4q}\bigr).
\end{gather*}

% State the W-cycle convergence result.
\begin{corollary}[W-cycle convergence]\label{cor:wcycle}
    Under the assumptions of Theorem~\ref{thm:twolevel}, the theoretical
    bound~\eqref{eq:twolevel-redisc} gives $q\le\tilde\eta<1$. Moreover, the
    W-cycle contraction numbers obey the recursion
    \begin{gather}\label{eq:wcycle-recursion}
        \varrho_\ell \;\le\; q \;+\; \varrho_{\ell-1}^{\,2}.
    \end{gather}
    Assume in addition that
    \begin{gather}\label{eq:wcycle-assumption}
        q \;\le\; \tfrac14 .
    \end{gather}
    Then
    \begin{gather}\label{eq:wcycle-bound}
        \varrho_\ell \;\le\; \varrho_\star
        \;\le\; 2q \;<\; 1
        \qquad\text{for all }\ell,
    \end{gather}
    a contraction independent of the number of levels, of the mesh size, and
    of the position of $\Gamma$.
\end{corollary}

\begin{proof}
    % Perturbation identity with matched coarse operators.
    Write $S_\ell:=(I-P_J)\cdots(I-P_1)$ for the smoothing sweep, so that $\widetilde E_{\mathrm{TG},\ell}=S_\ell(I-T_0)$
    with $T_0=C_{\mathrm{pr}}^{-2}\prol{\ell}\mathcal R_\ell$ as in Theorem~\ref{thm:twolevel}(iii). Started from a zero
    initial guess, two recursive cycles return $\bigl(I-(E^{\mathrm{W}}_{\ell-1})^{2}\bigr)\mathcal R_\ell u$ in place of
    the exact coarse solution $\mathcal R_\ell u$; the recursion thus approximates the \emph{same} coarse operator the
    correction inverts, which is the point of part (iii). Subtracting,
    \begin{gather*}
        E^{\mathrm{W}}_\ell
        \;=\; \widetilde E_{\mathrm{TG},\ell}
        \;+\; C_{\mathrm{pr}}^{-2}\, S_\ell\, \prol{\ell}\,
        (E^{\mathrm{W}}_{\ell-1})^{2}\, \mathcal R_\ell ,
    \end{gather*}
    Hackbusch's perturbation identity \cite[Sect.~7]{Hackbusch85}.

    % The damping cancels the two transfer factors.
    The perturbation term has $A_\ell$-norm at most $\varrho_{\ell-1}^{2}$: the sweep is $A$-non-expansive,
    $\energy{S_\ell}_\ell\le1$; the prolongation costs~\eqref{eq:prolongation}, $\energy{\prol{\ell}z}_\ell \le
        C_{\mathrm{pr}}\energy{z}_{\ell-1}$; and the residual transfer costs the same factor by duality,
    \begin{gather*}
        \energy{\mathcal R_\ell u}_{\ell-1}
        \;=\; \sup_{z\in V_{\ell-1}}
        \frac{A_{\ell-1}(\mathcal R_\ell u, z)}{\energy{z}_{\ell-1}}
        \;=\; \sup_{z\in V_{\ell-1}}
        \frac{A_\ell(u, \prol{\ell} z)}{\energy{z}_{\ell-1}}
        \;\le\; C_{\mathrm{pr}}\, \energy{u}_\ell ,
    \end{gather*}
    so the damping $C_{\mathrm{pr}}^{-2}$ cancels the two transfer factors
    exactly. Taking $A_\ell$-norms gives
    $\varrho_\ell\le q_\ell+\varrho_{\ell-1}^2\le
        q+\varrho_{\ell-1}^2$, which is~\eqref{eq:wcycle-recursion}.

    % Fixed-point argument for the scalar recursion.
    The scalar recursion $x_\ell\le q+x_{\ell-1}^{2}$ with $x_1=0$ is monotone increasing and bounded above by the smaller
    fixed point of $x=q+x^{2}$, which is real precisely when~\eqref{eq:wcycle-assumption} holds; that root
    is~\eqref{eq:wcycle-bound}, and $\varrho_\star\le2q$ follows from $1-\sqrt{1-4q}\le4q$.
\end{proof}

\begin{remark}[on the smallness assumption]\label{rem:wcycle}
    % What the theory guarantees and what it does not.
    Theorem~\ref{thm:twolevel}(iii) guarantees $q\le\tilde\eta<1$ uniformly in the mesh size and the cut, but the
    quantitative bound $\tilde\eta$ does not imply \eqref{eq:wcycle-assumption}: the interaction count $2N_O(d)^2+1$ alone
    places the proven $\tilde\eta$ within a fraction of a percent of one. The assumption is therefore a statement about the
    actual uniform two-level rate $q$, not a consequence of the constants above. Additional smoothing sweeps do not close
    this gap within the abstract framework: $m$ sweeps yield $S_\ell^{\,m}(I-T_0)$, for which $A$-non-expansiveness
    \cite[Lemma~2.14]{ToselliWidlund05} gives again the bound $q$, not $q^{m}$ --- the latter belongs to $m$ repetitions of
    the entire two-level iteration, coarse solve included. The damping $C_{\mathrm{pr}}^{-2}$ of the coarse correction is
    likewise a device of the analysis; in practice the undamped correction is used.
\end{remark}

% Explain the two coloring variants.
Parallel execution colors patches and processes the color classes multiplicatively
\cite[Sect.~2.5.1]{ToselliWidlund05}. The standard $2^d$ geometric coloring is $A_\ell$-orthogonal on fitted meshes,
but ghost faces can couple same-color patches at vertex distance two (Lemma~\ref{lem:interaction}). One may either use
stride-$3$ colors on the ghost-adjacent strip to restore exact orthogonality or retain the standard colors and apply
additive updates within each color, yielding the practical \emph{semi-multiplicative} smoother. We state the exactly
orthogonal variant first.

% Define the coloring and sweep used for the exactly orthogonal variant.
For this variant, the patches whose cells touch a ghost face, namely the cut and collar patches, are colored so that
same-color vertices differ by at least three cells in some coordinate direction; $3^d$ colors suffice. The remaining
interior patches receive the standard $2^d$ geometric colors, with a palette disjoint from that of the cut and collar
patches. The resulting $N^c\le 2^d+3^d$ color classes are processed sequentially after the exact $A_\ell$-orthogonal
coarse correction $P_0$ from Theorem~\ref{thm:twolevel}(ii), while all patch solves within one color are performed
simultaneously. Each solve uses the exact residual and extracted local matrix; strip patches therefore require a
one-cell-wider residual halo. We denote the error-propagation operator of this colored sweep by $E_{\mathrm{col}}$.

\begin{corollary}[parallel execution by coloring]\label{cor:colored}
    % State the convergence estimate for the colored sweep.
    The colored sweep defined above satisfies
    \begin{gather}\label{eq:colored}
        \energy{E_{\mathrm{col}}}_\ell^2
        \;\le\; 1 - \frac{1}{\bigl(2\,(2^d+3^d)^2+1\bigr)\,
            C_{\mathrm{sd}}(p)} \;<\;1,
    \end{gather}
    uniformly in $\ell$ and in the cut configuration.
\end{corollary}

\begin{proof}
    % Establish orthogonality within each of the disjoint color palettes.
    Same-color interior patches have supports with disjoint interiors, so the volume and Nitsche terms of $A_\ell$ vanish
    on such pairs. Moreover, no cell of an interior patch touches a ghost face, by the definition of the strip palette, so
    $g_\ell$ vanishes on such pairs as well. Same-color strip patches at stride $3$ are separated by a full cell layer and
    thus have disjoint supports with no face-adjacent cells. Patches from the two classes never share a color because their
    palettes are disjoint. Hence, within each color, the restriction of $A_\ell$ to the direct sum is block diagonal, the
    exact solve on the merged subspace is the simultaneous patch solves, and the colored sweep is the multiplicative method
    over the merged subspaces.

    % Regroup the stable decomposition and apply the Schwarz estimate.
    For any decomposition in~\eqref{eq:stable-split}, orthogonality within each color gives
    \begin{gather*}
        \energy{v_0}_\ell^2
        + \sum_c \energy{\sum_{j\in c}v_j}_\ell^2
        = \sum_{j=0}^J \energy{v_j}_\ell^2.
    \end{gather*}
    Thus the stable decomposition holds for the merged family with the same
    constant $C_{\mathrm{sd}}(p)$. The interaction matrix of the
    $N^c$ merged subspaces has spectral radius $\le N^c\le 2^d+3^d$
    \cite[Lemma~2.10]{ToselliWidlund05}. Now apply
    \cite[Theorem~2.9]{ToselliWidlund05} as in
    Theorem~\ref{thm:twolevel}(ii).
\end{proof}

% Define the semi-multiplicative variant and its inexact local form.
For the damped semi-multiplicative variant analyzed here, all patches use the standard $2^d$ colors. After the exact
coarse correction $P_0$, colors are processed multiplicatively and their patches additively; cut and collar updates are
damped by $\theta\in(0,1)$, while the remaining updates are undamped. Since same-color patches share no degrees of
freedom, each additive update is an exact solve for the patch-block-diagonal approximation of $A_\ell$ on the merged
color space, with inexactness measured by the abstract constant $\omega$~\cite[Assumption~2.4]{ToselliWidlund05}. We
denote the resulting error-propagation operator by $E_{\mathrm{semi}}$.

\begin{corollary}[semi-multiplicative sweep]
    \label{cor:hybrid}
    % State the local-stability and convergence estimates.
    Let $c_A$ be the cut-independent coercivity constant in~\eqref{eq:coercive}. The abstract local-stability constant
    satisfies
    \begin{gather}\label{eq:semi-omega}
        \omega \le \overline\omega_\theta
        := \max\bigl\{1,\,(1+c_A^{-1})\theta\bigr\}.
    \end{gather}
    Consequently, for $\theta\le c_A/(1+c_A)$,
    \begin{gather}\label{eq:hybrid}
        \energy{E_{\mathrm{semi}}}_\ell^2
        \;\le\; 1 - \frac{\theta}{\bigl(2\cdot4^d+1\bigr)\,
            C_{\mathrm{sd}}(p)},
    \end{gather}
    uniformly in $\ell$ and the cut configuration; the largest
    admissible damping $\theta=c_A/(1+c_A)$ gives the best bound.
\end{corollary}

\begin{proof}
    % Interpret the additive color update as an exact solve with an inexact form.
    Within one color the patches are pairwise dof-disjoint (two vertices of a common cell differ by $1$ in some coordinate,
    so they never share a color at stride $2$), hence the color space is the direct sum $V_c=\bigoplus_{j\in c} V_{\ell,j}$
    and the damped additive update is the exact solve with $\tilde a_c(u,u) := \sum_{j\in c}
        \theta_j^{-1}\energy{u_j}_\ell^2$, $\theta_j\in\{\theta,1\}$. For Assumption~2.4 of \cite{ToselliWidlund05} we must
    bound $\energy{\sum_{j\in c} u_j}_\ell^2 \le \omega\,\tilde a_c(u,u)$. Same-color patches have disjoint open supports,
    so the volume, Nitsche and boundary terms are additive and the cross terms are pure ghost penalty. Within a color the
    cells of distinct patches are disjoint, so each ghost face couples at most one pair of patches, and face-wise
    Cauchy--Schwarz with Young's inequality gives
    \begin{gather*}
        \sum_{\substack{i,j\in c\\ i\ne j}} g_\ell(u_i,u_j)
        \;\le\; \sum_{\substack{j\in c\\ \mathrm{ghost\text{-}adj}}}
        g_\ell(u_j,u_j)
        \;\le\; c_A^{-1}
        \sum_{\substack{j\in c\\ \mathrm{ghost\text{-}adj}}}
        \energy{u_j}_\ell^2,
    \end{gather*}
    where the last step follows from $g_\ell(u_j,u_j)\le N_\ell(u_j)^2
        \le c_A^{-1}\energy{u_j}_\ell^2$, and the restricted sums reflect that
    the cross terms are nonzero only for ghost-adjacent patches, where
    $\theta_j=\theta$. Therefore,
    \begin{gather*}
        \energy{\sum_{j\in c} u_j}_\ell^2
        \le \sum_{j\in c}\energy{u_j}_\ell^2
        +c_A^{-1}\sum_{\substack{j\in c\\ \mathrm{ghost\text{-}adj}}}
        \energy{u_j}_\ell^2
        \le \overline\omega_\theta\,\tilde a_c(u,u),
    \end{gather*}
    which proves~\eqref{eq:semi-omega}. The
    merged subspaces number $2^d$, so $\rho(\mathcal E)\le 2^d$
    \cite[Lemma~2.10]{ToselliWidlund05}. As in
    Theorem~\ref{thm:twolevel}(iii), the stable decomposition must be
    measured in the solver forms: regrouping~\eqref{eq:stable-split} by
    color,
    \begin{gather*}
        \energy{v_0}_\ell^2 + \sum_c \tilde a_c(v_c,v_c)
        = \energy{v_0}_\ell^2
        + \sum_{j=1}^{J} \theta_j^{-1}\energy{v_j}_\ell^2
        \le \theta^{-1} C_{\mathrm{sd}}(p)\,\energy{v}_\ell^2,
    \end{gather*}
    Thus one may take $C_0^2 = \theta^{-1}C_{\mathrm{sd}}(p)$; the damping
    required for local stability introduces the factor $\theta^{-1}$ in the
    decomposition constant.
    \cite[Theorem~2.9]{ToselliWidlund05} with
    $\overline\omega_\theta=\widehat\omega=1$ for
    $\theta\le c_A/(1+c_A)$ then gives~\eqref{eq:hybrid}.
\end{proof}

\begin{remark}[elliptic ghost-free part]\label{rem:psd-damping}
    % Recover a coercivity-free damping when the ghost-free form is elliptic.
    The dependence of the damping on the coercivity constant is the price of small cut cells. If $\gamma_D$ is large enough
    that the ghost-free part of $A_\ell$ is elliptic on the given configuration, then $g_\ell(v,v)\le\energy{v}_\ell^2$
    with constant one, the argument above gives $\omega\le\max\{1,2\theta\}$, and $\theta=\tfrac12$
    yields~\eqref{eq:hybrid} with $\theta=\tfrac12$: the coercivity constant drops out. Such a $\gamma_D$ is not
    cut-uniform, since it grows without bound as cut cells degenerate, which is why the corollary argues through the global
    coercivity~\eqref{eq:coercive} instead.
\end{remark}

\begin{remark}[undamped sweep]\label{rem:undamped}
    % Explain why the cut-uniform estimate does not cover the undamped sweep.
    At $\theta=1$, estimate~\eqref{eq:semi-omega} gives $\overline\omega_1=1+c_A^{-1}$, which is not guaranteed to be below
    the threshold $2$ in \cite[Theorem~2.9]{ToselliWidlund05}. Thus the estimate does not prove convergence of the undamped
    semi-multiplicative sweep, but it does not prove divergence either. An undamped result would require a sharper
    cut-uniform spectral bound on the within-color ghost coupling; establishing or disproving such a bound lies beyond the
    stable-decomposition theory used here. All semi-multiplicative experiments in this paper use this undamped choice
    $\theta=1$ and are reported as numerical evidence rather than as instances of Corollary~\ref{cor:hybrid}.
\end{remark}

% Illustrate the level independence of the two smoothing sweeps.
\paragraph{Numerical verification.}
Table~\ref{tab:level-convergence} reports both sweeps for the settings of Section~\ref{sec:smoother}. At fixed degree,
the counts are nearly level-independent through $\Q_3$. The finest two-dimensional results drift at $\Q_4$ and $\Q_5$,
whereas the available three-dimensional results are flat but cover fewer levels. The two sweeps perform similarly,
while the increase in iteration counts with $p$, particularly in three dimensions, motivates the study in
Section~\ref{sec:p-dependence}.

% Report the GMRES iteration counts across levels for both sweep variants in two and three dimensions.
%% Data: numerical/calc/sweep_level_convergence.sh and
%% numerical/calc/sweep_level_convergence_3d.sh. Both use gamma = 0.1,
%% n_c = 2, full-residual patches, tensor-product ghost penalty, and
%% gamma_D = 5p(p+1). Missing entries exceed memory because patch inverses
%% are stored densely. In 3D, the exact solution ranges from 1 to 5/3.
\begin{table}[htbp]
    \centering
    \caption{GMRES iteration counts for the $d$-dimensional ball, across mesh
        levels and polynomial degrees, at $\gamma=0.1$ with $n_c=2$ passes over
        the cut patches, the full-residual patch spaces, and symmetric Nitsche at
        $\gamma_D=5(p+1)p$ and unit relaxation. The left subtable uses the
        semi-multiplicative sweep, the right one the multiplicative sweep.
        Entries marked \texttt{---} are beyond the memory of the test machine.}
    \label{tab:level-convergence}
    \begin{subtable}[t]{0.49\textwidth}
        \centering
        \caption{Semi-multiplicative sweep.}
        \setlength{\tabcolsep}{4pt}
        \begin{tabular}{clccccc}
            \toprule
                                  & $L$ & $\Q_1$ & $\Q_2$ & $\Q_3$ & $\Q_4$ & $\Q_5$ \\
            \midrule
            \multirow{5}{*}{$2D$} & $4$ & $6$    & $5$    & $6$    & $13$   & $26$   \\
                                  & $5$ & $6$    & $5$    & $7$    & $15$   & $32$   \\
                                  & $6$ & $6$    & $4$    & $6$    & $13$   & $27$   \\
                                  & $7$ & $6$    & $5$    & $8$    & $15$   & $35$   \\
                                  & $8$ & $6$    & $5$    & $9$    & $21$   & $50$   \\
            \midrule
            \multirow{3}{*}{$3D$} & $3$ & $6$    & $7$    & $12$   & $26$   & $57$   \\
                                  & $4$ & $6$    & $7$    & $12$   & $26$   & ---    \\
                                  & $5$ & $6$    & $7$    & $13$   & ---    & ---    \\
            \bottomrule
        \end{tabular}
    \end{subtable}%
    \hfill
    \begin{subtable}[t]{0.49\textwidth}
        \centering
        \caption{Multiplicative sweep.}
        \setlength{\tabcolsep}{4pt}
        \begin{tabular}{clccccc}
            \toprule
                                  & $L$ & $\Q_1$ & $\Q_2$ & $\Q_3$ & $\Q_4$ & $\Q_5$ \\
            \midrule
            \multirow{5}{*}{$2D$} & $4$ & $6$    & $4$    & $7$    & $16$   & $30$   \\
                                  & $5$ & $6$    & $5$    & $7$    & $15$   & $30$   \\
                                  & $6$ & $6$    & $4$    & $6$    & $12$   & $26$   \\
                                  & $7$ & $6$    & $5$    & $7$    & $15$   & $33$   \\
                                  & $8$ & $6$    & $5$    & $9$    & $19$   & $45$   \\
            \midrule
            \multirow{3}{*}{$3D$} & $3$ & $6$    & $6$    & $11$   & $25$   & $55$   \\
                                  & $4$ & $6$    & $7$    & $12$   & $25$   & ---    \\
                                  & $5$ & $6$    & $7$    & $13$   & ---    & ---    \\
            \bottomrule
        \end{tabular}
    \end{subtable}
\end{table}
\FloatBarrier

%%%%%%%%%%%%%%%%%%%%%%%%%%%%%%%%%%%%%%%%%%%%%%%%%%%%%%%%%%%%
\section{The $p$-dependence of the method}
\label{sec:p-dependence}
%%%%%%%%%%%%%%%%%%%%%%%%%%%%%%%%%%%%%%%%%%%%%%%%%%%%%%%%%%%%

% Relate the tracked degree dependence to the baseline experiment.
The two-level estimate is uniform in the level and the cut configuration, but its proven dependence on the polynomial
degree is superexponential. For the standard derivative-jump weights with a fixed global multiplier, the bounds
assembled in Theorem~\ref{thm:additive} give
\begin{gather}\label{eq:csd-asymptotic}
    C_{\mathrm{sd}}(p) \;\le\; \exp\bigl(2p\log p+O(p)\bigr),
    \qquad
    1-\energy{E_{\mathrm{TG}}}_\ell^2
    \;\gtrsim\; \exp\bigl(-2p\log p-O(p)\bigr).
\end{gather}
The leading term comes from the ghost-energy comparison constant
$\Theta'_p=(Cp)^{2p}e^{O(p)}$, rather than from the polynomial trace and Markov constants or the exponential transfer
factors $\Xi_p$ and $Z_p$; see Remark~\ref{rem:csd-parameters}. Table~\ref{tab:baseline} tests whether this loss is
visible in the exact sweep.

% Report the baseline degree dependence.
\begin{table}[htbp]
    \centering
    \caption{Iteration counts versus $p$ on the $L=5$ circle hierarchy for the exact multiplicative vertex-patch
        smoother with the classical Dirichlet patch spaces, $n_c=2$, $\gamma=0.1$, and unit relaxation. $>\!200$ denotes failure to
        reach it within $200$ iterations.}
    \label{tab:baseline}
    \begin{tabular}{lccccccc}
        \toprule
        element            & $\Q_1$ & $\Q_2$ & $\Q_3$ & $\Q_4$ & $\Q_5$   & $\Q_6$   & $\Q_7$   \\
        \midrule
        stationary V-cycle & $7$    & $6$    & $18$   & $92$   & $>\!200$ & $>\!200$ & $>\!200$ \\
        GMRES              & $6$    & $5$    & $9$    & $20$   & $39$     & $94$     & $>\!200$ \\
        \bottomrule
    \end{tabular}
\end{table}
% Compare the measurements with the tracked worst-case bound.
$\Q_1$ and $\Q_2$ require at most six GMRES iterations, whereas from $\Q_3$ onward the GMRES count roughly doubles per degree and the
stationary V-cycle reaches the iteration cap at $\Q_5$. Over $p=3,\ldots,6$, the data appear exponential but cannot
distinguish this growth from the onset of~\eqref{eq:csd-asymptotic}. Since fitted vertex-patch Schwarz methods are
$p$-robust~\cite{Pavarino94additive}, we analyze the cut-strip contraction separately and assess whether
$\Theta'_p$ describes it sharply.

\subsection{Contraction of the boundary-strip sweep}

% Define the boundary-strip error operator.
Let
\begin{gather}\label{eq:strip-patches}
    \mathbb J_c := \{\, j :\ \omega_j \text{ contains a cut cell} \,\}
\end{gather}
be the \emph{strip patches}, $V_c^\Sigma := \sum_{j\in\mathbb J_c} V_{\ell,j}$ their
combined span. Fix an ordering $\mathbb J_c=\{j_1,\ldots,j_m\}$ and define
\begin{gather*}
    E_{\mathrm{cut}} := (I-P_{j_m})\cdots(I-P_{j_1})
\end{gather*}
as the error operator of one exact multiplicative sweep in this ordering. Each $P_j$ maps $V_c^\Sigma$ into itself,
so $E_{\mathrm{cut}}$ preserves $V_c^\Sigma$ and fixes its $A$-orthogonal complement. With the ordering and
discretization parameters understood, define
\begin{gather}\label{eq:sigma-def}
    \sigma_\ell(p) := \bigl\|E_{\mathrm{cut}}\big|_{V_c^\Sigma}\bigr\|_A .
\end{gather}

% State the strip-contraction result with its parameter dependence.
\begin{proposition}[strip contraction]\label{prop:sigma}
    Under the assumptions of Theorem~\ref{thm:additive}, for every prescribed
    ghost weight $\gamma=\gamma(p)$ and admissible Nitsche parameter
    $\gamma_D=\gamma_D(p)$ there is
    $\sigma(p,\gamma,\gamma_D)<1$, independent of the level $\ell$, of the
    position of $\Gamma$ relative to the mesh, and of the chosen sweep ordering,
    such that
    \begin{gather*}
        \sigma_\ell(p)\le\sigma(p,\gamma,\gamma_D).
    \end{gather*}
    For the prescribed parameter choices, we abbreviate the right-hand side by
    $\sigma(p)$.
\end{proposition}

% Summarize the proof mechanism, scope, and repeated-sweep consequence.
The proof, given in Appendix~\ref{app:sigma-proof}, combines a Friedrichs inequality on the $O(h_\ell)$ band around
$\Gamma$, which supplies the low-frequency control that the coarse space provides globally, with the stable-splitting
machinery of Theorem~\ref{thm:additive} restricted to the strip patches. The result is cut- and level-uniform for every
fixed degree, but gives neither a sharp rate nor a polynomial lower bound for $1-\sigma(p)$. After $n_c$ cut sweeps,
\begin{gather*}
    \bigl\|E_{\mathrm{cut}}^{\,n_c}\big|_{V_c^\Sigma}\bigr\|_A
    \;\le\; \sigma_\ell^{\,n_c}.
\end{gather*}
Because a product of $A$-orthogonal projections need not be normal, the inequality need not be an identity. Moreover,
complete-solver iteration counts do not isolate $E_{\mathrm{cut}}$, so the effective rates inferred below are
diagnostic estimates rather than bounds for $\sigma_\ell(p)$.

\subsection{Scaling of the strip-sweep count}

% Motivate a diagnostic estimate of the required sweep count.
The derived bound on $\sigma(p)$ is not sharp. For any candidate per-sweep rate $s\in(0,1)$, the number of repetitions
that makes the modeled pre- and post-smoothing factor $s^{2n_c}$ no larger than a reference interior/coarse rate
$\rho_0$ is
\begin{gather}\label{eq:nc-balance}
    n_c^\star(s) \;=\; \Bigl\lceil \frac{\ln(1/\rho_0)}{2\ln(1/s)} \Bigr\rceil .
\end{gather}
We set $\rho_0=10^{-8/6}$ from the low-degree counts and infer $\sigma_{\mathrm{eff}}(p)$ from the $n_c=3$ run of each
degree. At $\Q_6$, for example, this calibration gives $\sigma_{\mathrm{eff}}\approx0.95$ and predicts
$n_c^\star\approx30$. This prediction agrees with the observed count at $\Q_4$ but underestimates the
required work at $\Q_5$--$\Q_7$. Both sweeps behave similarly, with the multiplicative variant slightly slower at high
degree.
%% Both subtables measured at gamma = 0.1, plain Burman (no gamma factors),
%% n levels = 6 (paper L=5), symmetric Nitsche (nitsche sign = 1,
%% nitsche parameter = 5, i.e. gamma_D = 5(p+1)p), full residual patches,
%% tensor-product ghost penalty.
%% Script: numerical/calc/sweep_nc_law.sh
\begin{table}[htbp]
    \centering
    \caption{GMRES iteration counts versus the number $n_c$ of cut-patch sweeps on the $L=5$ circle hierarchy, with
        $\gamma=0.1$, unit relaxation ($\theta=1$ for the semi-multiplicative sweep), the tensor-product ghost penalty with standard weights, full-residual patch spaces, and symmetric
        Nitsche parameter $\gamma_D=5p(p+1)$. Entries $>200$ exceed the iteration cap.}
    \label{tab:nc-law}
    \begin{subtable}[t]{0.49\textwidth}
        \centering
        \caption{Semi-multiplicative sweep.}
        \setlength{\tabcolsep}{3pt}
        \begin{tabular}{lccccccc}
            \toprule
                  & \multicolumn{7}{c}{GMRES iterations at $n_c$}                                                       \\
            \cmidrule(l){2-8}
            $n_c$ & $\Q_1$                                        & $\Q_2$ & $\Q_3$ & $\Q_4$ & $\Q_5$ & $\Q_6$ & $\Q_7$ \\
            \midrule
            $1$   & $7$                                           & $6$    & $10$   & $20$   & $44$   & $97$   & $>200$ \\
            $2$   & $6$                                           & $5$    & $7$    & $15$   & $32$   & $70$   & $164$  \\
            $3$   & $6$                                           & $4$    & $6$    & $12$   & $26$   & $59$   & $136$  \\
            $15$  & $6$                                           & $4$    & $4$    & $5$    & $12$   & $28$   & $66$   \\
            $30$  & $6$                                           & $4$    & $4$    & $4$    & $9$    & $21$   & $48$   \\
            \bottomrule
        \end{tabular}
    \end{subtable}%
    \hfill
    \begin{subtable}[t]{0.49\textwidth}
        \centering
        \caption{Multiplicative sweep.}
        \setlength{\tabcolsep}{3pt}
        \begin{tabular}{lccccccc}
            \toprule
                  & \multicolumn{7}{c}{GMRES iterations at $n_c$}                                                       \\
            \cmidrule(l){2-8}
            $n_c$ & $\Q_1$                                        & $\Q_2$ & $\Q_3$ & $\Q_4$ & $\Q_5$ & $\Q_6$ & $\Q_7$ \\
            \midrule
            $1$   & $7$                                           & $6$    & $10$   & $21$   & $42$   & $99$   & $>200$ \\
            $2$   & $6$                                           & $5$    & $7$    & $15$   & $30$   & $71$   & $181$  \\
            $3$   & $6$                                           & $4$    & $6$    & $12$   & $25$   & $58$   & $149$  \\
            $15$  & $6$                                           & $4$    & $4$    & $5$    & $11$   & $27$   & $70$   \\
            $30$  & $6$                                           & $4$    & $4$    & $4$    & $8$    & $20$   & $51$   \\
            \bottomrule
        \end{tabular}
    \end{subtable}
\end{table}
% Interpret the sweep-count measurements without identifying them with the operator norm.
Repeated cut sweeps reduce the high-degree deterioration, but the required count grows rapidly with $p$. Their work is
confined to the $O(h^{-(d-1)})$ strip and is therefore asymptotically smaller than the $O(h^{-d})$ interior work. The
data are consistent with Remark~\ref{rem:ghostchain}, but neither identify a superexponential rate nor establish
$\Theta'_p$ as its mechanism. A sharp a priori estimate of the $p$-dependence remains open.

\subsection{Limits of ghost-penalty shedding}\label{sec:cut-adaptive}

% Introduce the two modifications of the ghost weights.
The robustness of vertex-patch Schwarz methods on fitted meshes~\cite{Pavarino94additive} suggests that the observed
degree dependence is associated with the cut-strip terms rather than with the patch decomposition alone. We therefore
examine the ghost weights. They may be varied through either the global multiplier $\gamma\in(0,1]$ in
\eqref{eq:intro-ghost} or positive derivative-order weights $\gamma_k h_F^{2k-1}/(k!)^2$. Remark~\ref{rem:duality}
analyzes the latter choice; the experiments below vary the global multiplier.

% State the trade-off for the global multiplier.
Lowering $\gamma$ improves the smoothing comparison $\gamma\Theta'_p$ in Lemma~\ref{lem:ghostchain}(b), but degrades
the stability bounds through $(p+1)\gamma^{-1}$ in Proposition~\ref{prop:cutfem} and the quasi-interpolation
constant~\eqref{eq:CQ}. The objective is to determine how the smallest stable global weight depends on $p$. Geometric
face weights based on the cut fraction~\cite{FreiKnokeSteinbachWenskeWick24} are not considered here. The opposite end
of the same design variable is treated in~\cite{BurmanHansboLarson26Locking}: with the penalty localized so that
locking is avoided, the weight may be increased without bound, and the limit enforces the algebraic constraints that
characterize the discrete-extension spaces. The question here is how far the weight may be decreased instead.

% Show that redistributing the derivative-order weights cannot improve the assembled bounds.
\begin{remark}[profile-product lower bound]\label{rem:duality}
    For positive derivative-order multipliers $\gamma=(\gamma_1,\dots,\gamma_p)$, set
    \begin{gather}\label{eq:profile-constants}
        S(\gamma) := \sum_{k=1}^p
        \gamma_k\,\frac{(C_M^2p^4)^{k-1}}{(k!)^2},
        \qquad
        B(\gamma) := \sum_{k=1}^p \gamma_k^{-1}.
    \end{gather}
    Up to polynomial factors, $S(\gamma)$ controls the ghost-energy upper bound and $B(\gamma)$ the extension and
    coercivity estimates. Cauchy--Schwarz gives
    \begin{gather}\label{eq:duality}
        S(\gamma)B(\gamma)
        \;\ge\;
        \left(
        \sum_{k=1}^{p}\frac{(C_Mp^2)^{k-1}}{k!}
        \right)^2
        \;\ge\;
        \frac{(C_Mp)^{2p}}{C_M^2p^4}.
    \end{gather}
    The first inequality is sharp for $\gamma_k\propto k!(C_Mp^2)^{-(k-1)}$; the second follows from the term $k=p$
    and $p!\le p^p$. Thus, no positive profile makes both sides polynomial in the present analysis, although this does
    not exclude sharper estimates or a polynomial true strip rate.
\end{remark}

% Summarize the effect of lowering the global ghost weight.
Table~\ref{tab:wmax} shows that decreasing $\gamma$ substantially reduces the iteration counts for $p\ge3$, and the
weights at which low counts first occur decrease as $p$ increases; no comparable trend appears for $\Q_1$ and $\Q_2$.
Since GMRES requires nonsingularity rather than definiteness, these runs do not establish coercivity, and the small
weights lie outside the range certified by Proposition~\ref{prop:cutfem} (Remark~\ref{rem:admissible-gammaD}). The mode
analysis below compares the scale of the observed improvement with the coercivity constraint.
\begin{table}[htbp]
    \centering
    %% Source: numerical/calc/sweep_p_gamma.sh,
    %% results/sweep_p_gamma_20260818_130236.
    \caption{GMRES iterations as the global ghost weight $\gamma$ decreases on the $L=5$ circle hierarchy, using the
        multiplicative smoother with $n_c=1$, unit relaxation, full-residual patch spaces, and symmetric Nitsche parameter
        $\gamma_D=5p(p+1)$. }
    \label{tab:wmax}
    \begin{tabular}{lcccccccccc}
        \toprule
        $\gamma$ & $1$  & $10^{-1}$ & $10^{-2}$ & $10^{-3}$ & $10^{-4}$ & $10^{-5}$ & $10^{-6}$ & $10^{-7}$ & $10^{-8}$ & $10^{-9}$ \\
        \midrule
        $\Q_1$   & $10$ & $7$       & $17$      & $8$       & $12$      & $8$       & $9$       & $9$       & $9$       & $9$       \\
        $\Q_2$   & $11$ & $6$       & $4$       & $13$      & $9$       & $5$       & $13$      & $7$       & $6$       & $6$       \\
        $\Q_3$   & $19$ & $10$      & $7$       & $5$       & $6$       & $4$       & $4$       & $4$       & $4$       & $4$       \\
        $\Q_4$   & $42$ & $21$      & $13$      & $9$       & $6$       & $4$       & $4$       & $3$       & $3$       & $3$       \\
        $\Q_5$   & $95$ & $42$      & $25$      & $17$      & $12$      & $8$       & $7$       & $5$       & $4$       & $3$       \\
        \bottomrule
    \end{tabular}
\end{table}
% Two single-cell mode calculations set the competing scales.
The degree dependence of the weights at which lower counts first occur in Table~\ref{tab:wmax} is consistent with two
explicit mode calculations in an aligned two-cell model. A degree-$p$ mode sets the scale at which the penalty becomes
spectrally visible, while a piecewise-linear mode sets the nonnegativity scale of the symmetric model form. The two
scales become incompatible as $p$ grows. These calculations are diagnostic: they do not provide a coercivity threshold
for the assembled operator.

\begin{proposition}[incompatible scales in the two-cell model]\label{prop:scales}
    Let $\cell$ be a cell of size $h$ cut by a plane parallel to a face, with inside part a slab of relative thickness
    $\kappa\in(0,1)$ measured from the ghost face $F$ opposite the interface, an uncut neighbour across $F$, and the
    boundary segment $\Gamma\cap\cell$ parallel to $F$. Let $A_\ell^{(2)}$ denote the corresponding two-cell form,
    containing the physical and Nitsche terms on these cells and the ghost penalty on the single face $F$.
    \begin{itemize}
        \item[(i)] (\emph{visibility scale}) For the mode $\xi(x)=c\,x^p\in\Q_p(\cell)$, with $x$ the distance to $F$,
              normalized to unit full-cell energy ($c^2=(2p-1)/p^2$; tangential directions integrate out), the physical
              energy is $|\xi|^2_{1,\cell\cap\Omega}=\kappa^{2p-1}$ and the ghost energy is $\gamma c^2$. The two are
              comparable at
              \begin{gather}\label{eq:gamma-visibility}
                  \gamma_{\mathrm{vis}}(p,\kappa) \;\simeq\; \frac{p^2}{2p-1}\,\kappa^{\,2p-1},
              \end{gather}
              the weight below which the penalty is spectrally negligible against the physical energy on this mode and
              above which it dominates it.
        \item[(ii)] (\emph{coercivity obstruction}) For the piecewise $\Q_1$ function $v_\kappa$ that vanishes on the
              neighbour and equals the distance to $F$ on $\cell$,
              \begin{gather}\label{eq:sym-floor-energies}
                  A_\ell^{(2)}(v_\kappa,v_\kappa) \;=\; \bigl(\gamma-\kappa+\gamma_D\kappa^2\bigr)\,h^d,
              \end{gather}
              so, for $\gamma_D>\tfrac12$, nonnegativity of the model form, uniformly over the cut fraction, requires
              \begin{gather}\label{eq:sym-floor}
                  \gamma \;\ge\; \max_{\kappa\in(0,1)} \kappa\,(1-\gamma_D\kappa) \;=\; \frac{1}{4\gamma_D},
              \end{gather}
              attained at the moderate cut fraction $\kappa=1/(2\gamma_D)$, a floor of order $p^{-2}$ at the
              standard Nitsche scaling $\gamma_D\simeq Cp^2$.
        \item[(iii)] (\emph{incompatibility}) At the standard scaling and for any fixed $\kappa\in(0,1)$,
              \begin{gather*}
                  \frac{\gamma_{\mathrm{vis}}(p,\kappa)}{1/(4\gamma_D)}
                  \;\lesssim\; p^{3}\,\kappa^{\,2p-1} \;\longrightarrow\; 0
                  \qquad (p\to\infty):
              \end{gather*}
              no global multiplier reaches the visibility scale of the degree-$p$ mode while retaining cut-uniform
              nonnegativity of the two-cell form.
    \end{itemize}
\end{proposition}

\begin{proof}
    (i) All normal jumps of $\xi$ across $F$ of order below $p$ vanish at $x=0$, and the order-$p$ jump equals $p!\,c$,
    so the $k=p$ term of~\eqref{eq:intro-ghost} gives the ghost energy $\gamma c^2$; the physical energy is
    $\int_0^{\kappa} (c\,p\,x^{p-1})^2\,\mathrm dx = \kappa^{2p-1}$ by the normalization of $c$. Equating the two
    gives~\eqref{eq:gamma-visibility}.

    (ii) Only the first-order normal jump of $v_\kappa$ across $F$ is nonzero, and it equals one, so the $k=1$ term
    of~\eqref{eq:intro-ghost} gives $g_\ell(v_\kappa,v_\kappa)=\gamma\,h_F\,|F|=\gamma h^d$. The gradient is the unit
    normal on $\cell$ and zero on the neighbour, so the physical energy is $|\cell\cap\Omega|=\kappa h^d$. On
    $\Gamma\cap\cell$ one has $\partial_n v_\kappa=1$ and $v_\kappa=\kappa h$, with $|\Gamma\cap\cell|=h^{d-1}$, so the
    Nitsche consistency and penalty terms contribute $-2\kappa\,h^d$ and $\gamma_D\kappa^2h^d$. Every term scales as
    $h^d$, which gives~\eqref{eq:sym-floor-energies}; nonnegativity for every $\kappa$ is~\eqref{eq:sym-floor}, the
    right-hand side being maximal at $\kappa=1/(2\gamma_D)$ with value $1/(4\gamma_D)$; the maximizer lies in $(0,1)$
    exactly when $\gamma_D>\tfrac12$, which every standard scaling satisfies.

    (iii) Substitute $\gamma_D\simeq Cp^2$ into the quotient of~\eqref{eq:gamma-visibility}
    and~\eqref{eq:sym-floor}.
\end{proof}

% Qualify the two-cell nonnegativity floor.
Equation~\eqref{eq:sym-floor} is necessary but not sufficient for cut-uniform nonnegativity of the two-cell form. Its
$p^{-2}$ scale contrasts with the exponential decay of~\eqref{eq:gamma-visibility}; bringing the model threshold down
to the visibility scale would require an exponentially large $\gamma_D$ and poor conditioning.

% State the exact vanishing-cut threshold in the two-cell model.
For the two-cell configuration of Proposition~\ref{prop:scales}, with $a_{\kappa,\gamma}$ the one-dimensional
reduction~\eqref{eq:two-cell-1d} of the two-cell form at cut fraction $\kappa$ and $a_{0,\gamma}$ its vanishing-cut
limit~\eqref{eq:two-cell-limit}, Proposition~\ref{prop:spd-limit} of Appendix~\ref{app:spd-floor} gives the exact
threshold
\begin{gather}\label{eq:spd-floor-exact}
    \gamma_{\mathrm{lim}}(p) \;:=\;
    \min\{\gamma:\ a_{0,\gamma}\succeq0\}
    \;=\; \frac{1}{\gamma_D-p^2},
    \qquad \gamma_D>p^2,
\end{gather}
and no positive weight makes the limit form semidefinite when $\gamma_D\le p^2$. For every
$\gamma<\gamma_{\mathrm{lim}}(p)$ the form is indefinite at all sufficiently small cut fractions, so
$\gamma_{\mathrm{lim}}(p)$ is a necessary cut-uniform lower bound for this model, but its sharpness over all cut
fractions remains open. The sampled finite-$\kappa$
thresholds approach this value as $\kappa$ decreases. The $\Q_1$ witness has the correct rate but underestimates the
limit by $4\gamma_D/(\gamma_D-p^2)$; Table~\ref{tab:spd-floor} evaluates both quantities at $\gamma_D=5p(p+1)$.

\begin{table}[htbp]
    \centering
    %% All entries are exact rationals from \eqref{eq:spd-floor-exact}:
    %%   gamma_lim = 1/(gamma_D - p^2) = 1/(p(4p+5)) at gamma_D = 5p(p+1),
    %%   prop bound = 1/(4 gamma_D) = 1/(20 p(p+1)) per prop:scales(ii),
    %%   ratio = 4 gamma_D/(gamma_D - p^2) = 20(p+1)/(4p+5).
    %% The level rows are floor(p, kappa_min(L), gamma_D) computed at 40 digits;
    %% kappa_min from numerical/scripts/circle_min_cut_fraction.py (unit disc in
    %% [-1.21,1.21]^2): 1.536e-2, 8.244e-4, 6.570e-5 at L = 5, 6, 7.
    %% Driver: numerical/scripts/spd_floor_vs_shed_weight.py, on top of the exact
    %% eigensolve in numerical/scripts/symmetric_nitsche_floor.py.
    %% CAVEAT: one ghost face on the cut cell. A 2D cut cell has up to three more,
    %% each adding ghost energy, so these are UPPER bounds for the assembled
    %% operator and "below the floor" means "not certified definite".
    \caption{Ghost-weight thresholds in the two-cell model at $\gamma_D=5p(p+1)$. The realized rows restrict the cut
        fractions to those occurring on the circle at level $L$; the final row gives the two-cell bound from
        Proposition~\ref{prop:scales}(ii). Bold marks cases in which this bound exceeds the realized threshold.}
    \label{tab:spd-floor}
    \begin{tabular}{lccccc}
        \toprule
                                                           & $\Q_1$            & $\Q_2$            & $\Q_3$                & $\Q_4$                & $\Q_5$                \\
        \midrule
        $\gamma_D=5(p+1)p$                                 & $10$              & $30$              & $60$                  & $100$                 & $150$                 \\
        $\gamma_{\mathrm{lim}}=1/(\gamma_D-p^2)$           & $1/9$             & $1/26$            & $1/51$                & $1/84$                & $1/125$               \\
        \midrule
        realized at $L=5$                                  & $9.2\cdot10^{-2}$ & $2.0\cdot10^{-2}$ & $\bf 1.4\cdot10^{-3}$ & $\bf 5.6\cdot10^{-6}$ & $\bf 1.6\cdot10^{-6}$ \\
        realized at $L=6$                                  & $1.1\cdot10^{-1}$ & $3.7\cdot10^{-2}$ & $1.8\cdot10^{-2}$     & $1.1\cdot10^{-2}$     & $6.9\cdot10^{-3}$     \\
        realized at $L=7$                                  & $1.1\cdot10^{-1}$ & $3.8\cdot10^{-2}$ & $2.0\cdot10^{-2}$     & $1.2\cdot10^{-2}$     & $7.9\cdot10^{-3}$     \\
        \midrule
        bound $1/(4\gamma_D)$, Prop.~\ref{prop:scales}(ii) & $2.5\cdot10^{-2}$ & $8.3\cdot10^{-3}$ & $4.2\cdot10^{-3}$     & $2.5\cdot10^{-3}$     & $1.7\cdot10^{-3}$     \\
        \bottomrule
    \end{tabular}
\end{table}

% State the scope of the model threshold.
The realized threshold depends on the smallest cut fraction and approaches~\eqref{eq:spd-floor-exact} as that fraction
decreases. Several $L=5$ runs in Table~\ref{tab:wmax} nevertheless converge below the corresponding two-cell threshold
and produce solution values consistent with the exact solution. Because additional ghost faces add stabilization, the
one-face model does not certify indefiniteness of the assembled form or explain this GMRES behaviour; the usable weight
range and strip contraction rate remain open.

% Test level independence at the two-cell threshold.
\paragraph{Level independence at the threshold weight.}
Table~\ref{tab:shedding-levels} repeats the level sweep of Table~\ref{tab:level-convergence} with the degree-dependent
weight
\begin{gather}\label{eq:gamma-shed}
    \gamma(p) \;=\; \gamma_{\mathrm{lim}}(p) \;=\; \frac{1}{\gamma_D-p^2},
    \qquad \gamma_D = c\,(p+1)p, \quad c=5,
\end{gather}
which is the vanishing-cut semidefiniteness threshold of the two-cell model in~\eqref{eq:spd-floor-exact} and has
order $p^{-2}$. All other discretization parameters agree with Table~\ref{tab:level-convergence}.

% Report level independence and the remaining qualification.
The counts are nearly level-independent through $\Q_3$. The higher degrees retain some finest-level drift in two
dimensions, whereas the available three-dimensional results remain flat over fewer levels. Relative to the fixed weight
$\gamma=0.1$, the counts decrease, particularly at high degree. Since~\eqref{eq:gamma-shed} is the vanishing-cut
threshold only of the two-cell model, these results provide evidence only for the tested configurations.

% Report the level sweep at the threshold weight for both sweeps.
%% 2D source: numerical/calc/sweep_shedding_levels_spd.sh,
%% results/sweep_shedding_levels_spd_20260818_141112.
%% 3D source: numerical/calc/sweep_shedding_levels_spd_3d.sh,
%% results/sweep_shedding_levels_spd_3d_* into results/shed_spd_3d_matrix.txt.
\begin{table}[htbp]
    \centering
    \caption{GMRES iteration counts on the $d$-dimensional ball with $n_c=2$, unit relaxation
        ($\theta=1$ for the semi-multiplicative sweep), full-residual patch spaces,
        $\gamma_D=5p(p+1)$, and the degree-dependent weight~\eqref{eq:gamma-shed} listed in each column.
        \texttt{---} denotes a case exceeding the available memory.}
    \label{tab:shedding-levels}
    \begin{subtable}[t]{0.49\textwidth}
        \centering
        \caption{Semi-multiplicative sweep.}
        \setlength{\tabcolsep}{4pt}
        \begin{tabular}{clccccc}
            \toprule
                                  & $L$              & $\Q_1$ & $\Q_2$ & $\Q_3$ & $\Q_4$ & $\Q_5$ \\
            \midrule
                                  & $\gamma(p)$
                                  & $\tfrac{1}{9}$
                                  & $\tfrac{1}{26}$
                                  & $\tfrac{1}{51}$
                                  & $\tfrac{1}{84}$
                                  & $\tfrac{1}{125}$                                              \\
            \midrule
            \multirow{5}{*}{$2D$} & $4$              & $6$    & $4$    & $4$    & $8$    & $16$   \\
                                  & $5$              & $6$    & $4$    & $5$    & $9$    & $17$   \\
                                  & $6$              & $6$    & $4$    & $5$    & $8$    & $15$   \\
                                  & $7$              & $6$    & $4$    & $5$    & $10$   & $19$   \\
                                  & $8$              & $6$    & $4$    & $6$    & $12$   & $25$   \\
            \midrule
            \multirow{3}{*}{$3D$} & $3$              & $6$    & $6$    & $9$    & $18$   & $36$   \\
                                  & $4$              & $6$    & $6$    & $9$    & $18$   & ---    \\
                                  & $5$              & $6$    & $6$    & $10$   & ---    & ---    \\
            \bottomrule
        \end{tabular}
    \end{subtable}%
    \hfill
    \begin{subtable}[t]{0.49\textwidth}
        \centering
        \caption{Multiplicative sweep.}
        \setlength{\tabcolsep}{4pt}
        \begin{tabular}{clccccc}
            \toprule
                                  & $L$              & $\Q_1$ & $\Q_2$ & $\Q_3$ & $\Q_4$ & $\Q_5$ \\
            \midrule
                                  & $\gamma(p)$
                                  & $\tfrac{1}{9}$
                                  & $\tfrac{1}{26}$
                                  & $\tfrac{1}{51}$
                                  & $\tfrac{1}{84}$
                                  & $\tfrac{1}{125}$                                              \\
            \midrule
            \multirow{5}{*}{$2D$} & $4$              & $6$    & $4$    & $4$    & $8$    & $15$   \\
                                  & $5$              & $6$    & $4$    & $5$    & $9$    & $17$   \\
                                  & $6$              & $6$    & $4$    & $5$    & $8$    & $15$   \\
                                  & $7$              & $6$    & $4$    & $5$    & $9$    & $17$   \\
                                  & $8$              & $6$    & $4$    & $6$    & $12$   & $24$   \\
            \midrule
            \multirow{3}{*}{$3D$} & $3$              & $6$    & $6$    & $9$    & $17$   & $33$   \\
                                  & $4$              & $6$    & $6$    & $9$    & $17$   & ---    \\
                                  & $5$              & $6$    & $6$    & $10$   & ---    & ---    \\
            \bottomrule
        \end{tabular}
    \end{subtable}
\end{table}
\FloatBarrier

%%%%%%%%%%%%%%%%%%%%%%%%%%%%%%%%%%%%%%%%%%%%%%%%%%%%%%%%%%%%

%%%%%%%%%%%%%%%%%%%%%%%%%%%%%%%%%%%%%%%%%%%%%%%%%%%%%%%%%%%%
\section{Conclusion}
\label{sec:conclusion}
%%%%%%%%%%%%%%%%%%%%%%%%%%%%%%%%%%%%%%%%%%%%%%%%%%%%%%%%%%%%

% Uniform two-level convergence on cut meshes and the intentional multilevel test.
At the two-level scale, geometric multigrid with vertex-patch smoothing extends from fitted to cut meshes with a
convergence rate independent of the mesh size and the interface position. Although the level forms are non-inherited,
the injection prolongation is stable in the stabilized energy norm (Lemma~\ref{lem:prolongation}), the vertex-patch
decomposition splits stably with constants independent of the cut (Theorem~\ref{thm:additive}), and the two-level
method contracts uniformly in the mesh size and the position of the interface (Theorem~\ref{thm:twolevel}). Under the
additional smallness condition $q\le\tfrac14$ on the two-level rate, the same conclusion holds for the W-cycle
(Corollary~\ref{cor:wcycle}). The experiments intentionally test the stronger V-cycle, whose analysis remains open.

% The guarantees cover the parallel smoother used in practice.
The analysis also covers both parallel variants, with damping in the semi-multiplicative case. The ghost penalty
couples same-color patches along the boundary strip, so the standard $2^d$ coloring is not exactly multiplicative. With
cut-independent damping, the resulting semi-multiplicative smoother satisfies a cut-independent convergence bound
(Corollary~\ref{cor:hybrid}). A stride-$3$ coloring on the strip instead gives exactly orthogonal color classes
(Corollary~\ref{cor:colored}). The reported semi-multiplicative experiments use the undamped choice and therefore lie
outside Corollary~\ref{cor:hybrid}.

% Dependence on the polynomial degree.
The degree dependence is governed by the boundary-strip contraction $\sigma_\ell(p)$, bounded away from one uniformly
in the level and the cut (Proposition~\ref{prop:sigma}). Repeating the strip sweep, at a relative cost that vanishes
under refinement, mitigates the deterioration, although the diagnostic count from~\eqref{eq:nc-balance} underestimates
the required number of sweeps above $\Q_4$. No derivative-jump profile makes the present bounds polynomial in $p$
(Remark~\ref{rem:duality}). The two-cell analysis identifies incompatible weight scales, but does not provide a
threshold for the assembled operator.

% Open questions.
Two questions remain open. First, the present estimates do not determine whether the $p$-dependence of the true strip
rate $\sigma(p)$ is polynomial. Their superexponential growth is not a proven property of the method, and a sharper
analysis must account for the coupling accumulated along the strip rather than the worst single patch. Second, the
level-uniform W-cycle result requires a smallness assumption on the two-level rate. A V-cycle bound for the
non-inherited forms requires BPX-type machinery~\cite{BramblePasciakXu91}.

% Declare the assistance used while preparing the manuscript.
\paragraph{Declaration of assistance} The author declares support of two local feline agents (Micro and Conda) running locally, alongside commonly available
language models (Gemini, Claude, ChatGPT) throughout the preparation of this work.

The manuscript is the result of a long iterative process, in which these models contributed to text drafting,
proof-reading, checks of mathematical consistency, and adversarial review of the statements and their proofs. The paper
was audited by the author, who retains full accountability for all scientific content.

\begin{comment}
% The reported anomaly is an implementation defect, not a property of the method.
\paragraph{A remark on the anomaly from \cite{CuiKanschat25CutFEM}} The method analysed above converges in both the stationary and the preconditioned mode. The reference implementation
of~\cite{CuiKanschat25CutFEM} reports an anomaly its authors leave unexplained: used as a stationary iteration the
$\Q_3$ V-cycle diverges when the cut patches are swept once ($n_c=1$), is rescued to a plateau by repeating the cut
sweep ($n_c\ge2$), while the \emph{same} V-cycle preconditions GMRES perfectly well. We were able to reproduce this
behaviour only by emulating a smoother that drops the ghost penalty on the uncut patches adjacent to the cut, from both
the local residual and the local solve, and never revisits them in the $n_c$ repetitions. This leaves a low-rank defect
supported away from the cut, which the extra cut sweeps cannot reach; a stationary iteration then converges at the rate
of the resulting isolated outlier (divergent, then $n_c$-flat), whereas GMRES deflates it at the cost of a few
iterations. The method carries no such defect and exhibits no anomaly. The discrepancy is therefore most likely an
implementation bug rather than a property of the method, and a full account lies outside the scope of the present
paper.
\end{comment}

%% file: splitted/introduction.tex
%!TEX root = ../main.tex

% Unfitted methods: decouple the mesh from the geometry.
Constructing a body-fitted mesh is often the most laborious stage of a finite element simulation, a burden compounded
for moving or evolving geometries by repeated remeshing or mesh deformation. Unfitted finite element methods mitigate
this difficulty: the computational domain $\Omega$ is embedded into a fixed background mesh (in practice a Cartesian
grid) and the boundary $\Gamma=\partial\Omega$ is free to cut through the cells. The geometry enters only through the
intersection of $\Gamma$ with the mesh, so complex and evolving domains are handled on the same grid, and the
technology developed for Cartesian meshes carries over. The cut finite element method
(CutFEM)~\cite{BurmanClausHansboLarsonMassing15} is a realization of this idea.

% Geometric multigrid with vertex-patch smoothing.
Geometric multigrid is a natural route to scalable solvers for CutFEM: the Cartesian background grid provides a
hierarchy down to a single cell without geometry-dependent remeshing. Coarse-grid correction treats smooth error, while
for high-order spaces the established relaxation is the \emph{vertex-patch} smoother~\cite{Pavarino94additive}, which
solves local problems on the $2^d$ cells sharing a vertex. On fitted meshes it is robust in the polynomial degree and
supports efficient parallel, matrix-free realizations~\cite{MeggendorferKanschatKraus24,WitteArndtKanschat21,
    BrubeckFarrell22,WichrowskiMunchKronbichlerKanschat25,CuiGrosseBleyKanschatStrzodka25,CuiKanschat24IP}. The patch
problems themselves need not be factorized: solving them inexactly by a nested $p$-multigrid keeps the smoother
matrix-free and preserves $p$-robustness~\cite{wichrowski2025local}, and the same strategy has been evaluated for the
local saddle-point problems of the Stokes equations~\cite{wichrowski2026StokesPatch}. The recent extension of the
vertex-patch smoother to CutFEM~\cite{CuiKanschat25CutFEM} has demonstrated its effectiveness in computations, but, to our
knowledge, no convergence theory is available for this method. We establish this theory and quantify how ghost-penalty
stabilization governs the convergence of the vertex-patch smoother.

% The price of cutting: small cuts and ghost-penalty stabilization.
The geometric freedom has a price. Since $\Gamma$ need not align with the background mesh, the volume fraction of a
cut cell inside $\Omega$ may become arbitrarily small. Consequently, physical-domain contributions associated with
basis functions supported on such cells may approach zero, degrading discrete coercivity and the conditioning of the
linear system. In particular, a direct application of Nitsche's method with a cut-independent penalty loses uniform
discrete ellipticity as the cut fraction tends to zero: its penalty is scaled with the background-cell diameter $h$,
whereas the interior energy scales with the actual size of the physical intersection. Consequently, the boundary flux
terms cannot be controlled uniformly. The established remedy is the \emph{ghost
    penalty}~\cite{burman2010ghost}: a penalty on jumps of normal derivatives across faces in the boundary zone. By weakly
gluing neighboring elements together~\cite{BadiaNeivaVerdugo22}, it allows poorly cut cells to borrow stability from
neighbors with substantial physical support and restores coercivity and conditioning uniformly in the position of the
boundary relative to the mesh~\cite{BurmanHansbo12,MassingLarsonLogg14,dePrenterVerhooselvanBrummelenLarsonBadia23}.
The solver analysis below therefore incorporates the ghost penalty explicitly.

% Model problem and background-grid hierarchy.
We study geometric multigrid for a CutFEM discretization of the Poisson problem, $-\Delta u=f$ in $\Omega$ with $u=g$
on $\Gamma$, in the CutFEM framework going back to Hansbo and Hansbo~\cite{HansboHansbo02} and in the fictitious
domain form of Burman and Hansbo~\cite{BurmanHansbo12}. The discretization is constructed on a hierarchy
$\mesh_0\sqsubset\dots\sqsubset\mesh_L$ of Cartesian background grids, which can be coarsened down to a single cell. On
each grid $\mesh_\ell$, the active mesh $\mesh_{\ell,\Omega}:=\{\cell\in\mesh_\ell:\cell\cap\Omega\ne\emptyset\}$
consists of the cells that intersect the physical domain. Its union defines the active domain
$\Omega_\ell\supseteq\Omega$, and $V_\ell$ is the $\Q_p$ finite element space on $\mesh_{\ell,\Omega}$.

% Stabilized level operator and Nitsche formulation.
The level operator is induced by the stabilized bilinear form
\begin{gather}\label{eq:intro-form}
    A_\ell(u,v) = a_\ell(u,v) + g_\ell(u,v),
\end{gather}
where $a_\ell$ combines the Poisson term with the Nitsche terms that impose the Dirichlet condition weakly,
\begin{gather}\label{eq:intro-nitsche}
    a_\ell(u,v) = (\nabla u, \nabla v)_\Omega
    - (\partial_n u, v)_\Gamma
    - (u, \partial_n v)_\Gamma
    + \gamma_D\, h_\ell^{-1} (u, v)_\Gamma,
\end{gather}
with a sufficiently large penalty parameter $\gamma_D$.

% High-order ghost-penalty stabilization.
Burman's ghost penalty is given by
\begin{gather}\label{eq:intro-ghost}
    g_\ell(u, v) =
    \gamma \sum_{F\in\mathbb F_G}
    \sum_{k=1}^p \frac{h_F^{2k-1}}{(k!)^2}
    \bigl(\llbracket \partial_n^k u \rrbracket,
    \llbracket \partial_n^k v \rrbracket\bigr)_F,
    \qquad \gamma\in(0,1],
\end{gather}
and acts on the faces $\mathbb F_G$ between cut cells and their active neighbors, that is, on every face of the
active mesh adjacent to at least one cut cell, including faces shared by two cut cells
(Figure~\ref{fig:cutfem-setting}). It penalizes jumps of all normal derivatives up to degree $p$ with a global weight
$\gamma$; the conventional choice is a fixed $O(1)$ constant~\cite{burman2010ghost}. Face-dependent weights have also
been introduced for fluid--structure interaction~\cite{FreiKnokeSteinbachWenskeWick24}, and a general design framework
for the stabilized quantities and the element couplings has been developed
in~\cite{BurmanHansboLarson26Locking}, where the penalty parameter is driven to infinity without locking. We instead
reduce all ghost-penalty terms by a single global factor that may decay polynomially with the finite element degree,
and track the resulting dependence.

% Figure: cut cells, interior cells, ghost faces.
\begin{figure}[!ht]
    \centering
    \begin{tikzpicture}[scale=1.]
        % Draw the Cartesian grid
        \draw[step=1cm, gray, very thin] (-0.2,0) grid (5.2,3.2);

        % Draw the curved boundary
        \node at (-0.35,2.5) {$\Gamma$};
        \draw[line width=2pt, red, dashed] plot [smooth, tension=0.5] coordinates {(-0.2,2.3) (1,2.2) (2,2.5) (3,2.2)
                (4,1.5)
                (5.2,1.5)};

        \node at (2.5,0.7) {$\Omega$};
        \node at (3.2,2.7) {$\mathbb F_G$};

        % Highlight intersected cells
        \foreach \x/\y in {0/2, 1/2, 2/2, 3/1, 3/2, 4/1}
            {
                \fill[blue, opacity=0.3] (\x,\y) rectangle (\x+1,\y+1);
            }
        % Highlight cells inside the domain
        \foreach \x/\y in {0/0, 0/1, 2/1, 1/0, 1/1, 2/0, 3/0, 4/0}
            {
                \fill[green, opacity=0.3] (\x,\y) rectangle (\x+1,\y+1);
            }

        % Highlight the ghost faces
        \foreach \x/\y in {0/2, 1/2, 2/2, 2/1, 3/1, -1/2, 4/1} {
                \draw[blue, line width=1.5pt] (\x+1,\y) -- (\x+1,\y+1);
            }

        \foreach \x/\y in {0/1, 1/1, 2/1, 3/0, 4/0, 3/1} {
                \draw[blue, line width=1.5pt] (\x,\y+1) -- (\x+1,\y+1);
            }

    \end{tikzpicture}
    \caption{The CutFEM setting: the boundary $\Gamma$ (red dashed line) intersects a Cartesian background mesh,
        creating cut cells (blue) and interior cells (green). The ghost penalty acts on the faces $\mathbb F_G$
        (thick blue lines) between cut cells and their active neighbors and provides stability uniformly in the
        position of the boundary relative to the mesh.}
    \label{fig:cutfem-setting}
\end{figure}
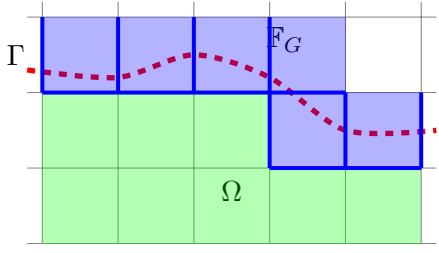

% Multilevel and domain-decomposition methods for unfitted problems.
Multigrid methods for unfitted \emph{interface} problems were developed in~\cite{LudescherGrossReusken20}, with a
prolongation tailored to the unfitted space and an interface-correction smoother whose update solves a global system
on the interface unknowns; the thesis~\cite{Ludescher20} extends this to the Stokes problem and to higher order. Gross
and
Reusken~\cite{GrossReusken23} analyze a two-subspace additive Schwarz preconditioner, splitting the CutFEM space into
the background finite element space and the span of the cut basis functions, and obtain optimal preconditioners for
the interface and the fictitious domain problem. For CutFEM, the cell-based additive Schwarz method
of~\cite{PrenterVerhooselBrummelen19} is cut-independent but $h$-dependent because it does not include a coarse space;
see also the survey~\cite{dePrenterVerhooselvanBrummelenLarsonBadia23}.

% Stabilization results and the remaining solver question.
The ghost-penalty analyses of~\cite{burman2010ghost,BurmanHansbo12,MassingLarsonLogg14} establish the discrete
extension property used below. Face-dependent weights were introduced in~\cite{FreiKnokeSteinbachWenskeWick24} to
improve conditioning, while~\cite{wichrowski2026TensorGhostPenalty} develops a tensor-product evaluation of high-order
ghost penalties. A parameter-free alternative modifies the space instead of the form: the extended finite element
spaces of~\cite{BurmanHansboLarson22Ext} build the extension into the discretization, motivated precisely by the
observation that designing and evaluating ghost penalties becomes costly at high order. We retain the penalty
formulation, since it leaves the finite element space on each background grid unchanged and therefore preserves the
multigrid hierarchy, and we quantify the degree dependence it carries. These results do not provide a convergence theory for the multigrid vertex-patch smoother on CutFEM,
whose reported iteration counts lose $p$-robustness on cut domains~\cite{CuiKanschat25CutFEM}.

% Non-inherited forms and the parallel implementation challenge.
The level forms are \emph{non-inherited}: the ghost penalty acts on a different face set on each level, and the energy
is defined on $\Omega$, whereas the finite element space extends over $\Omega_\ell$. Parallel patch processing
introduces a second issue. A coloring partitions the vertex patches into groups processed concurrently. On a fitted
Cartesian mesh, the standard $2^d$ coloring separates same-color vertices by a stride of two cells and makes their
patch spaces orthogonal in the energy form. With ghost-penalty stabilization, however, same-color patches can remain
coupled across the boundary strip. For two such patches $i\ne j$, the volume and Nitsche integrals vanish, but a pure
ghost-penalty coupling can remain,
\begin{gather}\label{eq:intro-coupling}
    A_\ell(u,v) = g^{\,ij}_\ell(u,v),
    \qquad u\in V_{\ell,i},\; v\in V_{\ell,j}.
\end{gather}
Thus, the standard fitted-mesh coloring does not give exact simultaneous subspace corrections. The additional
boundary-strip coupling is also associated with the observed dependence on the polynomial degree and enters the
convergence analysis.

% This paper: what it does, what is new.
This paper provides a convergence theory, localizes the observed degree dependence to the boundary strip, and treats
the ghost penalty as a design variable. The cut-uniform convergence theory with a multigrid coarse space, its parallel
colored realizations, the $p$-explicit reduction to the strip contraction $\sigma_\ell(p)$, and the two-cell model of
the competing ghost-weight scales appear to be new. The main contributions are:
\begin{itemize}
    \item \emph{Cut-robust convergence.} Although the forms are non-inherited, natural injection prolongation is stable in the stabilized energy
          norm (Lemma~\ref{lem:prolongation}). Using the discrete extension property, the vertex-patch decomposition
          admits a stable splitting independent of how $\Gamma$ cuts the mesh (Theorem~\ref{thm:additive}). The two-level
          contraction is uniform in the mesh size and the cut (Theorem~\ref{thm:twolevel}); under the additional
          smallness condition $q\le\tfrac14$ on the two-level rate, the same holds for the W-cycle
          (Corollary~\ref{cor:wcycle}).
    \item \emph{A parallel semi-multiplicative smoother.} The ghost penalty couples same-color patches on the strip~\eqref{eq:intro-coupling}, so standard $2^d$ coloring loses exact multiplicativity.
          Retaining the cheap $2^d$ coloring while accounting for this coupling defines a
          \emph{semi-multiplicative} smoother (additive only through the thin-strip coupling). For a cut-independent
          damping of the strip updates, we prove a convergence bound that differs from the exact multiplicative estimate
          only by a cut-independent factor (Corollary~\ref{cor:hybrid}). The numerical experiments use the undamped
          variant and are not instances of this corollary. Exact orthogonality is restored by a stride-$3$ coloring on the strip
          (Corollary~\ref{cor:colored}).
    \item \emph{The degree dependence.} The convergence constants degrade with $p$. We isolate this dependence in the contraction $\sigma_\ell(p)$ of the
          exact cut sweep on the boundary strip, showing $\sigma_\ell\le\sigma(p)<1$ uniform in the mesh size and the
          cut (Proposition~\ref{prop:sigma}). A derived balance condition gives the number of cut-sweep repetitions
          $n_c^\star$ required for the modeled strip factor not to exceed a reference interior/coarse
          contraction~\eqref{eq:nc-balance}; we compare this diagnostic with measured sweep counts. Within the
          derivative-jump family, no choice of weights makes the present assembled bounds polynomial in $p$
          (Remark~\ref{rem:duality}).
\end{itemize}

% Scope and remaining analytical questions.
Two questions bound the present theory. Although $\sigma_\ell\le\sigma(p)<1$ is uniform in the level and the cut, its
sharp dependence on $p$ remains open; resolving it requires accounting for how coupling accumulates across the boundary
strip. In addition, the analysis covers the two-level method and, under the smallness condition $q\le\tfrac14$, the
W-cycle. The numerical study intentionally uses the stronger V-cycle, whereas a V-cycle bound for the non-inherited
forms requires BPX-type machinery~\cite{BramblePasciakXu91}.

% Organization.
The remainder of the paper is organized as follows. Section~\ref{sec:theory} fixes the geometric setting and the CutFEM
stability properties the analysis assumes. Section~\ref{sec:transfer} establishes the stability of the intergrid
transfer. Section~\ref{sec:smoother} defines the smoother variants and verifies the vertex-patch covering.
Section~\ref{sec:schwarz} develops the cut-robust two-level Schwarz theory. Section~\ref{sec:p-dependence} treats the
degree dependence, with the sweep-count balance, the weight-profile duality, and the limits of ghost-penalty shedding
(Section~\ref{sec:cut-adaptive}).

%% file: splitted/appendix.tex
\section{Proof of the strip contraction (Proposition~\ref{prop:sigma})}
\label{app:sigma-proof}

% Recall the objects of the proposition and outline the argument.
We use the notation of Section~\ref{sec:p-dependence}: the strip patches $\mathbb J_c$ of~\eqref{eq:strip-patches},
their combined span $V_c^\Sigma$, and the sweep error operator $E_{\mathrm{cut}}$ with norm~\eqref{eq:sigma-def}. The
Friedrichs inequality on the strip band yields a stable decomposition of $V_c^\Sigma$ into the strip patches. The
multiplicative Schwarz estimate then gives the contraction bound.

\begin{proof}[Proof of Proposition~\ref{prop:sigma}]
    % The strip band lies in a thin tube around the interface.
    Write $h:=h_\ell$ and $\Omega_c:=\bigcup_{j\in\mathbb J_c}\omega_j$ for the strip band. By~\eqref{eq:strip-patches}
    every patch of $\mathbb J_c$ contains a cut cell, and a patch has diameter at most $2\sqrt d\,h$, so every point of
    $\Omega_c$ lies within distance $C_0h$ of $\Gamma$ with $C_0:=3\sqrt d$. By \eqref{eq:classical-patch-space}
    and~\eqref{eq:full-residual-patch-space}, every basis function generating $V_{\ell,j}$ has active support contained in
    $\omega_j$. Hence every $v\in V_c^\Sigma$ vanishes on $\Omega_\ell\setminus\Omega_c$, in particular at every point of
    $\Omega_\ell$ at distance greater than $C_0h$ from $\Gamma$. We prove the full-band Friedrichs inequality
    \begin{gather}\label{eq:band-friedrichs}
        \|v\|^2_{0,\Omega_\ell} \;\le\; C_F(p,\gamma)\, h^2\,
        \bigl( |v|^2_{1,\Omega} + g_\ell(v,v) \bigr),
        \qquad v\in V_c^\Sigma,
    \end{gather}
    with $C_F(p,\gamma)$ independent of $\ell$ and of the cut, splitting the left-hand side into the physical part
    $\|v\|^2_{0,\Omega}$ and the fictitious part $\|v\|^2_{0,\Omega_\ell\setminus\Omega}$.

    % Physical part: one-dimensional Friedrichs along the normal fibers of the tube.
    \emph{Physical part.} Set $C_1:=4\sqrt d$, so that $C_1h\le\delta_0/2$ by the resolution
    condition~\eqref{eq:resolution}. By the positive reach of Assumption~\ref{ass:geometry}, the normal map
    $\Phi(y,s):=y+s\,n(y)$ is a $C^1$ diffeomorphism of $\Gamma\times(-C_1h,0)$ onto the interior collar
    $\{x\in\Omega:\operatorname{dist}(x,\Gamma)<C_1h\}$, and since $C_1h\le\delta_0/2$ bounds the principal curvatures
    against the fiber length, its Jacobian is bounded above and below by constants depending only on $d$. The collar
    contains $\Omega\cap\Omega_c$, because $C_0<C_1$. Since $v\in H^1(\Omega)$, its restriction to almost every normal
    fiber is absolutely continuous. The inner endpoint $\Phi(y,-C_1h)$ lies at distance $C_1h>C_0h$ from $\Gamma$,
    where $v=0$. The fundamental theorem of calculus along almost every fiber and Cauchy--Schwarz give, for
    $-C_1h<t<0$,
    \begin{gather*}
        |v(\Phi(y,t))|^2
        \;\le\; C_1h \int_{-C_1h}^{0} |\nabla v(\Phi(y,s))|^2 \,\mathrm ds ,
    \end{gather*}
    and integrating over $t$ and over $y\in\Gamma$, with the Jacobian bounds on both sides, yields
    \begin{gather*}
        \|v\|^2_{0,\Omega} \;=\; \|v\|^2_{0,\Omega\cap\Omega_c}
        \;\le\; C_b\, h^2\, |v|^2_{1,\Omega},
    \end{gather*}
    with $C_b$ independent of $\ell$, $p$ and the cut.

    % Fictitious part: transfer along the ghost chains, paid in ghost energy.
    \emph{Fictitious part.} $\Omega_\ell\setminus\Omega$ is covered by the cut cells. For a cut cell $\cell$, apply the
    one-face transfer~\eqref{eq:transfer} with $q=p$ and $w=v$ (whose face jumps of order $k=0$ vanish by continuity)
    successively along the chain of Lemma~\ref{lem:geom}(i), truncated at the first uncut cell, from the uncut cell
    $I(\cell)\subset\Omega$ back to $\cell$. The truncated chain has at most $C_d$ cells, and each of its faces is
    adjacent to a cut cell and thus belongs to $\mathbb F_G$. Therefore, its jump terms are controlled by $g_\ell$ at
    the price $\gamma^{-1}$. Consequently,
    \begin{gather*}
        \|v\|^2_{0,\cell}
        \;\le\; C\,\Xi_p \bigl( \|v\|^2_{0,I(\cell)}
        + (p+1)\,\gamma^{-1} h^2\, g_\ell^{\mathrm{chain}(\cell)}(v,v) \bigr),
    \end{gather*}
    with $\Xi_p$ the chain amplification~\eqref{eq:Xi} and $g_\ell^{\mathrm{chain}(\cell)}$ the contribution to $g_\ell$
    of the chain faces. Each cell serves in at most $C_d$ chains (Lemma~\ref{lem:geom}(i)), so summing over the cut
    cells gives
    \begin{gather*}
        \|v\|^2_{0,\Omega_\ell\setminus\Omega}
        \;\le\; C\,\Xi_p \bigl( \|v\|^2_{0,\Omega}
        + (p+1)\,\gamma^{-1} h^2\, g_\ell(v,v) \bigr),
    \end{gather*}
    and combining with the physical part proves~\eqref{eq:band-friedrichs} with
    $C_F(p,\gamma)\le C\,\Xi_p\,\bigl(C_b+(p+1)\,\gamma^{-1}\bigr)$.

    % Stable decomposition of the strip span, without a coarse space.
    \emph{Stable decomposition.} Each local space $V_{\ell,j}$ is the span of the nodal basis functions it
    contains, so $V_c^\Sigma$ is the span of their union over $\mathbb J_c$: every nonzero nodal coefficient of $v\in
        V_c^\Sigma$ belongs to a degree of freedom updated by at least one strip patch. Fixing one such patch per degree of
    freedom and splitting nodally as in Step~2 of the proof of Theorem~\ref{thm:additive}, with $w:=v$ and no coarse
    component, gives $v=\sum_{j\in\mathbb J_c}v_j$ with $v_j\in V_{\ell,j}$ and, by~\eqref{eq:nodal-stable} and Step~3 of
    the same proof,
    \begin{gather*}
        \begin{aligned}
            \sum_{j\in\mathbb J_c} \energy{v_j}^2_\ell
             & \le C(p,\gamma,\gamma_D)\,
            \bigl( |v|^2_{1,\Omega_\ell}
            + h^{-2}\|v\|^2_{0,\Omega_\ell}\bigr)
            \\
             & \le C(p,\gamma,\gamma_D)\,
            \bigl(1+C_F(p,\gamma)\bigr)\,\energy{v}^2_\ell
            =: C_{\mathrm{str}}(p,\gamma,\gamma_D)\,\energy{v}^2_\ell ,
        \end{aligned}
    \end{gather*}
    the middle step by the discrete extension~\eqref{eq:ghost-extension}, the full-band
    inequality~\eqref{eq:band-friedrichs}, and the norm equivalence~\eqref{eq:coercive}. This is a stable decomposition
    of $V_c^\Sigma$ into the strip patches, with constant uniform in $\ell$ and in the cut.

    % Multiplicative Schwarz on the strip span.
    The interaction bound of Lemma~\ref{lem:interaction} restricts to $\mathbb J_c$, so the standard multiplicative Schwarz
    estimate \cite[Ch.~2]{ToselliWidlund05}, \cite{Xu92}, applied on the space $V_c^\Sigma$ with exact subspace solves,
    yields, for every fixed ordering,
    \begin{gather*}
        \sigma_\ell^2
        \;\le\;
        1-\frac{1}{\bigl(2N_O(d)^2+1\bigr)
            C_{\mathrm{str}}(p,\gamma,\gamma_D)}
        \;=:\;\sigma(p,\gamma,\gamma_D)^2\;<\;1.
    \end{gather*}
\end{proof}

\section{The two-cell threshold in the vanishing-cut limit}
\label{app:spd-floor}

% Fix the two-cell model and normalize the mesh size out.
We prove the limit statement behind~\eqref{eq:spd-floor-exact}. The configuration is that of
Proposition~\ref{prop:scales}: a cell $\cell$ of edge $h$ cut by a plane parallel to a face, with physical part a slab
of relative thickness $\kappa\in(0,1)$ measured from the ghost face $F$, an uncut active neighbour $\cell'$ across $F$,
and $\Gamma\cap\cell$ parallel to $F$. The two-cell space is the restriction of the conforming space $V_\ell$ to
$\cell\cup\cell'$, that is, the subspace of $\Q_p(\cell)\oplus\Q_p(\cell')$ of functions continuous across $F$. On it
the two-cell form $A_\ell(\gamma)$ consists of the physical energy on $\cell'\cup(\cell\cap\Omega)$, the symmetric
Nitsche terms on $\Gamma\cap\cell$ with parameter $\gamma_D$, and the ghost penalty~\eqref{eq:intro-ghost} on the
single face $F$. Every term scales as $h^{d-2}$ under the dilation $x\mapsto x/h$, so we set $h=1$.

% Reduce to one dimension along the normal coordinate.
\emph{One-dimensional reduction.} Let $x$ be the coordinate normal to $F$, vanishing on $F$ and positive in $\cell$,
so that $\cell'$ corresponds to $x\in(-1,0)$, the physical slab to $x\in(0,\kappa)$, and $\Gamma\cap\cell$ to
$x=\kappa$. By the tensor-product structure~\eqref{eq:Qp}, the restriction of a two-cell function $v$ to a
normal line with frozen tangential variables is a pair $(u,q)\in\mathbb P_p\times\mathbb P_p$ of univariate
polynomials on $(-1,0)$ and $(0,1)$, with $u(0)=q(0)$ by continuity across $F$. The normal-derivative jumps on
$F$, the traces on $\Gamma\cap\cell$, and the normal part of the physical energy act on each line separately,
while the tangential derivatives contribute nonnegative energy only. Hence $A_\ell(\gamma)(v,v)$ is bounded below
by the tangential integral of the one-dimensional form
\begin{gather}\label{eq:two-cell-1d}
    a_{\kappa,\gamma}(u,q) := \int_{-1}^{0}(u')^2
    + \int_{0}^{\kappa}(q')^2
    - 2\,q'(\kappa)\,q(\kappa) + \gamma_D\, q(\kappa)^2
    + \gamma \sum_{k=1}^{p} \frac{\bigl(q^{(k)}(0)-u^{(k)}(0)\bigr)^2}{(k!)^2},
\end{gather}
with equality for tangentially constant $v$. Since $a_{\kappa,\gamma}$ contains $u$ only through $u'$, it is
invariant under adding a constant to $u$, so every pair in $\mathbb P_p\times\mathbb P_p$ can be shifted to
satisfy $u(0)=q(0)$ without changing the value of the form. Semidefiniteness of $A_\ell(\gamma)$ is therefore
equivalent to semidefiniteness of $a_{\kappa,\gamma}$ on all of $\mathbb P_p\times\mathbb P_p$, and it suffices
to study $a_{\kappa,\gamma}$. Its vanishing-cut limit, obtained at fixed $(u,q)$ as $\kappa\to0$, is
\begin{gather}\label{eq:two-cell-limit}
    a_{0,\gamma}(u,q) := \int_{-1}^{0}(u')^2
    - 2\,q'(0)\,q(0) + \gamma_D\, q(0)^2
    + \gamma \sum_{k=1}^{p} \frac{\bigl(q^{(k)}(0)-u^{(k)}(0)\bigr)^2}{(k!)^2}.
\end{gather}

% Define the coercivity threshold in the ghost weight.
The ghost term is positive semidefinite, so $a_{\kappa,\gamma}\succeq0$ implies $a_{\kappa,\gamma'}\succeq0$ for every
$\gamma'\ge\gamma$. Accordingly we define the threshold
\begin{gather*}
    \gamma_*(\kappa) := \inf\{\gamma>0:\ a_{\kappa,\gamma}\succeq0\} \in [0,\infty],
    \qquad \inf\varnothing := \infty .
\end{gather*}
By monotonicity the set on the right is an interval; it is closed relative to $(0,\infty)$, since
$a_{\kappa,\gamma}$ depends continuously on $\gamma$ and semidefiniteness passes to limits on the finite-dimensional
space $\mathbb P_p\times\mathbb P_p$. Hence the infimum is attained whenever it is finite and positive.

\begin{proposition}[vanishing-cut threshold of the two-cell model]\label{prop:spd-limit}
    Let $p\ge1$ and $\gamma_D>0$.
    \begin{itemize}
        \item[(i)] If $\gamma_D>p^2$, then $a_{0,\gamma}\succeq0$ on $\mathbb P_p\times\mathbb P_p$ if and only if
              \begin{gather*}
                  \gamma \;\ge\; \frac{1}{\gamma_D-p^2}\,.
              \end{gather*}
              If $\gamma_D\le p^2$, then $a_{0,\gamma}$ is indefinite for every $\gamma>0$.
        \item[(ii)] For every $\gamma>0$ with $\gamma\,(\gamma_D-p^2)<1$ there is $\kappa_0>0$ such that
              $a_{\kappa,\gamma}$ is indefinite for all $\kappa\in(0,\kappa_0)$. In particular
              \begin{gather*}
                  \sup_{\kappa\in(0,1)}\,\gamma_*(\kappa)
                  \;\ge\; \frac{1}{\gamma_D-p^2}
                  \quad\text{if }\gamma_D>p^2,
                  \qquad
                  \sup_{\kappa\in(0,1)}\,\gamma_*(\kappa) \;=\; \infty
                  \quad\text{if }\gamma_D\le p^2.
              \end{gather*}
    \end{itemize}
\end{proposition}

\begin{proof}
    % Endpoint Christoffel bound for the neighbouring-cell contribution.
    \emph{Step 1 (endpoint bound).} For $s\in\mathbb P_{p-1}$ on $(-1,0)$,
    \begin{gather}\label{eq:christoffel}
        s(0)^2 \;\le\; p^2 \int_{-1}^{0} s^2 ,
    \end{gather}
    with equality for a unique direction. Indeed, with the orthonormal basis
    $\varphi_j(x)=\sqrt{2j+1}\,L_j(2x+1)$, $j=0,\dots,p-1$, of shifted Legendre polynomials of $L^2(-1,0)$, the
    reproducing kernel $K$ of $\mathbb P_{p-1}$ satisfies
    $K(0,0)=\sum_{j=0}^{p-1}\varphi_j(0)^2=\sum_{j=0}^{p-1}(2j+1)=p^2$, since $L_j(1)=1$;
    inequality~\eqref{eq:christoffel} is the reproducing property $s(0)=(s,K(\cdot,0))$ with Cauchy--Schwarz, and
    equality holds exactly for $s$ proportional to $K(\cdot,0)$.

    % Eliminate all variables of the limit form except three scalars.
    \emph{Step 2 (sufficiency in (i)).} Let $\gamma(\gamma_D-p^2)\ge1$, so in particular $\gamma_D>p^2$. Fix $(u,q)$
    and set $a:=q(0)$, $b:=q'(0)$, $c:=u'(0)$. Dropping the nonnegative ghost terms of orders $k\ge2$
    from~\eqref{eq:two-cell-limit} and applying~\eqref{eq:christoffel} to $s=u'\in\mathbb P_{p-1}$,
    \begin{gather*}
        a_{0,\gamma}(u,q) \;\ge\; \frac{c^2}{p^2} + \gamma\,(b-c)^2 - 2ab + \gamma_D\, a^2 .
    \end{gather*}
    Minimizing the right-hand side over $c$ at fixed $b$ gives $c=\gamma p^2 b/(1+\gamma p^2)$ and the value
    \begin{gather}\label{eq:two-by-two}
        a_{0,\gamma}(u,q) \;\ge\; \gamma_{\mathrm{eff}}\, b^2 - 2ab + \gamma_D\, a^2,
        \qquad
        \gamma_{\mathrm{eff}} := \frac{\gamma}{1+\gamma p^2}\,.
    \end{gather}
    The quadratic form on the right of~\eqref{eq:two-by-two} is positive semidefinite in $(a,b)$ if and only if
    $\gamma_{\mathrm{eff}}\,\gamma_D\ge1$, which is equivalent to $\gamma(\gamma_D-p^2)\ge1$. Hence $a_{0,\gamma}\succeq0$.

    % A witness turning every inequality of Step 2 into an equality.
    \emph{Step 3 (necessity in (i)).} Let $\gamma(\gamma_D-p^2)<1$, so that $\gamma_{\mathrm{eff}}\gamma_D<1$ and the
    form in~\eqref{eq:two-by-two} is indefinite; pick $(a,b)$ with $\gamma_{\mathrm{eff}}b^2-2ab+\gamma_D a^2<0$. Set
    $c:=\gamma p^2 b/(1+\gamma p^2)$, let $u\in\mathbb P_p$ be an antiderivative of $s:=c\,K(\cdot,0)/p^2$, the
    equality direction of~\eqref{eq:christoffel} with $s(0)=c$, and let $q\in\mathbb P_p$ be the polynomial with
    $q(0)=a$, $q'(0)=b$, and $q^{(k)}(0)=u^{(k)}(0)$ for $2\le k\le p$. Then the ghost terms of orders $k\ge2$ vanish
    and every inequality of Step 2 holds with equality, so
    $a_{0,\gamma}(u,q)=\gamma_{\mathrm{eff}}b^2-2ab+\gamma_D a^2<0$. When $\gamma_D\le p^2$, the condition
    $\gamma(\gamma_D-p^2)<1$ holds for every $\gamma>0$, and the same witness proves the final claim of (i).

    % Transfer the negativity to small positive cut fractions.
    \emph{Step 4 (proof of (ii)).} Let $\gamma(\gamma_D-p^2)<1$ and fix the witness $(u,q)$ of Step 3, for which
    $a_{0,\gamma}(u,q)<0$. Then
    \begin{gather*}
        a_{\kappa,\gamma}(u,q) - a_{0,\gamma}(u,q)
        = \int_0^\kappa (q')^2
        - 2\bigl(q'(\kappa)q(\kappa)-q'(0)q(0)\bigr)
        + \gamma_D\bigl(q(\kappa)^2-q(0)^2\bigr)
        \;\longrightarrow\; 0
        \qquad(\kappa\to0),
    \end{gather*}
    since $q$ is a fixed polynomial. Hence $a_{\kappa,\gamma}(u,q)<0$, and therefore $\gamma_*(\kappa)>\gamma$
    because the set $\{\gamma':a_{\kappa,\gamma'}\succeq0\}$ is a closed interval not containing $\gamma$, for all
    sufficiently small $\kappa$. If $\gamma_D>p^2$, taking the supremum over $\kappa$ and then over
    $\gamma<1/(\gamma_D-p^2)$ proves the first bound of (ii); if $\gamma_D\le p^2$, every $\gamma>0$ satisfies
    $\gamma(\gamma_D-p^2)<1$, so $\sup_{\kappa\in(0,1)}\gamma_*(\kappa)=\infty$.
\end{proof}

% Interpret the mechanism and state what is left unproved.
Elimination in~\eqref{eq:two-by-two} shows that the uncut neighbour controls part of the first-derivative jump with the
sharp endpoint constant $p^{-2}$ from~\eqref{eq:christoffel}. Consequently, the effective weight in the Nitsche block
is $\gamma/(1+\gamma p^2)$ rather than $\gamma$, and the reduction $\gamma_D\mapsto\gamma_D-p^2$ equals the endpoint
Christoffel constant of $\mathbb P_{p-1}$. The proposition identifies the threshold of the vanishing-cut limit and
shows that it is a lower bound for the cut-uniform threshold $\sup_{\kappa\in(0,1)}\gamma_*(\kappa)$ of the model.
Whether the cut-uniform threshold equals the limit, that is, whether no cut fraction demands a weight beyond
$1/(\gamma_D-p^2)$, remains unproved: the exact generalized-eigenvalue computations reported with
Table~\ref{tab:spd-floor} support it at the sampled cut fractions, and we leave an analytic proof of the matching upper
bound open.